\documentclass[reqno]{amsart}

\usepackage{color,float,bm, epsfig,amsmath,amsfonts,amssymb,amscd,amsthm,amsbsy,bbm, epsf,calc,  pdfsync, mathtools, latexsym, enumerate, framed}
\usepackage[subnum]{cases}
\usepackage{hyperref,appendix,cite}
\numberwithin{equation}{section}

\newtheorem{thm}{Theorem}[section]
\newtheorem{lem}[thm]{Lemma}
\newtheorem{cor}[thm]{Corollary}
\newtheorem{prop}[thm]{Proposition}

\theoremstyle{definition}
\newtheorem{defin}[thm]{Definition}
\newtheorem{rmk}[thm]{Remark}

\DeclareMathOperator*{\diam}{diam}

\DeclareMathOperator{\spt}{spt}

\DeclareMathOperator{\Dom}{Dom}
\DeclareMathOperator{\interior}{int}
\def\bdry{\partial}

\DeclareMathOperator*{\ch}{conv}
\DeclareMathOperator{\cl}{cl}

\DeclareMathOperator{\sgn}{sgn}

\renewcommand{\epsilon}{\varepsilon}
\newcommand{\N}{\mathbf{N}}
\newcommand{\Z}{\mathbf{Z}}

\newcommand{\R}{\mathbf{R}}

\renewcommand{\S}{\mathbf{S}}

\newcommand{\dist}{d}
\DeclareMathOperator*{\affdim}{dim}
\DeclareMathOperator*{\Lip}{Lip}
\newcommand{\Lipnorm}[2][]{\lVert {#2}\rVert_{\Lip{#1}}}

\newcommand{\Leb}[1]{\lvert {#1}\rvert_{\mathcal{L}}}
\newcommand{\norm}[1]{\lvert {#1}\rvert}
\newcommand{\subdiff}[3][]{\partial_{#1}{#2}\left(#3\right)} %optional variable->#1, function ->#2, point or set ->#3

\newcommand{\p}{{\partial}}

\def\xbar{\bar{x}}
\def\ybar{\bar{y}}
\def\tbar{\bar{t}}
\def\gbar{\bar {g}}
\def\M{M}
\def\Mbar{\bar{M}}
\def\Cbar{\bar{C}}

\newcommand{\sing}[1][]{\Sigma_{#1}}

\newcommand{\euclidean}[2]{\langle{#1},{#2}\rangle}%euclidean inner product

\newcommand{\haus}[1][]{\mathcal{H}^{#1}}

\def\anglemin{\Theta} %constant in the lipschitz bound for explicit function theorem

\def\outerdom{\Omega} %source domain where cost function actually defined
\def\outertarget{\overline\Omega} %target domain where cost function actually defined

\def\source{\spt{\mu}} %support of source measure
\def\target{\spt{\nu}} %support of target measure

\def\sourcemeas{\mu} %source measure
\def\targetmeas{\nu} %target measure

\def\sourceint{\left(\spt{\mu}\right)^{\interior}} %interior of support of source measure
\newcommand{\csubdiff}[2]{\partial_c{#1}(#2)} %function ->#1, point or set ->#2
\newcommand{\cstarsubdiff}[2]{\partial_{c^*}{#1}(#2)} %function ->#1, point or set ->#2

\newcommand{\coord}[2]{\left[#1\right]_{#2}} %cotangent coordinate representation of set, set ->#1, base point ->#2

\newcommand{\cExp}[2]{exp^c_{#1}({#2})} %base point ->#1, point in cotangent space ->#2
\newcommand{\cstarExp}[2]{exp^{c^*}_{#1}({#2})} %base point ->#1, point in cotangent space ->#2

\newcommand{\Winfty}[2]{\mathcal{W}_\infty{\left(#1, #2\right)}} %measures ->#1, #2

\newcommand{\nbhd}[1][]{\mathcal{N}_{#1}}%radius->#1
\newcommand{\nbhdof}[2][]{\nbhd[#1]{\left(#2\right)}}%radius->#1, set->#2

\newcommand{\cotanspM}[1]{T^{\ast}_{#1}\M}
\newcommand{\tanspM}[1]{T_{#1}\M}
\newcommand{\cotanspMbar}[1]{T^{\ast}_{#1}\Mbar}
\newcommand{\tanspMbar}[1]{T_{#1}\Mbar}

\newcommand{\xprime}[1][]{x^{\prime}_{#1}}

\newcommand{\xbarprime}[1][]{\bar{x}^{\prime}_{#1}}
\newcommand{\ybarprime}[1][]{\bar{y}^{\prime}_{#1}}

\newcommand{\targetpiece}[1][]{\overline{\Omega}_{#1}}

\newcommand{\threshold}[1][]{\delta_{#1}}%size of W_\infty perturbation
\newcommand{\midheight}[1][]{a_{#1}}%location of separating plane in explicit function theorem
\newcommand{\height}[1][]{d_{#1}}%distance of separation in explicit function theorem

\newcommand{\convpot}[1][]{u_{#1}}
\newcommand{\convpotdual}[1][]{u^{\ast}_{#1}}

\newcommand{\gradconvpot}[1][]{\nabla {u_{#1}}}

\newcommand{\heightfunc}[1][]{h_{#1}}
\newcommand{\heightfuncplus}[1][]{h_{#1}^{+}}
\newcommand{\heightfuncminus}[1][]{h_{#1}^{-}}
\newcommand{\heightfuncpm}[1][]{h_{#1}^{\pm}}
\newcommand{\alttargetpiece}[1][]{\overline{\Lambda}_{#1}}
\newcommand{\altsourcedom}[1][]{\Lambda_{#1}}
\newcommand{\convpotproj}[1][]{u_{#1}}
\newcommand{\convpotprojdual}[1][]{u^{\ast}_{#1}}
\newcommand{\vertaltsourcedom}[1][]{\Lambda^{#1}}%fiber in domain above a projected point #1

\newcommand{\perturbsing}[2][]{\Sigma_{#1}^{#2}}
\newcommand{\perturbmeas}[1][]{\nu^{#1}}

\newcommand{\perturbconvpot}[2][]{u_{#1}^{#2}}

\def\util{\tilde{u}}
\def\utiltil{\tilde{\tilde{u}}}
\begin{document}

\author{Jun Kitagawa}
\address{Department of Mathematics, Michigan State University, 619 Red Cedar Road,
East Lansing, MI 48824}
\email{kitagawa@math.msu.edu}

\author{Robert McCann}
\address{Department of Mathematics, University of Toronto, Toronto, Ontario, Canada, M5S 2E4}
\email{mccann@math.toronto.edu}
\title{Free discontinuities in optimal transport}
\subjclass[2010]{35J96}
\thanks{$^*$JK's research was supported in part by a Simons Foundation Travel Grant and National Science Foundation grant DMS-1700094. RM's research 
was supported in part by NSERC grants 217006-08 and -15 and by a 
Simons Foundation Fellowship.
Parts of this project were carried out while both authors were in residence at 
the Mathematical Sciences Reseach Institute in Berkeley CA  during the Fall 2013
program supported by National Science Foundation Grant No. 0932078 000,
and later at the Fields Insititute for the Mathematical Sciences in Toronto during Fall 2014.
\copyright \today}
\begin{abstract}
We prove a nonsmooth implicit function theorem applicable to the zero set of the difference of convex functions. This theorem is explicit and global: it gives a formula representing this zero set as a difference of convex functions which holds throughout the entire domain of the original functions. As applications, we prove results on the stability of singularities of envelopes of semi-convex functions, and solutions to optimal transport problems under appropriate perturbations, along with global structure theorems on certain discontinuities arising in optimal transport maps
for Ma-Trudinger-Wang costs. For targets whose components satisfy additional convexity, separation, multiplicity and affine independence assumptions we show these discontinuities occur on submanifolds of the appropriate codimension which are parameterized locally
as differences of convex functions (DC, hence $C^2$ rectifiable), and --- depending on the 
precise assumptions --- $C^{1,\alpha}$ smooth.  In this case the highest codimension submanifolds
 consists of isolated points, each uniquely identified by the (affinely independent) components of the target to which
it is transported.
\end{abstract}
\maketitle

\tableofcontents

\section{Introduction}

The question of regularity for maps solving the optimal transportation problem of Monge and Kantorovich is a 
celebrated problem \cite{Trudinger06} \cite{Villani03}.  
Under strong hypotheses relating the target's convexity to curvature
properties of the transportation cost,  optimal maps are known to be smooth,
following work of Caffarelli on quadratic costs \cite{Caffarelli92} and Ma, Trudinger, and 
Wang more generally \cite{MaTrudingerWang05}.
In the absence of such convexity and curvature properties,  much less is true.  Partial regularity results --- which
quantify the size of the singular set --- are available in at least three flavors. 
The set of discontinuities of an optimal map is known to be contained in the non-differentiabilities
of a (semi-)convex function,  hence to have Hausdorff dimension at most $n-1$ in $\R^n$. 
In fact, Zaj\'i\v cek \cite{Zajicek79} has shown
such discontinuities lie in a countable union of submanifolds parameterized as graphs of differences of convex functions 
 --- referred to as DC submanifolds hereafter. 
The {\em closure} of this set of discontinuities was shown to have zero volume by Figalli with Kim (for quadratic
costs \cite{FigalliKim10}) or with DePhilippis (for non-degenerate costs \cite{DePhilippisFigalli15}),  and is conjectured to have dimension at most $n-1$. 
However,  this conjecture has only been verified in the special case of a quadratic transportation cost on 
$\R^2$ \cite{Figalli10}. See related work of Chodosh et al \cite{ChodoshJainLindseyPanchevRubinstein15}
and Goldman and Otto \cite{GoldmanOtto17p}.  
The present manuscript is largely devoted to providing evidence for this conjecture in higher 
dimensions by providing concrete geometries in which it can be confirmed.
Typically these consist of transportation to a collection of disjoint
target components,  which we allow to be convex or non-convex.  This forces discontinuities along which the optimal
map tears the source measure into separate components,  one corresponding to each component of the target.
We study the regularity of such tears.  We show that when the target components can be separated by a hyperplane,
the corresponding tear is a DC hypersurface.  
For quadratic costs, when several tears meet,  their intersection is a DC submanifold of the appropriate
codimension provided the corresponding target components are affinely independent.  
When the corresponding target components are strictly convex,  we show the tears are $C^{1,\alpha}$ smooth,
and that the optimal maps are smooth on their complement.
We show stability of such tears when the data are subject to perturbations which are small in a sense made precise
below.
%, as measured in the Kantorovich-Rubinstein-Wasserstein $L^\infty$ metric.

A core result of this paper is a nonsmooth version of the classical implicit function theorem for convex functions.
More specifically, we wish to write the set where two convex functions coincide as the graph of a DC function, where DC stands for difference of convex,
alternately denoted $c-c$ \cite{Gigli11} or $\Delta$-convex \cite{VeselyZajicek89} in some references. 
The idea of inverse and implicit 
function theorems have been explored in various nonsmooth settings, e.g. by Clarke \cite{Clarke76} and
Vesely and Zaj\'i\v cek \cite[Proposition 5.9]{VeselyZajicek89}; see also \cite{Warga78} 
\cite[Appendix]{McCann95} \cite[Theorem 10.50]{Villani09}. 
Two major aspects set apart the version we present here from previous theorems. The first is the explicit nature of the theorem: we are able to explicitly write down the function whose graph gives the coincidence set in terms of 
partial Legendre transforms of the original convex functions, thus we term this an ``explicit function theorem'' in contrast to the traditional implicit version. %``implicit function theorem''. 
Second, our result is of a global, rather than a local nature: existing implicit function theorems generally state the existence of a neighborhood on which a surface can be written as the graph of a function, in our theorem we obtain that the domain of this function is actually the projection of the entire original domain on some hyperplane. Our method of proof relies on the construction of Alberti from \cite[Lemma 2.7]{Alberti94},  
foreshadowed in % has its roots in the work of
 Zaj\'i\v cek's work \cite{Zajicek79}.

Our interest in this theorem is motivated by its application
to the \emph{optimal transport} problem of Monge and Kantorovich mentioned above. 
Let $\outerdom$ and $\outertarget$ be compact subsets of $n$-dimensional Riemannian manifolds $(M, g)$ and $(\Mbar, \gbar)$ respectively, and a real valued \emph{cost function} $c\in C^4(\outerdom\times \outertarget)$. The optimal transport problem is: given any two probability measures $\sourcemeas$ and $\targetmeas$ on $\outerdom$ and $\outertarget$ respectively, find a measurable mapping $T: \source \to \target$ %defined $\sourcemeas$-a.e. which
pushing $\mu$ forward to $\nu$ (denoted $ T_\#\sourcemeas=\targetmeas$), such that
\begin{align}
 \int_\outerdom c(x, T(x))\sourcemeas(dx)&=\inf_{S_\#\sourcemeas=\targetmeas} \int_\outerdom c(x, S(x))\sourcemeas(dx).\label{OT}\tag{OT}
\end{align}
The applications we present here concern the global structure of discontinuities in $T$, 
stability results for such tears, and the regularity of $T$ on their complement.
For the first application, we ask if there is some structure for these discontinuities when the support of the target measure is separated into two compact sets --- by a hyperplane (in appropriate coordinates). One would expect the source domain to be partitioned into two sets, which are then transported to each of the pieces in the target. Under suitable hypotheses we show this is the case, and the interface between these two pieces is actually a DC hypersurface (thus $C^2$ rectifiable) which can be parameterized as a globally Lipschitz graph.
% Such questions involving the structure of singular sets in optimal transport is of great interest and has been the subject %of much research in recent years (see \cite{Figalli10, FigalliKim10, DePhilippisFigalli15}
%\cite{ChodoshJainLindseyPanchevRubinstein15}). 
In the second application, we consider a target measure consisting of several connected components. This should result in a transport map that must split mass amongst the pieces, and we investigate the structure and stability of this splitting. %under certain perturbations of the target measure. 
It turns out a stability result can be obtained when considering perturbations of the target measure under the Kantorovich-Rubinstein-Wasserstein $L^\infty$ metric ($\mathcal{W}_\infty$ in Definition \ref{def: W-infinity} below), along with an appropriate notion of affine independence for the pieces
(Definition \ref{defin: affine independence of sets} below). 
We also provide an example to illustrate this independence condition plays the role of an implicit function hypothesis and is crucial for stability.

%\marginpar{rework}

The outline of the paper is as follows. In Section~\ref{section: explicit function theorem} we set up and prove the ``explicit function theorem'' for convex differences.
We then apply the explicit function theorem in Section~\ref{section: stability},
to show stability for singular points of envelopes of semi-convex functions under certain perturbations. In Section~\ref{section: stability of OT}, we recall some necessary background material concerning 
the optimal transport problem and begin to explore consequences of known regularity results in our setting.  
For the quadratic cost $c(x,\bar x) = -\langle x, \bar x \rangle$ on Euclidean space,
Section \ref{section: global quadratic} proves DC rectifiability of the (codimension $k$) tears 
along which the source is split into $k+1$ components whose images have affinely independent convex hulls.  
For $k=n$, Proposition \ref{prop: only one n+1 order singularity} shows the corresponding tear consists
of a single point.
%isolated points, each uniquely identified by the (affinely independent) components of the target to which it is %transported.
Section \ref{section: smoother quadratic}
shows these tears are $C^{1,\alpha}$ provided the corresponding target components are strictly convex;  
in the simplest case $k=1$, 
a similar result was found by Chen \cite{Chen16p} simultaneously and independently of the present
manuscript: the main thrust of his work is to improve regularity of the tear to $C^{2,\alpha}$ when the pair of strictly convex target components are sufficiently far apart.
Smoothness of the map away from such tears is shown for Ma-Trudinger-Wang costs --- known
as MTW costs \cite{MaTrudingerWang05} \cite{TrudingerWang09} --- in Corollary \ref{cor: interior homeo/diffeo}.
Section \ref{section: global MTW} extends our DC rectifiability result for tears to MTW costs % \cite{MaTrudingerWang05}
in the prototypical case $k=1$.  Section \ref{section: stability of tears} shows such tears are stable.
Lastly, we include an appendix presenting an example to show the affine independence of target measures components is necessary for stability.

Throughout this paper, for $1\leq i\leq n$ we will use the notation $\pi_{i}: \R^n\to\R^{i}$ to denote orthogonal projection onto the first $i$th coordinates, and $e_{i}$ for the $i$th unit coordinate vector. We also reserve the notation $A^{\cl}$, $A^{\interior}$, and $A^{\bdry}$ for the closure, interior, and boundary of a set $A$ respectively. Also, given any point $x\in\R^n$, we will write $x^i$ for the $i$th coordinate of $x$. $\haus[i]$ will refer to the $i$-dimensional Hausdorff measure of a set in Euclidean space and $\haus[i]_g$ will be the $i$-dimensional Hausdorff measure of a set defined using the distance derived from a Riemannian metric $g$. Finally, $\ch(A)$ denotes the \emph{closed} convex hull of a set $A$ while $\nbhdof[\epsilon]{A}=\{x\mid dist(x, A)\leq \epsilon\}$.

\section{An ``explicit function theorem'' for convex differences}\label{section: explicit function theorem}
For the remainder of the paper, by \emph{convex function} with no other qualifiers we will tacitly mean a \emph{closed, proper, convex function on $\R^n$} i.e., a function defined on $\R^n$ taking values in $\R\cup\{\infty\}$, whose epigraph is a non-empty, closed, convex set. If we refer to \emph{a convex function on $\Lambda$} for some set $\Lambda\subset \R^n$, this will mean a function satisfying the above definition
when it is extended lower semicontinuously to $\Lambda^{\cl}$ and
(re)defined to be $\infty$ on $(\R^n\setminus \Lambda)^{\interior}$.
%, then it becomes a convex function in the above sense. 
Also, we will use the notations $\xprime:=\pi_{n-1}{(x)}$ and $A^\prime:=\pi_{n-1}{\left(A\right)}$ for any point $x\in \R^n$ and set $A\subset \R^n$. By the classical implicit function theorem, if $f$, $g: \R^n\to \R$ are smooth, the set $\{f=g\}$ is the graph of a smooth function of $n-1$ variables, near any point on the set where $\nabla f\neq \nabla g$. We aim to prove an analogue of this theorem, but for two convex functions without any assumptions of differentiability. In order to do so, we need an appropriate replacement for the inequality of gradients, which will be formulated in terms of the \emph{subdifferential}: recall for a convex function $u$ and $x_0$ in its domain,
\begin{align}\label{eqn: convex subdifferential}
 \subdiff{u}{x_0}:=\left\{\xbar\in\R^n\mid \euclidean{x-x_0}{\xbar}+u(x_0)\leq u(x),\ \forall x\right\},
\end{align}
while for a subset $A$ of its domain,
\begin{align*}
 \subdiff{u}{A}:=\bigcup_{x\in A}\subdiff{u}{x}.
\end{align*}
We also recall here the \emph{Legendre transform} of a (proper) convex function $u$ with effective domain $\Dom(u):=\{x\in\R^n\mid u(x)<\infty\}$ as the (closed, proper, convex) function $u^*: \R^n\to \R\cup \{\infty\}$ defined by
\begin{align}\label{eqn: legendre transform}
 u^*(\xbar):=\sup_{x\in \R^n}[\euclidean{x}{\xbar}-u(x)]=\sup_{x\in \Dom(u)}[\euclidean{x}{\xbar}-u(x)].
\end{align}
\begin{defin}[Separating hyperplane]\label{def: strongly separating hyperplane}
 If $\altsourcedom[+]$ and $\altsourcedom[-]$ are any two sets in $\R^n$ and $v$ is a fixed unit vector, recall that a hyperplane $\left\{x\in\R^n\mid\euclidean{x}{v}=\midheight\right\}$ is said to \emph{strongly separate $\altsourcedom[+]$ and $\altsourcedom[-]$ (with spacing $\height$)} if there exists a $\height>0$ such that
\begin{align*}
 \euclidean{x_1}{v}< \midheight-\height<\midheight+\height< \euclidean{x_2}{v}
\end{align*}
for any $x_1\in\altsourcedom[+]$ and $x_2\in\altsourcedom[-]$.
\end{defin}
Let us also recall some terminology on DC (difference of convex) functions here.
\begin{defin}[DC functions, mappings \cite{AmbrosioBertrand15, PerelmanUnpublished}]\label{defin: DC}
 A function $h: \Lambda\to\R$ on a convex domain $\Lambda\subset \R^n$ is said to be a \emph{DC function} if it can be written as the difference of two convex functions that are finite on $\Lambda$. A mapping from $\Lambda$ to a Euclidean space $\R^m$ is said to be a \emph{DC mapping} if each of its coordinate components is a DC function.
\end{defin}

The key hypothesis of our theorem is the strong separation of the subdifferentials of two convex functions. One feature that differentiates our theorem from the usual implicit function theorem is that we can actually write down the function whose graph gives the equality set between the two convex functions we consider, and explicitly state the domain of this function. Thus we term this an ``explicit function theorem.'' We first state the following Theorem \ref{thm: explicit function theorem} in terms of the subdifferential of the envelope of two convex functions, and formulate the actual explicit function theorem as Corollary \ref{cor: explicit function theorem} below.

\begin{thm}[DC tears]\label{thm: explicit function theorem}
 Let $\convpot[+]$ and $\convpot[-]$ be convex functions, $\altsourcedom\subset \Dom(u)\subset\R^n$ a convex (but not necessarily bounded) set, and $\alttargetpiece[+]$, $\alttargetpiece[-]$ compact subsets of $\R^n$ with $\subdiff{\convpot[+]}{\altsourcedom}\subset\alttargetpiece[+]$ and $\subdiff{\convpot[-]}{\altsourcedom}\subset\alttargetpiece[-]$. We define
 \begin{align*}
 u:&=\max{\left\{\convpot[+], \convpot[-]\right\}},\\
 \sing[]:&=\left\{x\in\altsourcedom^{\cl}\mid\subdiff{u}{x}\cap\alttargetpiece[+]\neq\emptyset\text{ and }\subdiff{u}{x}\cap\alttargetpiece[-]\neq\emptyset\right\},\\
 C_+:&=\left\{x\in\altsourcedom^{\cl}\mid \subdiff{u}{x}\cap\alttargetpiece[-]=\emptyset\right\},\\
 C_-:&=\left\{x\in\altsourcedom^{\cl}\mid \subdiff{u}{x}\cap\alttargetpiece[+]=\emptyset\right\}.
\end{align*} 
Also, suppose that (after a rotation of coordinates) for some $\midheight[0]\in\R$ the hyperplane $\Pi:=\left\{x^n=\midheight[0]\right\}$ strongly separates $\alttargetpiece[+]$ and $\alttargetpiece[-]$ with spacing $\height[0]>0$.

Writing $\Lambda^{\prime}:=\pi_{n-1}{(\Lambda)}$, define the functions $\heightfuncpm: \R^{n-1}\to \R$, $\heightfunc[]: (\Lambda^{\prime})^{\cl}\to\R$ by 
\begin{align}\label{eqn: explicit formula}
\heightfuncpm(\xprime):&= 
\begin{cases}
-\frac{\convpotprojdual[\xprime](\midheight[0]\mp\height[0])}{2\height[0]},&\xprime\in (\Lambda^{\prime})^{\cl},\\
\infty,&\xprime\in \R^{n-1}\setminus(\Lambda^{\prime})^{\cl}
\end{cases}
\\
 \heightfunc[](\xprime):&=\heightfuncplus[](\xprime)-\heightfuncminus[](\xprime),
\end{align}
where $\convpotprojdual[\xprime]$ is the Legendre transform of the function $\convpotproj[\xprime](t):=\convpot(\xprime, t)$ of one variable. Then $\heightfuncpm$ are both convex on $\R^{n-1}$ and finite on $\Lambda^\prime$ (so in particular, $\heightfunc$ is a DC function), with 
\begin{align*}
 \sing&=\{(\xprime, \heightfunc(\xprime))\mid \xprime\in \Lambda^\prime\}\cap\altsourcedom^{\cl},\\
C_+ &=\{(\xprime, x^n)\mid \xprime\in \Lambda^\prime,\ \heightfunc(\xprime)<x^n\}\cap\altsourcedom^{\cl},\\
 C_-&=\{(\xprime, x^n)\mid \xprime\in \Lambda^\prime,\ \heightfunc(\xprime)>x^n\}\cap\altsourcedom^{\cl}.
\end{align*}

Moreover, 
\begin{align}\label{eqn: lipschitz bound}
\Lipnorm[((\Lambda^{\prime})^{\cl})]{\heightfunc[]}\leq \tan{\anglemin}
%\leq\frac{\sqrt{1-\anglemin^2}}{\anglemin}
\le \frac{\diam[\pi_{n-1}(\overline \Lambda_+ \cup \overline \Lambda_-)]}{2d_0}
%\frac{\diam{\paren{\alttargetpiece[+]\cup\alttargetpiece[-]\cup\curly{0}}}}{\height[0]},
\end{align} 
where
\begin{align*}
 \cos{\anglemin}:=\inf_{\xbar_+\in \alttargetpiece[+], \xbar_-\in \alttargetpiece[-]}\euclidean{\frac{\xbar_+-\xbar_-}{\norm{\xbar_+-\xbar_-}}}{e_{n}}.
\end{align*}
\end{thm}
\begin{rmk}\label{rmk: extendability}
 Both functions $u_\pm$ can be extended in a continuous way to all of $\altsourcedom^{\cl}$. Indeed, since $\subdiff{\convpot[\pm]}{\altsourcedom}$ is bounded, we can exhaust $\altsourcedom$ by compact sets and apply \cite[Theorem 24.7]{Rockafellar70} to find that $\convpot[\pm]$ are uniformly Lipschitz on $\altsourcedom$; in particular they can be extended continuously to $\altsourcedom^{\cl}$ with finite values. Moreover, by compactness of $\alttargetpiece[\pm]$ we see that $\subdiff{u_\pm}{x}\neq\emptyset$ for any $x\in\altsourcedom^\partial$ as well.
\end{rmk}
We will need the following classical result on subdifferentials of envelopes of convex functions (which can be obtained for example, by \cite[Proposition 2.3.12]{Clarke90} applied to convex functions).
\begin{lem}\label{lem: dm}
 If $\convpot=\max_{i}\convpot[i]$ for some finite collection of convex functions $u_i$, then
\begin{align*}
 \subdiff{\convpot}{x_0}=\ch{\left(\bigcup_{i\in I}\subdiff{\convpot[i]}{x_0}\right)}
\end{align*}
where $I:=\left\{i\mid \convpot(x_0)=\convpot[i](x_0)\right\}$.
\end{lem}
Using this result, we find the following reformulation of Theorem \ref{thm: explicit function theorem}.
\begin{cor}[Explicit function theorem]\label{cor: explicit function theorem}
Under the same notation and hypotheses as Theorem \ref{thm: explicit function theorem}, 
\begin{align*}
 \left\{x\in\altsourcedom^{\cl}\mid \convpot[+](x)=\convpot[-](x)\right\}&=\{(\xprime, \heightfunc(\xprime))\mid \xprime\in (\Lambda^\prime)^{\cl}\}\cap\altsourcedom^{\cl},\\
 \left\{x\in\altsourcedom^{\cl}\mid \convpot[+](x)>\convpot[-](x)\right\}&=\{(\xprime, x^n)\mid \xprime\in (\Lambda^\prime)^{\cl},\ \heightfunc(\xprime)<x^n\}\cap\altsourcedom^{\cl},\\
 \left\{x\in\altsourcedom^{\cl}\mid \convpot[+](x)<\convpot[-](x)\right\}&=\{(\xprime, x^n)\mid \xprime\in (\Lambda^\prime)^{\cl},\ \heightfunc(\xprime)>x^n\}\cap\altsourcedom^{\cl}.
\end{align*}
%where all of the functions involved are extended continuously with finite values to $\altsourcedom^{\cl}$ or $(\Lambda^\prime)^{\cl}$.
\end{cor}
\begin{proof}
Lemma~\ref{lem: dm} combined with Remark \ref{rmk: extendability} immediately yields the corollary from Theorem \ref{thm: explicit function theorem}.
% Lemma~\ref{lem: dm} immediately implies the sets $\sing$, $C_+$, and $C_-$ in the statement of Theorem \ref{thm: explicit function theorem} equal the sets above on the left hand side, but with $\altsourcedom^{\cl}$ replaced with $\altsourcedom^{\interior}$, thus we have each of the equalities above with $\empty^{\cl}$ replaced by $\empty^{\interior}$ throughout. As in the remark above, $\heightfunc$ can also be continuously extended to $(\Lambda^\prime)^{\cl}$. Since $(\altsourcedom^{\interior})^{\cl}=\altsourcedom^{\cl}$ by convexity, by the continuity of these extensions the proof is complete.
\end{proof}
\begin{proof}[Proof of Theorem~\ref{thm: explicit function theorem}]
 Fix any such strongly separating hyperplane, %. We begin by rotating coordinates so that this hyperplane coincides with the plane $\curly{x^n=\midheight[0]}$ for some $\midheight[0]\in\R$, such that 
 by our assumptions we have $\alttargetpiece[+]\subset \left\{x^n>\midheight[0]+\height[0]\right\}$ and $\alttargetpiece[-]\subset \left\{x^n<\midheight[0]-\height[0]\right\}$. 
% Finally, let us write $\Lambda^{\prime}$ and $\xprime[]$ for the projections of $\altsourcedom$ and any $x\in\altsourcedom$ onto this hyperplane, and 
% 
 Also, if $\xprime\in\Lambda^\prime$, let us write $\vertaltsourcedom[{\xprime[]}]:=\left\{t\in\R\mid (\xprime[], t)\in\altsourcedom\right\}$. By Remark \ref{rmk: extendability}, we can assume $u_\pm$ are both continuous up to $\Lambda^{\cl}$ which is also convex, thus we will tacitly assume $\Lambda$ is a closed set for the remainder of the proof.

We first claim that given $\xprime\in\Lambda^{\prime}$, there is at most one $x^n\in\vertaltsourcedom[{\xprime[]}]$ such that $(\xprime, x^n)\in\sing[]$, and it must be that $x^n=\heightfunc(\xprime)$. %$\subdiff{\twomaxpot}{\xprime, x^n}$ intersects both $\alttargetpiece[+]$ and $\alttargetpiece[-]$. 
Indeed, fix an $\xprime\in\Lambda^{\prime}$ and suppose there exists such an $x^n$. %, and consider the convex function $\convpotproj[\xprime]$ of one variable, defined by
%\begin{align*}
% \convpotproj[\xprime](t):=\twomaxpot(\xprime, t).
%\end{align*}
First by \cite[Proposition 2.4]{Alberti94}, for any $(\xprime, t)\in\altsourcedom$ we have 
\begin{align}\label{eqn: subdifferential of projection}
\subdiff{\convpotproj[{\xprime}]}{t}=\pi^{n}\left(\subdiff{u}{\xprime, t}\right).
\end{align}
As $\subdiff{u}{\xprime, x^n}$ is convex and intersects both $\alttargetpiece[+]$ and $\alttargetpiece[-]$, we must have $[\midheight[0]-\height[0], \midheight[0]+\height[0]]\subset \subdiff{\convpotproj[\xprime]}{x^n}$, which implies $x^n\in\subdiff{\convpotprojdual[\xprime]}{[\midheight[0]-\height[0], \midheight[0]+\height[0]]}$ by \cite[Theorem 23.5]{Rockafellar70}. We also immediately see that the values $\convpotprojdual[\xprime](\midheight[0]\pm\height[0])$ are both finite. 
By the definition of subdifferential, we have the inequalities
\begin{align*}
 \convpotprojdual[\xprime](\midheight[0]+\height[0])&\geq \convpotprojdual[\xprime](\midheight[0]-\height[0])+x^n(\midheight[0]+\height[0]-(\midheight[0]-\height[0])),\\
 \convpotprojdual[\xprime](\midheight[0]-\height[0])&\geq \convpotprojdual[\xprime](\midheight[0]+\height[0])+x^n(\midheight[0]-\height[0]-(\midheight[0]+\height[0])),
\end{align*}
which combined implies
%This implies that the Legendre transform $\convpotprojdual[\xprime](s)$ is equal to an affine function of slope $x^n$ for $\midheight[0]-\height[0]\leq s\leq \midheight[0]+\height[0]$; in other words we have 
 $x^n=\heightfunc[](\xprime)$ defined by \eqref{eqn: explicit formula}, and in particular there can only be at most one such $x^n$ for each $\xprime$.
 
Now suppose $\xprime\in\Lambda^{\prime}$ is such that $\vertaltsourcedom[\xprime]\neq \emptyset$ but there is no $t\in \vertaltsourcedom[\xprime]$ with $(\xprime, t)\in\altsourcedom$ where $\subdiff{\convpot}{\xprime, t}$ intersects both of the sets $\alttargetpiece[\pm]$. Note since $\altsourcedom$ is convex the fiber $\vertaltsourcedom[\xprime]$ is connected. As the choice of cost function $c(x, \xbar):=-\euclidean{x}{\xbar}$ satisfies conditions \eqref{B1} and \eqref{MTW} (see Section \ref{section: stability of OT} below), we can apply Lemma \ref{lem: connectivity of c-subdifferential} to see that $\subdiff{\convpotproj[{\xprime}]}{\vertaltsourcedom[\xprime]}$ is connected. We comment here, Lemma \ref{lem: connectivity of c-subdifferential} does not directly apply if $\vertaltsourcedom[\xprime]$ is unbounded, but we can exhaust $\vertaltsourcedom[\xprime]$ with an increasing collection of bounded subintervals then take the union of their images under the subdifferential of $u_{\xprime}$ to obtain the claim. In particular by Lemma \ref{lem: dm} (recalling \eqref{eqn: subdifferential of projection}), either $\subdiff{\convpotproj[{\xprime}]}{\vertaltsourcedom[\xprime]}\subset [\midheight[0]+\height[0], \infty)$ or $\subdiff{\convpotproj[{\xprime}]}{\vertaltsourcedom[\xprime]}\subset(-\infty, \midheight[0]-\height[0]]$, suppose it is the former; this is equivalent to having on the set $\vertaltsourcedom[\xprime]$,
\begin{align}
\convpotproj[{\xprime}](\cdot)\equiv u_+(\xprime,\cdot).\label{eqn: plus is bigger}
\end{align} 
%We will now show $\vertaltsourcedom[\xprime]$ must be bounded from below. Suppose this were not the case, fix some $\tilde{t}\in \vertaltsourcedom[\xprime]$ and take $t_k\in\vertaltsourcedom[\xprime]$ with $t_k\searrow -\infty$ and $\tbar_k\in \subdiff{\convpotproj[{\xprime}]}{t_k}$. We have
%\begin{align*}
% u_+(\xprime, t_k)&=\convpotproj[{\xprime}](t_k)\leq  \convpotproj[{\xprime}](\tilde{t})-\tbar_k(\tilde{t}-t_k)\\
% &=u_+(\xprime, \tilde{t})-\tbar_k(\tilde{t}-t_k)\leq u_+(\xprime, \tilde{t})-(\midheight[0]+\height[0])(\tilde{t}-t_k).
%\end{align*}
%At the same time, taking $\tbar_-\in \pi^n(\subdiff{u_-}{\xprime, \tilde{t}})$, recalling \eqref{eqn: subdifferential of projection} we have
%\begin{align*}
% u_-(\xprime, t_k)&\geq u_-(\xprime, \tilde{t})+\tbar_-(t_k-\tilde{t})\geq u_-(\xprime, \tilde{t})+(\midheight[0]-\height[0])(t_k-\tilde{t}),
%\end{align*}
%thus
%\begin{align*}
% u_-(\xprime, t_k)- u_+(\xprime, t_k)&\geq u_-(\xprime, \tilde{t})-u_+(\xprime, \tilde{t})+2\height[0](\tilde{t}-t_k)>0
%\end{align*}
%for large enough $k$, contradicting \eqref{eqn: plus is bigger}. As a result $\vertaltsourcedom[\xprime]$ must be bounded from below.

Now we claim there exists a finite $t_0\in \R$ such that 
\begin{align*}
 u_{\xprime}(t_0)=u_+(\xprime, t_0)=u_-(\xprime, t_0).
\end{align*}
By \eqref{eqn: plus is bigger}, it is sufficient to show there is some $t$ for which $u_-(\xprime, t)>u_+(\xprime, t)$, then then intermediate value theorem will finish the claim. Fix some $\tilde{t}\in \vertaltsourcedom[\xprime]$ and suppose the claim fails, then \eqref{eqn: plus is bigger} would hold on all of $(-\infty, \tilde{t}]$. In turn, this means $u_+(\xprime, t)$ is finite for all $t\leq \tilde{t}$, as if it was infinite anywhere the subdifferential of $u_+(\xprime, \cdot)$ would contain an interval of the form $(-\infty, \tbar)$ for some $\tbar$, contradicting \eqref{eqn: subdifferential of projection} and the assumption $\subdiff{u_+}{\altsourcedom}\subset \alttargetpiece[+]$. Now take a sequence $t_k\searrow -\infty$ where $u_-(\xprime, t_k)\leq u_+(\xprime, t_k)$, with $t_k<\tilde{t}$ for all $k$. By the above remark we can find $\tbar_k\in \pi^n(\subdiff{u_+}{\xprime, t_k})\subset (\midheight[0]+\height[0], \infty)$. Using \cite[Proposition 2.4]{Alberti94} we then have
\begin{align*}
 u_+(\xprime, t_k)&\leq  u_+(\xprime, \tilde{t})-\tbar_k(\tilde{t}-t_k)\leq u_+(\xprime, \tilde{t})-(\midheight[0]+\height[0])(\tilde{t}-t_k).
\end{align*}
At the same time $u_-$ is finite on $\altsourcedom$, hence there exists $\tbar_-\in \pi^n(\subdiff{u_-}{\xprime, \tilde{t}})\subset (-\infty, \midheight[0]-\height[0])$, again by \cite[Proposition 2.4]{Alberti94} we have
\begin{align*}
 u_-(\xprime, t_k)&\geq u_-(\xprime, \tilde{t})+\tbar_-(t_k-\tilde{t})\geq u_-(\xprime, \tilde{t})+(\midheight[0]-\height[0])(t_k-\tilde{t}),
\end{align*}
thus
\begin{align*}
 u_-(\xprime, t_k)- u_+(\xprime, t_k)&\geq u_-(\xprime, \tilde{t})-u_+(\xprime, \tilde{t})+2\height[0](\tilde{t}-t_k)>0
\end{align*}
for large enough $k$, a contradiction, hence the claim is proven.

By Lemma \ref{lem: dm} and \eqref{eqn: subdifferential of projection} we can see that $\midheight[0]+\height[0]\in \subdiff{u_{\xprime}}{t_0}$, hence by \cite[Theorem 23.5]{Rockafellar70} we have
\begin{align}
 \convpotprojdual[\xprime](\midheight[0]+\height[0])&=t_0(\midheight[0]+\height[0])-\convpot(\xprime, t_0).%\notag\\
 %&\leq t_0(\midheight[0]+\height[0])-\convpot(\xprime, t_{\inf})-(t_0-t_{\inf})\tbar_{\inf}.
 \label{eqn: upper bound on transform}
\end{align}
Since by definition
\begin{align*}
 -\convpotprojdual[\xprime](\midheight[0]-\height[0])&=\inf_{t\in\R}(\convpot(\xprime, t)-t(\midheight[0]-\height[0]))\leq \convpot(\xprime, t_0)-t_0(\midheight[0]-\height[0]),
\end{align*}
we find that
\begin{align*}
 \heightfunc(\xprime)&\leq \frac{\convpot(\xprime, t_0)-t_0(\midheight[0]-\height[0])+t_0(\midheight[0]+\height[0])-\convpot(\xprime, t_0)}{2\height[0]}=t_0\leq \inf\vertaltsourcedom[\xprime],
\end{align*}
the last inequality from the fact that \eqref{eqn: plus is bigger} holds on $\vertaltsourcedom[\xprime]$.
The argument leading to \eqref{eqn: upper bound on transform} can also be applied to $\convpotprojdual[\xprime](\midheight[0]-\height[0])$, since $u_{\xprime}^*$ is a proper convex function, an upper bound implies finiteness, hence $\heightfuncpm[]$ are both finite valued for such $\xprime$. The case $\subdiff{\convpotproj[{\xprime}]}{\vertaltsourcedom[\xprime]}\subset(-\infty, \midheight[0]-\height[0]]$ can be handled by a symmetric argument yielding that $\heightfunc(\xprime)\geq \sup \vertaltsourcedom[\xprime]$, and we find $\heightfuncpm[]$ are both finite valued on all of $\Lambda^{\prime}$. To show closedness of $\heightfuncpm[]$, fix any $(\xprime[0], t_0)\in \R^n$. By \cite[Theorem 7.1]{Rockafellar70}, $u$ is lower semicontinuous on $\R^n$, thus for any $\epsilon>0$, there exists $\delta>0$ such that $\displaystyle u(\xprime[0], t_0)\leq \epsilon+\inf_{\xprime\in B_\delta(\xprime[0])\setminus \{\xprime[0]\}} u(\xprime, t_0)$, hence we have
\begin{align*}
 -\convpotprojdual[{\xprime[0]}](\midheight[0]\pm\height[0])&\leq u(\xprime[0], t_0)-t_0(\midheight[0]\pm\height[0])\\
 &\leq \inf_{\xprime\in B_\delta(\xprime[0])\setminus \{\xprime[0]\}} u(\xprime, t_0)-t_0(\midheight[0]\pm\height[0])+\epsilon,
\end{align*}
taking an infimum over $t_0\in \R$ shows that $\heightfuncpm$ is lower semicontinuous, hence closed by \cite[Theorem 7.1]{Rockafellar70} again.

 Next suppose $x\in\altsourcedom$ is such that $\subdiff{u}{x}\cap \alttargetpiece[-]=\emptyset$, and there exists an $(\xprime, t)\in \altsourcedom$ where $\subdiff{\convpot}{\xprime, t}$ intersects both of the sets $\alttargetpiece[\pm]$. By the argument above, we must have $t=\heightfunc(\xprime)$. Take $\xbar\in \subdiff{u}{x}$ and $(\ybarprime, \midheight[0])\in \subdiff{u}{\xprime, \heightfunc(\xprime)}$. By monotonicity of the subdifferential we find that
\begin{align*}
0&\leq \euclidean{x-(\xprime, \heightfunc(\xprime))}{\xbar-(\ybarprime, \midheight[0])}\\
&=(x^n-h(\xprime))(\xbar^n-\midheight[0]).
\end{align*}
However, by Lemma \ref{lem: dm} and since $\subdiff{u}{x}$ does not intersect $\alttargetpiece[-]$, we have must have $\xbar^n-\midheight[0]\geq 0$, thus $x^n\geq h(\xprime)$. A symmetric argument yields that if $\subdiff{u}{x}\cap \alttargetpiece[+]=\emptyset$, then $x^n\leq h(\xprime)$. Since $\subdiff{u}{\xprime, \heightfunc(\xprime)}$ intersects both sets $\alttargetpiece[\pm]$, the above inequalities must be strict. Combined with the arguments above, this proves the characterizations of $\sing$, $C_+$, and $C_-$ as the graph, epigraph, and subgraph of $h$ intersected with $\altsourcedom$.

We will next show $\heightfunc[\pm]$ are both convex (essentially, this is just the fact that a supremum of a family of jointly convex functions gives a concave function). To this end, fix $\xprime[0]$, $\xprime[1]\in\Lambda^{\prime}$ and $t_0$, $t_1\in \R$, and define $(\xprime[\lambda], t_\lambda):=((1-\lambda)\xprime[0]+\lambda\xprime[1], (1-\lambda)t_0+\lambda t_1)$. Then $\xprime[\lambda]\in\Lambda^{\prime}$, hence $\convpotprojdual[{\xprime[\lambda]}](\midheight[0]+\height[0])$ is finite, in particular $\heightfuncpm$ cannot take the value $-\infty$ anywhere and they must be proper. %\in\vertaltsourcedom[{\xprime[0]}]$ and $t_1\in\vertaltsourcedom[{\xprime[1]}]$, by the convexity of $\altsourcedom$ we have $(\xprime[\lambda], t_\lambda):=((1-\lambda)\xprime[0]+\lambda\xprime[1], (1-\lambda)t_0+\lambda t_1)\in\altsourcedom$. Hence 
By the convexity of $u$, we can calculate
\begin{align*}
\convpotprojdual[{\xprime[\lambda]}](\midheight[0]+\height[0])&\geq t_\lambda(\midheight[0]+\height[0])-u(\xprime[\lambda], t_\lambda)\\
&\geq (1-\lambda)t_0(\midheight[0]+\height[0])-(1-\lambda)u(\xprime[0], t_0)+\lambda t_1(\midheight[0]+\height[0])-\lambda u(\xprime[1], t_1),
\end{align*}
where the right hand sides of the second and third lines above may take the value $-\infty$. 
By taking a supremum on the right hand side, first over $t_0$, then over $t_1$, we obtain
\begin{align*}
 \convpotprojdual[{\xprime[\lambda]}](\midheight[0]+\height[0])\geq (1-\lambda)\convpotprojdual[{\xprime[0]}](\midheight[0]+\height[0])+\lambda\convpotprojdual[{\xprime[1]}](\midheight[0]+\height[0]),
\end{align*}
then since $\Lambda'$ is convex, the epigraph of $\heightfuncplus$ will be a convex set. %i.e. $\xprime \mapsto \convpotprojdual[{\xprime}](\midheight[0]+\height[0])$ is concave, hence as a negative multiple $\heightfuncplus$ is convex. 
A similar argument for $\convpotprojdual[{\xprime[\lambda]}](\midheight[0]-\height[0])$ proves the epigraph of $\heightfuncminus$ is convex as well.

Lastly we prove the Lipschitz bound \eqref{eqn: lipschitz bound}. To do so, we will show that any circular cone of slope $\tan{\anglemin}$ %$\frac{\sqrt{1-\anglemin^2}}{\anglemin}$ 
opening in the positive  or negative $e_{n}$ direction, with vertex on the set $\sing[]\cap \altsourcedom$ remains on one side of $\sing[]$. Specifically, fix a point in $\sing[]\cap\altsourcedom$ and after a temporary translation, assume it is the origin. We claim that if $x^n\geq \norm{\xprime}\tan{\anglemin}$ with $\xprime\in \Lambda^{\prime}$, then 
\begin{align}\label{eqn: cone lies on top}
\heightfunc[](\xprime)\leq x^n.
\end{align}
Let us assume $\heightfunc(\xprime)\geq 0$, otherwise the above claim is immediate. First note that 
\begin{align}\label{eqn: monotonicity gives lipschitz}
\exists\;\xbar_\pm\in\alttargetpiece[\pm]\text{ s.t. }\euclidean{(\xprime, \heightfunc(\xprime))}{\xbar_+-\xbar_-}\leq 0&\implies\text{ \eqref{eqn: cone lies on top} holds}.
\end{align}
 Indeed by the definition of $\anglemin$, this would imply that
\begin{align*}
 0&\geq\euclidean{\xprime}{\frac{{\xbarprime[+]}-{\xbarprime[-]}}{\norm{\xbar_+-\xbar_-}}}+\heightfunc(\xprime)\left(\frac{\xbar_+^n-\xbar_-^n}{\norm{\xbar_+-\xbar_-}}\right)\\
 &\geq \euclidean{\xprime}{\frac{{\xbarprime[+]}-{\xbarprime[-]}}{\norm{\xbar_+-\xbar_-}}}+\heightfunc(\xprime)\cos{\anglemin}
\end{align*}
and rearranging terms,
\begin{align*}
\heightfunc(\xprime)&\leq \frac{1}{\cos{\anglemin}}\euclidean{-\xprime}{\frac{{\xbarprime[+]}-{\xbarprime[-]}}{\norm{\xbar_+-\xbar_-}}}\\
&\leq  \frac{\norm{\xprime}}{\cos{\anglemin}}\frac{\norm{{\xbarprime[+]}-{\xbarprime[-]}}}{\norm{\xbar_+-\xbar_-}}\\
&\leq  \norm{\xprime}\tan{\anglemin}\leq x^n,
\end{align*}
giving \eqref{eqn: cone lies on top}. Now let $\xbar_{0, \pm}\in\subdiff{\convpot[\pm]}{0}$ and $\tilde{\xbar}_{\pm}\in\subdiff{\convpot[\pm]}{\xprime, \heightfunc(\xprime)}$; by Lemma~\ref{lem: dm} we have that $\xbar_{0, \pm}\in\subdiff{u}{0}$ and $\tilde{\xbar}_{\pm}\in\subdiff{u}{\xprime, \heightfunc(\xprime)}$. In particular, 
\begin{align*}
 u(y)&\geq u(0)+\max\left\{\euclidean{y}{\xbar_{0, +}}, \euclidean{y}{\xbar_{0, -}}\right\},\\
 u(y)&\geq u(\xprime, \heightfunc(\xprime))+\max\left\{\euclidean{y-(\xprime, \heightfunc(\xprime))}{\tilde\xbar_{+}}, \euclidean{y-(\xprime, \heightfunc(\xprime))}{\tilde\xbar_{-}}\right\}
\end{align*}
for any $y$. Taking $y=(\xprime, \heightfunc(\xprime))$ in the first and $y=0$ in the second inequality, plugging the second into the first and rearranging terms we obtain
\begin{align}
\euclidean{(\xprime, \heightfunc(\xprime))}{\tilde\xbar_{-}} &\geq\min\left\{\euclidean{(\xprime, \heightfunc(\xprime))}{\tilde\xbar_{+}}, \euclidean{(\xprime, \heightfunc(\xprime))}{\tilde\xbar_{-}}\right\}\notag\\
 &\geq \max\left\{\euclidean{(\xprime, \heightfunc(\xprime))}{\xbar_{0, +}}, \euclidean{(\xprime, \heightfunc(\xprime))}{\xbar_{0, -}}\right\}\notag\\
 &\geq \euclidean{(\xprime, \heightfunc(\xprime))}{\xbar_{0, +}}.\notag%\label{eqn: min-max inequality}
\end{align}
Thus we have \eqref{eqn: monotonicity gives lipschitz}, hence \eqref{eqn: cone lies on top}.

A symmetric argument can be used to show $x^n\leq \heightfunc(\xprime)$ whenever $x^n\leq -\norm{\xprime}\tan{\anglemin}$, as a result we obtain the Lipschitz bound \eqref{eqn: lipschitz bound}.
\end{proof}

\section{Stability of singularities}\label{section: stability}
In this section, we will use the explicit function theorem from the previous section to show a stability result for singularities, we will extend our discussion from convex functions to semi-convex functions. First a few definitions.
\begin{defin}[Semi-convexity]\label{defin: semiconvex function}
 Recall that a real valued function $u$ defined on some $\altsourcedom\subset\R^n$ is said to be \emph{semi-convex} if for any $x_0\in \altsourcedom$, there exists a neighborhood of $x_0$ and some $C>0$ for which the function $x\mapsto u(x)+C\norm{x-x_0}^2$ is convex on that neighborhood. We will say that a family $\{u_j\}$ of semi-convex functions has \emph{uniformly bounded constant of semi-convexity near $x_0$} if there is some neighborhood of $x_0$ on which the same constant $C>0$ can be chosen to make all of the functions $u_j+C\norm{\cdot-x_0}^2$ convex on that neighborhood.
 
A function $u$ defined on an open set in a smooth manifold is said to be \emph{semi-convex} if the above definition holds near any point in a local coordinate chart.
\end{defin}
\begin{defin}[Subdifferential of a semi-convex function]
 The \emph{subdifferential} of a semi-convex function $u$ defined on a subset of a Riemannian manifold $(M, g)$ is defined by
\begin{align*}
 \subdiff{u}{x_0}:=\left\{p\in \cotanspM{x_0}\mid u(\exp_{x_0}(v))\geq u(x_0)+p(v)+o(\norm{v}_g),\ \forall \tanspM{x_0}\ni v\to 0\right\}
\end{align*}
where $\exp_{x_0}$ is the Riemannian exponential map.
\end{defin}
If $u$ is a convex function on a subdomain of $\R^n$, this definition is equivalent to \eqref{eqn: convex subdifferential}.
%\begin{defin}\label{defin: singular point}
%Let $u$ be a semi-convex function defined on a domain $\altsourcedom\subset\R^n$. We say $x_0\in\altsourcedom$ is a \emph{singular point of order $k$ for $u$} if the affine dimension of $\subdiff{u}{x_0}$ is equal to $k$.
%\end{defin}
\begin{defin}[Legendre transform]\label{defin: legendre transform}
 If $u$ is a real-valued %lower semicontinuous 
function defined on some subdomain $\Dom(u)$ of $\R^n$, its \emph{Legendre transform} is the convex function defined by the equation \eqref{eqn: legendre transform} %after setting 
with the convention $u:=\infty$ outside $\Dom(u)$.
\end{defin}
It is well known that for a semi-convex function $u$, if $\subdiff{u}{x}$ is a singleton for some $x$, then $u$ is actually differentiable at $x$. We will be interested in the behavior of $u$ at points of \emph{nondifferentiability}, namely we will be concerned with the \emph{dimension} of $\subdiff{u}{x}$ (whenever we refer to the dimension of a convex set, we will always mean the dimension of its affine hull). In some sense, this dimension is a measure of how severe the singularity of $u$ is at $x$: for example the function $\norm{x}$ on $\R^n$ has an $n$ dimensional subdifferential at the origin which corresponds to a conical singularity, while $\norm{x^1}$ has a $1$ dimensional subdifferential at the origin, and the function remains differentiable in the $\{x^1=0\}$ subspace.

In particular, we are interested in the stability of the dimension of the subdifferential of a sequence of semi-convex functions, as detailed in the following theorem, whose proof is deferred to the end of this section.

\begin{thm}[Stability of singularities]\label{thm: stability theorem}
 Suppose that $\convpot$ is a real valued function, finite on an open neighborhood $\nbhd[x_0]$ of some point $x_0\in\R^n$, of the form
 \begin{align}\label{eqn: stability maximum representation of potential}
 \convpot &=\max_{1\leq i\leq K}\convpot[i],
 \end{align}
for some $K<\infty$  where all  $\convpot[i]$ are semi-convex. 
 Also fix some $1\leq k\leq \min\left\{K-1, n\right\}$ and assume that for any $1\leq i\leq k+1$: 
\begin{align*}
\convpot[i]&\in C^1(\nbhd[x_0]),\\
\convpot(x_0)&= \convpot[i](x_0)>\convpot[i'](x_0),\ \forall\;k+2\leq i'\leq K,
\end{align*}
and $\dim\subdiff{\convpot}{x_0}=k$. Finally, let $\left\{\perturbconvpot[i]{j}\right\}_{j=1}^\infty$ be a sequence for which  each $\perturbconvpot[i]{j}$ is semi-convex with uniformly bounded constant of semi-convexity near $x_0$, $\perturbconvpot[i]{j}\xrightarrow[j\to\infty]{}\convpot[i]$ uniformly in compact subsets of $\nbhd[x_0]$ for each $1\leq i\leq K$, and write $\displaystyle\perturbconvpot{j}:=\max_{1\leq i\leq K}\perturbconvpot[i]{j}$. Then for any $\epsilon>0$, there exists an index $J_\epsilon$ such that for any $j>J_\epsilon$, there exists a set $\perturbsing[n-k]{j}
\subset B_{\epsilon}{\left(x_0\right)}$ with $\haus[n-k]\left(\perturbsing[n-k]{j}\right)>0$ on which
\begin{align}
% \haus[n-k]\left(\perturbsing[n-k]{j}\right)>0, \qquad\haus[n-k-1+\epsilon]\left(N^j\right)=0,\ \forall\;\epsilon>0,%\notag\\%\label{eqn: zero hausdorff}\\
 \perturbconvpot{j}(x)=\perturbconvpot[i]{j}(x)>\perturbconvpot[i']{j}(x), %\notag\\ 
 \qquad \forall\; x\in \perturbsing[n-k]{j},\ 1\leq i\leq k+1,\ k+2\leq i'\leq K.\label{eqn: perturbed touching}
\end{align}
%$\haus[n-k]\left(\perturbsing[n-k]{j}\right)>0$, $\haus[n-k]\left(N^j\right)=0$, and for any $1\leq i\leq k+1$ and $k+2\leq i'\leq K$, we have
%\begin{align}\label{eqn: perturbed touching}
% \perturbconvpot{j}(x)=\perturbconvpot[i]{j}(x)>\perturbconvpot[i']{j}(x),\quad \forall\; x\in \perturbsing[n-k]{j}.
%\end{align} 
 Moreover, $\perturbsing[n-k]{j}$ is the graph of a DC mapping over an open set in $\R^{n-k}$ and 
\begin{align}\label{eqn: correct order}
%\begin{cases}
 \dim\subdiff{\perturbconvpot{j}}{x}\geq k \qquad \forall x\in\perturbsing[n-k]{j},
%\\  \dim\subdiff{\perturbconvpot{j}}{x}= k,&\forall x\in\perturbsing[n-k]{j}\setminus N^j.
%\end{cases}
\end{align}
with equality on a set of full $\mathcal{H}^{n-k}$ measure in $\perturbsing[n-k]{j}$.
%  for every $x\in\perturbsing[n-k]{j}$ it holds $\dim\subdiff{\perturbconvpot{j}}{x}=k$.
%with equality holding outside a set $N \subset 
\end{thm}

In preparation, we shall
need a result on stability of the subdifferentials of a sequence of convergent convex functions. By a straightforward modification of the proof of~\cite[Theorem 25.7]{Rockafellar70}, we obtain the following lemma.

\begin{lem}\label{lem: uniform convergence of subdifferentials}
 Suppose that $\convpot$ and $\left\{\convpot[j]\right\}_{j=1}^\infty$ are convex functions, finite and with $\convpot[j]\to\convpot$ pointwise on some open convex domain $\altsourcedom$, and also assume that $\convpot$ is differentiable on $\altsourcedom$. Then for any compact $\altsourcedom[0]\subset\altsourcedom$ and $\epsilon>0$ there exists $j_0$ such that 
\begin{align*}
 \subdiff{\convpot[j]}{x}\subset B_{\epsilon}{\left(\gradconvpot(x)\right)}
\end{align*}
for all $j\geq j_0$ and $x\in\altsourcedom[0]$.
\end{lem}
\begin{proof}
 Suppose that the proposition fails, then for some compact $\altsourcedom[0]\subset\altsourcedom$ and $\epsilon>0$, there exists a sequence $\left\{x_j\right\}_{j=1}^\infty\subset \altsourcedom[0]$ and $p_j\in\subdiff{\convpot[j]}{x_j}$ for which $\norm{p_j-\gradconvpot(x_j)}>\epsilon$. By passing to subsequences, we may assume that $x_j\to x_0\in\altsourcedom[0]$, and for some fixed index $1\leq i\leq n$ that $\euclidean{p_j-\gradconvpot(x_j)}{e_{i}}>\sqrt{\frac{\epsilon}{n}}$ for all $j$ (the case of $\euclidean{p_j-\gradconvpot(x_j)}{e_{i}}<-\sqrt{\frac{\epsilon}{n}}$ is treated by a similar argument). Then, for any $\lambda>0$, since $p_j\in\subdiff{\convpot[j]}{x_j}$ we find that
\begin{align*}
 \frac{\convpot[j](x_j+\lambda e_{i})-\convpot[j](x_j)}{\lambda}\geq\euclidean{p_j}{e_{i}}> \sqrt{\frac{\epsilon}{n}}+\euclidean{\gradconvpot(x_j)}{e_{i}}.
\end{align*}
Recalling that $\convpot[j]$ converges uniformly on compact subsets of $\altsourcedom$ and $\gradconvpot$ is continuous on $\altsourcedom$ (\cite[Theorem 10.8 and Theorem 25.5]{Rockafellar70}), by first taking the limit $j\to \infty$ (for all small enough $\lambda>0$ so that $x_j+\lambda e_{i}\in \altsourcedom$) and then $\lambda \searrow 0$, we obtain the contradiction $\euclidean{\gradconvpot(x_0)}{e_{i}}\geq \sqrt{\epsilon/n}+\euclidean{\gradconvpot(x_0)}{e_{i}}$, finishing the proof.
\end{proof}

\begin{rmk}\label{rmk: uniform convergence of subdifferentials fails when limit is not differentiable}
 We remark that if the limiting function $\convpot$ is not differentiable, then Lemma~\ref{lem: uniform convergence of subdifferentials} above fails, even upon replacing $B_{\epsilon}{\left(\gradconvpot(x)\right)}$ by $\nbhdof[\epsilon]{\subdiff{\convpot}{x}}$, as seen by the following example. On $\altsourcedom=\R$ let $\convpot[j]:=\norm{x-1/j}$ converging to $\convpot:=\norm{x}$, and take the compact subdomain $\altsourcedom[0]:=[-1, 1]$. Then if $\epsilon = 1/2$, for any $j_0\in\N$ we see that
\begin{align*}
 \subdiff{\convpot[j_0]}{\frac{1}{j_0}}=[-1, 1]\not\subset [\frac{1}{2}, \frac{3}{2}]=\nbhdof[1/2]{\subdiff{\convpot}{\frac{1}{j_0}}},
\end{align*}
hence there is no choice of $j_0$ for which the proposition holds uniformly over $[-1, 1]$.
\end{rmk}
Next we recall the \emph{generalized (Clarke) Jacobian} of a mapping $G$ (at a point $x_0$, in the last $k$ variables).
\begin{defin}[Clarke Jacobian]\label{def: generalized jacobian}
 If $G: B_\epsilon(x_0)\subset \R^n\to \R^k$ is a Lipschitz function on a neighbourhood of $x_0$, we define $J^CG(x_0)$ to be the closed convex hull of all $k\times n$ matrices which can be written as limits of the form
\begin{align*}
 \lim_{n\to\infty} DG(x_n)
\end{align*}
where $x_n\to x_0$ and $G$ is differentiable at each $x_n$.

Moreover if $1\leq k\leq n$, using the notation $x=(x', x'')\in \R^{n-k}\times \R^k$ we write $J^C_{x''}G(x_0)$ for the set of $k\times k$ matrices consisting of the last $k$ columns of elements in $J^CG(x_0)$.
%
% If $G: B_\epsilon(x_0)\subset \R^{n-k}\times \R^k\to \R^k$ is a Lipschitz function on a neighbourhood of $x_0=(x_0', x_0'')$, we define $J^C_{x''}G(x_0', x_0'')$ to be the closed convex hull of all $k\times k$ matrices which can be written as limits of the form
%\begin{align*}
% \lim_{n\to\infty} D_{x''}G(x_n', x_n'')
%\end{align*}
%where $x_n\to x_0$ and $G$ is differentiable at each $x_n$.
\end{defin}
A combination of Clarke's inverse function theorem \cite[Theorem 1]{Clarke76} and results of 
Vesely and Zaj\'i\v cek \cite{VeselyZajicek89} 
on DC mappings yields the following DC implicit function theorem.
\begin{thm}[DC implicit mapping theorem {\cite[Proposition 5.9]{VeselyZajicek89}}]\label{thm: DC IFT}
 Suppose $U\subset \R^{n-k}\times \R^k$ is open, $G: U\to \R^k$ is a DC mapping, and $G(x_0)=0$ for some $x_0=(x_0', x_0'')\in U$. Then if every element of $J^C_{x''}G(x_0)$ is invertible, 
there exists $\delta>0$ and a bi-Lipschitz, DC mapping $\phi$ from $B_\delta (x_0')\subset \R^{n-k}$ into 
$\R^k$ such that for all $(x',x'') \in B_\delta(x_0') \times B_\delta(x_0'') \subset \R^{n-k} \times \R^k$:
\begin{align*}
 G(x', x'') = 0 \quad {\rm if\ and\ only\ if} \quad x'' = \phi(x').
\end{align*}
\end{thm}
Additionally, a careful inspection of the proof of \cite[Theorem 3.1]{ButlerTimourianViger88} combined with \cite[Theorem 5.1]{VeselyZajicek89} yields the following DC constant rank theorem.
\begin{thm}[DC constant rank theorem]\label{thm: DC constant rank}
 Suppose $U\subset \R^n$ is open, $G: U\to \R^k$ is a DC mapping, and $G(x_0)=0$ for some $x_0\in U$. Then if every element of $J^CG(x_0)$ has rank $k$, after a possible re-ordering and rotation of coordinates, the same conclusion as Theorem \ref{thm: DC IFT} above holds.
%  there exists a $\delta>0$ and a bi-Lipschitz, DC mapping $\phi$ from $B_\delta (x_0')\subset \R^{n-k}$ into 
%$\R^k$ such that for all $(x',x'') \in B_\delta(x_0') \times B_\delta(x_0'') \subset \R^{n-k} \times \R^k$:
%\begin{align*}
% G(x', x'') = 0 \quad {\rm if\ and\ only\ if} \quad x'' = \phi(x').
%\end{align*}
\end{thm}

We shall also need:

\begin{lem}[Coincident roots]\label{lem: existence of singular point}
 Suppose $\phi^{\pm}_1, \ldots, \phi^{\pm}_k$ are real valued convex functions on $[-1, 1]^n$, such that $\phi^\pm_i >\phi^\mp_i$ on the set $\{x\in [-1, 1]^n\mid x^i=\pm 1\}$, and $\subdiff{\phi_i^+}{[-1, 1]^n}$ is strongly separated from $\subdiff{\phi_i^-}{[-1, 1]^n}$ by a hyperplane normal to $e_i$ for each $1\leq i\leq k$. Then, there exists a point in $]-1, 1[^n$ where all $2k$ functions $\phi^\pm_1=\ldots=\phi^\pm_k$ agree.
\end{lem}
\begin{proof}
 For any $x\in \R^n$, let us write $\hat{x}^i:=(x^1,\ldots, x^{i-1}, x^{i+1}, \ldots, x^n)$. Fix $1\leq i\leq k$, by Corollary \ref{cor: explicit function theorem}, there is a DC function $h_i$ defined on all of $\hat{I}_i:=\{\hat{x}^i\mid x\in [-1, 1]^n\}$ such that the graph of $h_i$ over this set is exactly 
 \begin{align*}
 \{x\in [-1, 1]^n\mid\phi^+_i(x)=\phi^-_i(x)\};
 \end{align*}
  by the intermediate value theorem we see for any $\hat{x}\in\hat{I}_i$ there exists $x\in [-1, 1]^n$ where $\phi^+_i(x)=\phi^-_i(x)$ and $\hat{x}^i=\hat{x}$, and in particular the range of $h_i$ is contained in $[-1, 1]$. Now define the mapping $F: [-1, 1]^n\to [-1, 1]^n$ by 
 \begin{align*}
 F(x):=(h_1(\hat{x}^1), \ldots, h_k(\hat{x}^k), x^{k+1},\ldots, x^n),
 \end{align*}
  this mapping is continuous by Theorem \ref{thm: explicit function theorem}, thus by Brouwer's fixed point theorem it has a fixed point in $[-1, 1]^n$. However, we see that at this fixed point we must have $\phi^\pm_1=\ldots=\phi^\pm_k$, by the assumptions on the $\phi^\pm_i$ this point clearly must be in the interior $]-1, 1[^n$.
\end{proof}
 
With these preparations, we are ready to prove the main stability result.

\begin{proof}[Proof of Theorem \ref{thm: stability theorem}]
By \cite[Theorem 1]{Alberti94}, the set of points $x$ where $\dim\subdiff{u}{x}\geq k+1$ has zero $\haus[n-k]$ measure, hence the final claim will follow immediately from \eqref{eqn: correct order}.

Suppose we are given $\convpot$, $x_0$, and a sequence $\left\{\perturbconvpot{j}\right\}_{j=1}^\infty$ as in the hypotheses of Theorem \ref{thm: stability theorem}.
 %Next we comment that the convexity of all relevant functions and the dimension of subdifferentials are preserved under any affine (invertible) change of coordinates on the domain $\R^n$, as well as under subtraction of a fixed linear function simultaneously from all $\perturbconvpot[i]{j}$ and $\convpot[i]$. 
Now by Lemma~\ref{lem: dm} we have 
\begin{align}
\subdiff{\convpot}{x_0}=\ch\left(\bigcup_{1\leq i\leq k+1}\{\gradconvpot[i](x_0)\}\right),\label{eqn: restricted subdifferential}
\end{align}
and since $\affdim{\left(\subdiff{\convpot}{x_0}\right)}=k$, the collection $\left\{\gradconvpot[i](x_0)-\gradconvpot[k+1](x_0)\right\}_{i=1}^{k}$ must be linearly independent, subtraction of a fixed linear function followed by a linear change of coordinates allows us to assume $\gradconvpot[i](x_0)=e_{n-k+i}$ for $1\leq i\leq k$ and $\gradconvpot[k+1](x_0)=0$. Next fix $\epsilon>0$, without loss of generality assume that $B_{\epsilon}{\left(x_0\right)}\subset \nbhd[x_0]$. By our assumptions, we may add a fixed quadratic function centered at $x_0$ to assume all $\perturbconvpot[i]{j}$ and $\convpot[i]$ are convex on $B_{\epsilon}{\left(x_0\right)}$, for $1\leq i\leq k+1$ (possibly shrinking $\epsilon$ as well). By taking $j$ large enough and possibly shrinking $\epsilon$ further, by the uniform convergence of each $u^j_i$ we may assume
\begin{align}\label{eqn: domination by first k+1}
\min_{1\leq i\leq k+1}u^j_i> \max_{k+2\leq i\leq K}u^j_i
\end{align}
on $B_\epsilon(x_0)$.

Define the mapping $F^j: B_\epsilon (x_0)\to \R^k$ by 
\begin{align*}
 F^j(x):=(\perturbconvpot[1]{j}(x)-\perturbconvpot[k+1]{j}(x),\ldots, \perturbconvpot[k]{j}(x)-\perturbconvpot[k+1]{j}(x))
\end{align*}
then we see that if $x\in B_{\epsilon}(x_0)$, the set $J^C_{x''}F^j(x)$ is contained in the collection of $k\times k$ matrices for which the $i$th row is contained in the convex hull of vectors of the form 
\begin{align*}
 \lim_{m\to\infty}D_{x''} (\perturbconvpot[i]{j}-\perturbconvpot[k+1]{j})(x_m)
\end{align*}
where $x_m\to x$ and $u^j_i$, $u^j_{k+1}$ are differentiable at each $x_m$. Here $D_{x''}$ indicates the projection of the gradient of a function onto the last $k$ variables. Since each function $u_i$ is $C^1$, after shrinking $\epsilon$ if necessary and taking $j$ large enough, by applying Lemma \ref{lem: uniform convergence of subdifferentials} we can assume that for any $x\in B_{\epsilon} (x_0)$ and $p^j_i\in \subdiff{\perturbconvpot[i]{j}}{x}$ we have 
\begin{align}\label{eqn: perturbed subdifferentials independent}
\begin{cases}
 p^j_i \in B_{\frac{1}{4}}(e_i),& 1\leq i\leq k,\\
 p^j_{k+1} \in B_{\frac{1}{4}}(0).&
\end{cases}
\end{align}
In particular, this implies that every matrix in $J^C_{x''}F^j(x)$ will be invertible, thus we can apply the DC implicit mapping theorem above to $F^j$, provided there exists at least one point $x_j \in B_{\epsilon}(x_0)$ where $F^j$ vanishes.

To this end, we translate so $x_0=0$, then we can apply the $C^1$ implicit function theorem to $u_i-u_{k+1}$ for each $1\leq i\leq k$.  For $\eta>0$ small enough we then get $u_i-u_{k+1}>0$ on $\{x\in [-\eta, \eta]^n\mid x^i= \eta\}$ while $u_i-u_{k+1}<0$ on $\{x\in [-\eta, \eta]^n\mid x^i=- \eta\}$ for all $i \le k$. 
For any $j$ large enough $u^j_i-u^j_{k+1}$ satisfies the same inequalities.
Thus recalling \eqref{eqn: perturbed subdifferentials independent},  a dilation by $1/\eta$ allows us
to apply Lemma \ref{lem: existence of singular point} above to conclude the existence of a sequence 
$x_j \in\ ]-\eta, \eta[^n \subset B_\epsilon(x_0)$ such that $F^j(x_j)=0$.
In particular, we may now apply the DC implicit mapping theorem to find a ball $B^j\subset \pi_{n-k}(B_\epsilon(x_0))$ and a DC mapping $\Phi^j: B^j\to B_\epsilon(x_0)$ whose graph passes through $x_j$
for which $u^j_1(\Phi^j(x'))=\ldots=u^j_{k+1}(\Phi^j(x'))$ for all $x'\in B^j$. Let 
\begin{align*}
 \perturbsing[n-k]{j}:=\{(x', \Phi^j(x'))\mid x'\in B^j\}\cap B_\epsilon(x_0).
\end{align*}
As a Lipschitz graph over $B^j\subset \R^{n-k}$ we see $\perturbsing[n-k]{j}$ has strictly positive $\haus[n-k]$ measure. Thus by Lemma \ref{lem: dm}, this implies %the first line of 
\eqref{eqn: correct order}, while \eqref{eqn: domination by first k+1} yields \eqref{eqn: perturbed touching}
to finish
%Finally,  define
%\begin{align*}
%N^j:= \left\{x\in B_{\epsilon}(x_0)\mid \dim\subdiff{\perturbconvpot{j}}{x}\geq k+1\right\}.
%\end{align*}
%By \cite[Theorem 1]{Alberti94}) we have $\haus[n-k-1+\epsilon]\left(N^j\right)=0$ for any $\epsilon>0$, while by %construction we have the second line of \eqref{eqn: correct order}, finishing 
the proof.
\end{proof}

\section{Applications to optimal transport}\label{section: stability of OT}
In this sequel, we apply the explicit function theorem and stability theorems from the previous two sections to the optimal transport problem. Throughout, $\outerdom$ and $\outertarget$ are compact subsets of $n$-dimensional Riemannian manifolds $(M, g)$ and $(\Mbar, \gbar)$ respectively, $\outerdom^\bdry$ is assumed to have dimension less than or equal to $n-1$, and $c\in C^4(\outerdom\times \outertarget)$. Also the notation $\haus[i]_g$ will refer to the $i$-dimensional Hausdorff measure of a set defined using the distance derived from the Riemannian metric $g$.

We begin by recalling a number of notions and conditions from the theory of the optimal
transportation problem \eqref{OT} which adapt concepts from convex analysis such as the Legendre
transform \eqref{eqn: legendre transform} to choices of cost other than $c(x,\bar x) = -\euclidean{x}{\bar x}$.

\begin{defin}[$c$-convex functions]\label{defin: c-transform}
 For a proper lower semicontinuous function $u:\outerdom\to\R\cup \left\{\infty\right\}$, its \emph{$c$-transform} $u^c: \outertarget\to\R\cup \left\{\infty\right\}$ is defined by 
\begin{align*}
 u^c(\xbar):=\sup_{x\in\outerdom}(-c(x, \xbar)-u(x)).
\end{align*}
Also its \emph{double $c$-transform} $u^{cc^*}:\outerdom\to\R\cup \left\{\infty\right\}$ is defined by 
\begin{align*}
 u^{cc^*}(x):=\sup_{\xbar\in\outertarget}(-c(x, \xbar)-u^c(\xbar));
\end{align*}
$u$ is said to be \emph{$c$-convex} if $u=u^{cc^*}$ on $\outerdom$. Its \emph{$c$-subdifferential} at a point $x_0\in\outerdom$ is the set 
\begin{align*}
 \csubdiff{u}{x_0}:=\left\{\xbar\in\outertarget\mid -c(x, \xbar)+c(x_0, \xbar)+u(x_0)\leq u(x),\quad \forall\;x\in\outerdom\right\}.
\end{align*}
Likewise, the \emph{$c^*$-subdifferential} of $u^c$ at $\xbar_0\in \outertarget$ is defined as
\begin{align*}
 \cstarsubdiff{u^c}{\xbar_0}:=\left\{x\in\outerdom\mid -c(x, \xbar)+c(x, \xbar_0)+u^c(\xbar_0)\leq u^c(\xbar),\quad \forall\;\xbar\in\outertarget\right\}.
\end{align*}
Finally, for any subset $A\subset \outerdom$, we write
\begin{align*}
 \csubdiff{u}{A}:=\bigcup_{x\in A}{\csubdiff{u}{x}},
\end{align*}
and analogously for $\bar A\subset \outertarget$ and $\cstarsubdiff{u^c}{\bar A}$.
\end{defin}
\begin{rmk}[Strict $c$-convexity]
 It can be shown $u$ is a $c$-convex function if and only if for every $x_0\in \Omega$, the set $\csubdiff{u}{x_0}\neq \emptyset$. If in addition, for every $x_0\in\outerdom$ and $\xbar_0\in \csubdiff{u}{x_0}$, we have
\begin{align*}
\left\{x\in \outerdom\mid -c(x, \xbar_0)+c(x_0, \xbar_0)+u(x_0)=u(x)\right\}= \left\{x_0\right\}
\end{align*}
we say that $u$ is \emph{strictly $c$-convex}.
\end{rmk}
%\begin{rmk}\label{rmk: restricted c transform}
% At times we may have use for a \emph{restricted} version of the $c$-transform of a function, if $A\subset \Omega$, we define
% \begin{align*}
% u^c_A(\xbar):=\sup_{x\in A}(-c(x, \xbar)-u(x)).
%\end{align*}
%It is not hard to see that $u^c_A$ is $c$-convex on $\outertarget$, and if $u$ is a $c$-convex function on $\outerdom$, then for any $x\in A$ we have
%\begin{align*}
% u(x)=\sup_{\xbar\in \outertarget}(-c(x, \xbar)+u^c_A(\xbar)).
%\end{align*}
%\end{rmk}
We also say $c$ satisfies \eqref{B1} if for any $x_0\in\outerdom$ and $\xbar_0\in \outertarget$, the mappings 
\begin{align}
 \xbar &\mapsto -D_x c(x_0, \xbar)\in\cotanspM{x_0},\notag\\
 x&\mapsto -D_{\xbar} c(x, \xbar_0)\in\cotanspMbar{\xbar_0},\label{B1}\tag{B1}
\end{align}
are diffeomorphisms on $\outerdom$ and $\outertarget$ respectively (these are classical conditions in optimal transport, corresponding for example to the twist and non-degeneracy conditions (A1) and (A2) in \cite{KimMcCann10}). We will also write $\cExp{x_0}{\cdot}$ for the inverse of the map in the first line above. Also for any sets $A\subset \outerdom$ and $\bar A\subset \outertarget$, we will use the shorthand notation
\begin{align}
 \coord{A}{\xbar}:&=-D_{\xbar} c(A, \xbar),\notag\\
 \coord{\bar A}{x}:&=-D_x c(x, \bar A).\label{eqn: c-exp image notation}
\end{align}
At this point we recall a classical result about existence of solutions to \eqref{OT}, originally due to Brenier for the case of the cost function $c(x, \xbar)=-\euclidean{x}{\xbar}$.
\begin{thm}[Optimal transport maps \cite{Brenier91, Gangbo95, GangboMcCann96, Levin99, McCann01, MaTrudingerWang05}]\label{thm: Brenier}
 If $c$ satisfies \eqref{B1} and $\sourcemeas$ is absolutely continuous with respect to the volume measure on $\M$, then there exists a $c$-convex function $u: \outerdom\to\R$ which is differentiable almost everywhere, and the map $T(x):=\cExp{x}{Du(x)}$ is a solution to \eqref{OT} with $T(\Dom Du) \subset \spt \mu$. We call such a $u$ an \emph{optimal potential} transporting $\sourcemeas$ to $\targetmeas$, with cost $c$.
\end{thm}

 In this first lemma, we show that if the support of the target measure consists of a (finite) union of disjoint, compact pieces, we can write the optimal potential as a maximum (of a finite number) of corresponding $c$-convex functions. For any function $u$, we will write $\Dom(Du)$ for the set of points where $u$ is differentiable, which in the case of a semi-convex function (thus in particular, for any $c$-convex function under our assumptions) is a set of full Lebesgue measure in $\Dom(u)$.

\begin{lem}[Optimal maps to separated targets]\label{lem: u is a maximum of potentials}
 Suppose a cost function $c$ satisfies \eqref{B1}, $\sourcemeas$ is absolutely continuous, and $\targetmeas$ is such that $\target$ is a disjoint union of an arbitrary (i.e. finite, coutable, or uncountable) 
collection $\{\targetpiece[i]\}_{i\in I}$ %,\ldots, \targetpiece[K]$ 
 of compact subsets of the compact set $\outertarget$. The $c$-convex functions $\convpot[i]: \outerdom\to \R$, $i\in I$
defined by
\begin{align}\label{eqn: piecewise c-transform}
\convpot[i](x):=\sup_{\xbar\in\targetpiece[i]}(-c(x, \xbar)-u^c_{}(\xbar))
\end{align}
satisfy
 \begin{align}
 \cExp{x}{D\convpot[i](x)}&\in \targetpiece[i],\quad \forall\;x\in\Dom(Du),\ \forall\; i \in I,\label{eqn: gradient maps into pieces}\\
 \convpot(x) &=\sup_{i\in I}\convpot[i](x),\quad\forall\;x\in \outerdom.\label{eqn: maximum representation of potential}
 \end{align}
\end{lem}
\begin{proof}
First observe $\convpot[i]$ is finite valued on all of $\outerdom$. Clearly $\convpot[i]$ is $c$-convex, hence differentiable a.e.. Fix $i$ and let $x$ be such a point of differentiability, by compactness of $\targetpiece[i]$ there exists an $\xbar\in\targetpiece[i]$ achieving the supremum in the definition of $\convpot[i](x)$. The inclusion \eqref{eqn: gradient maps into pieces} then follows immediately by differentiation of $\convpot[i]$ at $x$ and \eqref{B1}.
 
 Now as $\convpot$ is $c$-convex by we see that for $x\in\outerdom$, 
\begin{align}
 \convpot(x)
 &=\sup_{\xbar\in\outertarget}{\left[-c(x,\xbar)-u^c_{}(\xbar)\right]}\notag\\
 &=\sup_{\xbar\in\target}{\left[-c(x,\xbar)-u^c_{}(\xbar)\right]}\notag\\
 &=\sup_{i\in I}\convpot[i](x),\label{eqn: maximum rep}
\end{align}
proving \eqref{eqn: maximum representation of potential}. The reason why we may change the supremum above from being over $\outertarget$ to just over $\target$ is as follows. As mentioned previously, $u$ is differentiable almost everywhere on $\outerdom$, so there exists a sequence $x_j\to x$ where $u$ is differentiable at $x_j$ and $\exists\;\xbar_j\in \csubdiff{u}{x_j}=\left\{\cExp{x_j}{Du(x_j)}\right\}$ for each $j$. By \cite[Theorem 10.28]{Villani09} (the assumption \textbf{(H$\infty$)} of the reference is automatically satisfied by our assumption that $\outertarget$ is bounded) we must have $\xbar_j\in \target$, then by compactness, we may pass to a subsequence and assume $\xbar_j\to \xbar_0$ for some $\xbar_0\in\target$, necessarily $\xbar_0\in \csubdiff{u}{x}$. However, this implies
\begin{align*}
 \sup_{\xbar\in\outertarget}{\left[-c(x,\xbar)-u^c_{}(\xbar)\right]}&=\sup_{\xbar\in\outertarget}\inf_{y\in\outerdom}{[-c(x,\xbar)+c(y, \xbar)+u(y)]}\\
 &\leq u(x)\leq -c(x, \xbar_0)+c(y, \xbar_0)+u(y)
\end{align*}
for any $y\in\outerdom$, thus we may take the supremum merely over $\target$.
\end{proof}

\begin{rmk}\label{rmk: no compactness needed}
 We pause to remark here that the above lemma will hold true even if $\outerdom$ is not necessarily bounded.
%which will be relevant in the case of 
This is relevant for the cost $c(x, \xbar)=-\euclidean{x}{\xbar}$ with $\outerdom=\R^n$,
for which it is established as in \cite{GangboMcCann00};  see also \cite{GangboMcCann96} for more general costs.
%First, $u_i$ is finite valued everywhere. Indeed $u^c$ will be the Legendre transform $u^*$ which is a proper, convex %function (since $u$ is such). Thus $u^*$ is bounded from below on any compact subset of $\R^n$, thus by %compactness of each set $\targetpiece[i]$ we have finiteness of $u_i$ on all of $\R^n$. Also by \cite[Example 10.35]%{Villani09}, the cost function $c(x, y)=\norm{x-y}^2$ satisfies all of the assumptions of \cite[Theorem 10.28]{Villani09}, and it is well known that optimal maps $T$ for this cost and the cost $-\euclidean{x}{y}$ coincide, thus the %final proof where the supremum is restricted to $\target$ still holds. The remainder of the proof follows in the same manner as above.
% and in particular can be written as $T(x)=Du(x)$ for $\mu$-a.e. x, where $u:\R^n\to \R\cup \{\infty\}$ is a convex %function.  By \cite[Equation (10.21)]{Villani09}, we then have $\target=(Du(\source)\cap \Dom Du)^{\cl}$, $\Dom Du$ %being the set $u$ is differentiable, which is a full measure subset of $\R^n$. Since each $\targetpiece[i]$ is compact %and the sets are mutually disjoint, this implies that there must exist at least one point $\xbar_i\in\targetpiece[i]$ for each %$i\in I$ which can be written as $Du(x_i)$ for some $x_i\in \outerdom$. In turn, by \eqref{B1} this means %$\xbar_i\in\csubdiff{u}{x_i}=\subdiff{u}{x_i}$. Thus we easily see that $u^c(\xbar_i)<\infty$, hence $u_i$ cannot equal %$-\infty$ anywhere. On the other hand, $u_i\leq u$ everywhere, so it must be finite on $\outerdom$.
\end{rmk}
In order to discuss regularity, we will require some more geometric structure.
\begin{defin}[$c$-convex sets]\label{defin: c-convex sets}
 $A\subset \outerdom$  ($\bar A\subset \outertarget$) is \emph{$c$-convex ($c^*$-convex) with respect to $\xbar_0$ ($x_0$)} if the set $ \coord{A}{\xbar_0}$ ($\coord{\bar A}{x_0}$) from \eqref{eqn: c-exp image notation} is convex. We say \emph{$A$ is $c$-convex with respect to $\bar A$} if $A$ is $c$-convex with respect to every $\xbar\in \bar A$, and \emph{$\bar A$ is $c^*$-convex with respect to $A$} analogously. Finally, the phrase \emph{$A$ and $\bar A$ are $c$-convex with respect to each other} if
both hold. We also refer to \emph{strictly $c$-convex} and \emph{strongly $c$-convex} by adding the corresponding modifiers to the convexity of $\coord{A}{\xbar_0}$ or $\coord{\bar A}{x_0}$,  {\em strict} convexity meaning the midpoint of any nontrivial segment in $\coord{A}{\xbar_0}$  
lies in the interior of $\coord{A}{\xbar_0}$,  and {\em strong} convexity meaning $\coord{A}{\xbar_0}$ can be expressed as the intersection of a family of balls with fixed radii.
\end{defin}
Lastly, $c \in C^4(\Omega \times \bar \Omega)$ satisfies the \eqref{MTW}  or (Ma-Trudinger-Wang) condition if
\begin{defin}[MTW costs \cite{MaTrudingerWang05, TrudingerWang09}]
For some constant $a_3 \ge 0$,
all $(x, \xbar)\in \outerdom\times\outertarget$, $V\in \tanspM{x}$ and $\eta\in\cotanspM{x}$ with $\eta(V)=0$
satisfy
\begin{align}\label{MTW}\tag{MTW}
-(c_{ij, \bar{r}\bar{s}}-c_{ij, \bar{t}}c^{\bar{t}, s}c_{s, \bar{r}\bar{s}})c^{\bar{r}, k}c^{\bar{s}, l}(x, \xbar)V^iV^j\eta_k\eta_l\geq a_3 \lvert V\rvert^2_g\lvert \eta\rvert^2_g\ge 0.
\end{align}
\end{defin}
Here, local coordinate systems are fixed near $x$ and $\xbar$, and subscripts before a comma indicate differentiation with respect to the $x$ variable, those after a comma are differentiation with respect to the $\xbar$ variable, and two raised indices indicate the matrix inverse. This last condition was crucial for regularity in the pioneering works \cite{MaTrudingerWang05, TrudingerWang09}, it was later shown to have geometric implications by Loeper in \cite{Loeper09} and by Kim and McCann in \cite{KimMcCann10}.

A particular geometric consequence of \eqref{MTW} that we will need is the following lemma,
which follows from \cite{Loeper09} \cite{KimMcCann10}.

\begin{lem}[Connected $c$-subdifferential images]\label{lem: connectivity of c-subdifferential}
Let $c$ satisfy \eqref{B1} and \eqref{MTW},  and $\outertarget$ and $\outerdom$ be $c$-convex with respect to each other. Then if $\mathcal{C}\subset \outerdom$ is connected and $u$ is a $c$-convex function on $\outerdom$, then $\csubdiff{u}{\mathcal{C}}$ is connected.
\end{lem}
\begin{proof}
 Suppose not, then there exist disjoint, closed sets $\Cbar_1$ and $\Cbar_2\subset \outertarget$ such that $\csubdiff{u}{\mathcal{C}}\subset \Cbar_1\cup \Cbar_2$, and $\csubdiff{u}{\mathcal{C}}\cap\Cbar_i\neq \emptyset$ for $i=1$, $2$. Define $C_i:=\cstarsubdiff{u^c}{\Cbar_i}$ for $i=1$, $2$. Since $\xbar\in\csubdiff{u}{x}$ if and only if $x\in\cstarsubdiff{u^c}{\xbar}$ we immediately have $\mathcal{C}\cap C_i\neq \emptyset$ for each $i$ while $\mathcal{C}\subset C_1\cup C_2 $. On the other hand, suppose there exists $x\in C_1\cap C_2\cap \mathcal{C}$. Then there exist $\xbar_1\in\Cbar_1$ and $\xbar_2\in\Cbar_2$ such that both are contained in $\csubdiff{u}{x}$. However \cite[Theorem 3.1]{Loeper09} implies the set $L_x:=\{\cExp{x}{(1-\lambda)(-D_xc(x, \xbar_1))+\lambda (-D_xc(x, \xbar_2)}\mid \lambda\in [0, 1]\}$ is contained in $\csubdiff{u}{x}\subset \Cbar_1\cup \Cbar_2$. This is a contradiction as $\Cbar_1$ and $\Cbar_2$ would disconnect $L_x$, thus $C_1\cap C_2\cap \mathcal{C}=\emptyset$. Last, since each $\Cbar_i$ is compact, we immediately see that $C_i$ is also compact, hence closed. Thus we have a contradiction with the connectedness of $\mathcal{C}$.
\end{proof}

We are particularly interested in optimal transport problems between measures $\sourcemeas$ and $\targetmeas$ satisfying the following properties, which are related to regularity results proved for $c(x,\bar x)= -\euclidean{x}{\xbar}$
by Caffarelli \cite{Caffarelli92} and extended to other costs in \cite{FigalliKimMcCann13} \cite{GuillenKitagawa15};
see also Chen and Wang \cite{ChenWang16}, and V\'etois \cite{Vetois15}:

\begin{enumerate}[(I)]
% \item $\source\subset\outerdom$, while $\target\subset\outertarget$ is a disjoint union of a finite collection of sets $\left\{\targetpiece[i]\right\}_{i=1}^{K}$, and each $\targetpiece[i]$ is strictly $c$-convex with respect to $\outerdom$.
 \item Both $\sourcemeas$ and $\targetmeas$ are absolutely continuous with respect to the respective volume measures on $\M$ and $\Mbar$, and with densities bounded a.e. away from $0$ and $\infty$ on their supports.
 \item $\outerdom$ and $\outertarget$ are $c$-convex with respect to each other and either 
 \begin{align}\label{eqn: FKM conditions}
  \source \text{ and }\outertarget\text{ are strongly }c\text{-convex with respect to each other}
  \end{align}
 or 
 \begin{align}
& \source\subset\outerdom^{\interior},\ \target\subset\outertarget^{\interior},\notag\\
&\source\text{ is }c\text{-convex with respect to }\outertarget.\label{eqn: GK conditions}
 \end{align}
\end{enumerate}
Under the above conditions and \eqref{MTW}, we can make the following improvement of Lemma \ref{lem: u is a maximum of potentials}. The idea is based on one used by Caffarelli and McCann \cite[Theorem 6.3]{CaffarelliMcCann10} for the cost function $c(x, \xbar)=-\euclidean{x}{\xbar}$.

%\marginpar{this proposition shouldn't \\ require convexity of $\source$ \\ although its corollary does}

 \begin{prop}[Continuous optimal maps onto closed $c$-convex target pieces]\label{prop: u is a maximum of C^1 potentials}
 In addition to the hypotheses of Lemma \ref{lem: u is a maximum of potentials} and \eqref{MTW}, assume that $\sourcemeas$ and $\targetmeas$ on $\M$ and $\Mbar$ respectively satisfy conditions (I) and (II) above, and for some $i\in I$ the compact set $\targetpiece[i]$ is strictly $c$-convex with respect to the compact set
$\outerdom$. Then the $c$-convex function $\convpot[i]$ from Lemma \ref{lem: u is a maximum of potentials} belongs to $C^1(\outerdom)$,   
 \begin{align}
 \csubdiff{\convpot[i]}{\outerdom}&\subset \targetpiece[i],\label{eqn: c-subdifferential pieces}
 \end{align}
and for any $x\in\source$ the intersection $\csubdiff{\convpot}{x}\cap \targetpiece[i]$ contains at most one point.
\end{prop}

%\marginpar{Why do we need $\#(\partial_c u(x) \cap \bar \Omega_i)=1$?}

\begin{proof}
 Since $\targetpiece[i]$ is $c$-convex with respect to $\outerdom$, combining 
 \cite[Lemma 5.1]{MaTrudingerWang05} with \eqref{eqn: gradient maps into pieces} yields
$\csubdiff{\convpot[i]}{\outerdom}\subset\targetpiece[i]$ to establish \eqref{eqn: c-subdifferential pieces}. 
 
 Next we show that each $\convpot[i]$ is $C^1$ on $\outerdom$. Indeed, note that $u^c$ is an optimal potential transporting $\targetmeas$ to $\sourcemeas$ with cost function $c^*$ defined on $\outertarget\times\outerdom$ by $c^*(\xbar, x):=c(x, \xbar)$, then by \cite[Theorem 2.1]{FigalliKimMcCann13} under \eqref{eqn: FKM conditions}, or \cite[Lemma 2.19, Theorem 1.2]{GuillenKitagawa15} under \eqref{eqn: GK conditions}, we have that 
$u^c \in C^{1,\bar \alpha}_{loc}$ for $\bar \alpha$ as described and strictly $c^*$-convex when restricted to each $\targetpiece[i]^{\interior}$. If there was a point $x$ where $\convpot[i]$ fails to be differentiable, by \cite[Theorem 3.1]{Loeper09} this implies the existence of some nontrivial line segment $\ell\subset\subdiff{\convpot[i]}{x}=\coord{\csubdiff{\convpot[i]}{x}}{x}\subset\coord{\targetpiece[i]}{x}$. However, by the strict convexity of $\coord{\targetpiece[i]}{x}$, this would imply that $\ell\cap \coord{\targetpiece[i]}{x}^{\interior}$ contains more than one point. It can be seen that this contradicts the strict ${c}^*$-convexity of $u^c$ on $\targetpiece[i]^{\interior}$, thus $\convpot[i]$ must be differentiable on $\outerdom$. The fact that the $c$-subdifferential of a $c$-convex function has a closed graph then implies $\convpot[i]\in C^1(\outerdom)$.

{Now if $x\in \source$ and $\csubdiff{\convpot}{x}\cap \targetpiece[i]$ contains more than one point, the same argument as the previous paragraph combined with the representation \eqref{eqn: maximum rep} again yields a contradiction.}  
\end{proof}

As a corollary to its proof we obtain the following interior homeomorphism result,
which can be upgraded to a diffeomorphism using results from the literature.

\begin{cor}[Optimal homeomorphisms onto open $c$-convex target pieces]\label{cor: interior homeo/diffeo}
Assume the same hypotheses as Proposition \ref{prop: u is a maximum of C^1 potentials}, but if condition \eqref{eqn: FKM conditions} is assumed from (II), additionally suppose that $\targetpiece[i]$ is strongly $c$-convex with respect to $\outerdom$. Then the map 
the map $T_i(x):=\cExp{x}{Du_i(x)}$ % from Proposition \ref{prop: u is a maximum of C^1 potentials}
is a homeomorphism from the interior of 
$\{ x\in \spt \mu \mid u(x)=u_i(x)\}$ to $\targetpiece[i]^{\interior}$;  its inverse is $C^{\bar \alpha}_{loc}$ for 
some $\bar \alpha>0$ depending only on $n$ and the bounds (I).
If the densities of $\mu$ and $\nu$ are locally Dini continuous (respectively $C^{k+\alpha}_{loc}$ for any
$0< k +\alpha \not\in \N$)
on the interiors of these two sets, then $T_i$ defines a diffeomorphism %between their interiors 
whose derivatives are locally Dini continuous (or $C^{k+\alpha}_{loc}$ respectively), at least if $a_3>0$
or $c(x,\bar x) = - \euclidean{x}{\xbar}$;
see \cite{LTW10} re $a_3=0$.
\end{cor}

\begin{proof}
The strict $c^*$-convexity of $u^c \in C^{1,\bar \alpha}_{loc}$ from the preceding proof
shows the map $S(\bar x) := \cstarExp{\bar x}{Du^c(\bar x)}$
restricted to $\targetpiece[i]^{\interior}$ is a homeomorphism (and $C^{\bar \alpha}_{loc}$).  
We assert %\\: CLAIM: 
this restriction has range $R^{\interior}$ where 
$R := \{x \in \spt \mu \mid u(x)=u_i(x)\}$, and its inverse is $T_i$.  

First note that $u^c(\xbar)=(u_i)^c(\xbar)$ for $\xbar\in \targetpiece[i]$. Indeed, $u_i \le u$ implies $(u_i)^c \ge u^c$ everywhere,  while for $\bar x \in \targetpiece[i]$ 
the opposite inequality is obtained by  taking $\bar y=\bar x$ in
$$
(u_i)^c(\bar x) = \sup_{x \in \Omega} [-c(x,\bar x) + \inf_{\bar y \in \bar \Omega_i} (c(x,\bar y) + u^c(\bar y))].
$$
Then, recall 
\begin{equation}\label{eqn: c-Young inequality}
u(x) + u^c(\bar x) + c(x,\bar x) \ge 0 \qquad {\rm for\ all}\ (x,\bar x) \in \Omega \times \bar \Omega,
\end{equation}
and equality holds if and only if $\bar x \in \csubdiff{u}{x}$ (or equivalently $x \in \cstarsubdiff{u^c}{\bar x})$.
For $\bar x \in \targetpiece[i]^{\interior}$,  we have $\cstarsubdiff{u^c}{\bar x} = \{S(\bar x)\}$ thus
\begin{align*}
 u(S(\xbar))&=-c(S(\xbar), \xbar)-u^c(\xbar)=-c(S(\xbar), \xbar)-(u_i)^c(\xbar)\\
 &=-c(S(\xbar), \xbar)+\inf_{y\in\outerdom}(c(y, \xbar)+u_i(y))\leq u_i(S(\xbar)).
\end{align*}
Since the reverse inequality always holds, we have $u(S(\xbar))=u_i(S(\xbar))$.
%\begin{align}
%-c(S(\bar x),\bar x) &= u(S(\bar x)) + u^c(\bar x) \nonumber 
%\\ &\ge u_i(S(\bar x)) + (u_i)^c(\bar x) %+ c(x,y) %\nonumber 
%%\\ & = u_i(S(y)) + (u_i)^c(y) + c(S(y),y),
%\label{eqn: backwards c-Young}
%\end{align}
%where the fact that 
% $(u_i)^c=u^c$ on $\targetpiece[i]$
%follows from \eqref{eqn: piecewise c-transform}:
%more precisely $u_i \le u$ implies $(u_i)^c \ge u^c$,  while for $\bar x \in \targetpiece[i]$ 
%the opposite inequality is obtained by  taking $\bar y=\bar x$ in
%$$
%(u_i)^c(\bar x) = \sup_{x \in \Omega} -c(x,\bar x) + \inf_{\bar y \in \bar \Omega_i} c(x,\bar y) + u^c(\bar y).
%$$
%Comparing \eqref{eqn: c-Young inequality} to \eqref{eqn: backwards c-Young} shows $u$ agrees with $u_i$ at 
%$S(\bar x)$.
Then as $S$ is injective and continuous, the set $S(\targetpiece[i]^{\interior})$ is open, hence it must be contained in $R^{\interior}$.  

We now claim that $T_i$ pushes the restriction of $\mu$ to $R^{\interior}$ forward to the restriction of $\nu$ to $\targetpiece[i]$. Let us write $T(x):=\cExp{x}{Du(x)}$, defined for $x\in \Dom(Du)$ so $T_\#\mu=\nu$. By Lemma \ref{lem: dm} and \eqref{eqn: gradient maps into pieces}, we see that $x\in \Dom(Du)$ with $T(x)\in \targetpiece[i]$ only if $u(x)=u_i(x)$ and $u(x)>u_j(x)$ for all $j\neq i$, in particular, $T^{-1}(\targetpiece[i])\subset R^{\interior}$. On the other hand, if $x\in R^{\interior}$, then $u=u_i$ on a neighborhood of $x$ and in particular, $u$ is differentiable at $x$. Hence by \eqref{B1} we must have $\csubdiff{u}{x}=\{T_i(x)\}=\{T(x)\}$ for all $x\in R^{\interior}$. Thus if $\bar{E}\subset \targetpiece[i]$ is measurable, we have
\begin{align*}
 \mu(R^{\interior}\cap T_i^{-1}(\bar{E}))=\mu(R^{\interior}\cap T^{-1}(\bar{E}))=\mu(T^{-1}(\bar{E}))=\nu(\bar{E})
\end{align*}
and the claim is proven.

Thus again using \cite[Theorem 2.1]{FigalliKimMcCann13} under \eqref{eqn: FKM conditions} (along with the assumption of strong $c$-convexity of $\targetpiece[i]$ with respect to $\outerdom$ in this case), and \cite[Lemma 2.19, Theorem 1.2]{GuillenKitagawa15} under \eqref{eqn: GK conditions} gives that $T_i$ is continuous and injective on $R^{\interior}$, hence 
%\eqref{eqn: gradient maps into pieces} and the continuity of $T_i$ shown in Proposition 
%\ref{prop: u is a maximum of C^1 potentials} yield 
$T_i(R^{\interior}) \subset \targetpiece[i]^{\interior}$.  

We complete the proof of the claim by showing $S \circ T_i = id_{R^{\interior}}$.
Since for each $x \in R^{\interior}$,  we have $\csubdiff{u}{x} = \{T_i(x)\} \subset \targetpiece[i]^{\interior}$, as argued
above this yields $\cstarsubdiff{u^c}{T_i(x)}) = \{S(T_i(x)))\}$.  The equality conditions in \eqref{eqn: c-Young inequality} then
force $x=S(T_i(x))$ as required.

%Since $T_i$ extends continuously
When \eqref{MTW} holds with $a_3>0$,  the local Dini or H\"older continuity asserted then follows from \cite{LTW10}, where it is also claimed that the results extend to $a_3=0$, although details of this 
extension are deferred to a forthcoming publication.  
For $c(x,\bar x) = - \euclidean{x}{\xbar}$, the details can be 
found in \cite{Caffarelli92}, \cite{Wang92}, and \cite{JianWang07}.
\end{proof}

Next we wish to make some finer observations on the structure of the boundaries of the sets above, and in particular the sets where more than two of the functions $u_i$ coincide. For this we need some notion of ``independence'' for subcollections of $\{\targetpiece[i]\}_{i\in I}$, which we call \emph{affine independence}.
Its role is to guarantee the natural implicit function theorem hypothesis is satisfied in the applications which follow.

%\marginpar{NO: eq to simultaneous separation by hyperplanes?}

\begin{defin}[Affine independence]\label{defin: affine independence of sets}
 A finite collection $\left\{\alttargetpiece[i]\right\}_{i=1}^{k}$ of $k\le n+1$ subsets of an $n$ dimensional vector space is said to be \emph{affinely independent} if no $k-2$ dimensional affine subspace intersects all of the sets in the collection.
(Equivalently, any collection of $k$ points, each from a different set $\alttargetpiece[i]$, is affinely independent in the usual sense.)
\end{defin}
We also define an alternate notion measuring the ``size'' of a singular point that we call the \emph{multiplicity}. Essentially the multiplicity of a singular point counts ``how many pieces of the target domain does a singular point get transported to?''

%\marginpar{Compact components?}

\begin{defin}[Multiplicity along tears]\label{defin: multiplicity}
Let $c$ be a cost function satisfying \eqref{B1} and $\sourcemeas$, $\targetmeas$ probability measures with $\sourcemeas$ absolutely continuous with respect to volume measure. Also suppose $\target\subset\outertarget$ is a disjoint union of some collection of sets $\left\{\targetpiece[i]\right\}_{i\in I}$ for some index set $I$ and $\convpot$ is an optimal potential of \eqref{OT} transporting $\sourcemeas$ to $\targetmeas$, with $x_0\in \source$. Then we define the \emph{multiplicity of $\convpot$ at $x_0$ 
relative to $\left\{\targetpiece[i]\right\}_{i\in I}$} by 
\begin{align*}
 \#\left\{i\in I\mid \targetpiece[i]\cap \csubdiff{\convpot}{x_0}\neq \emptyset\right\}.
\end{align*} 
When the collection $\left\{\targetpiece[i]\right\}_{i\in I}$ is clear, we will simply refer to the multiplicity of $\convpot$ at $x_0$.
\end{defin}
Finally, in order to simplify the statements and proofs of our results, we define notation for coincidence sets and multiplicity sets of the functions $u_i$ and $u$.
\begin{defin}[Tearing and coincidence sets]\label{defin: coincidence sets}
Let $c$ be a cost function satisfying \eqref{B1}. Also take compactly supported probability measures $\sourcemeas$ and $\targetmeas$ with $\sourcemeas$ absolutely continuous, and $\target = \cup_{i\in I} \overline \Omega_i$ a {\em finite disjoint} union of {\em compact} sets $\overline \Omega_i$. Then Lemma \ref{lem: u is a maximum of potentials} asserts
\begin{equation*}%\label{eq: u is a maximum of potentials}
 u=\sup_{i\in I}u_i \quad \mbox{\rm with}\quad \cExp{x}{Du_i(x)}\in \targetpiece[i],\ \forall\;x\in \Dom(D u).
\end{equation*}
For any subset $I'\subset I$ of indices, we then define the set
\begin{align}
 \sing[I']:&=\{x\in \outerdom\mid u_i(x)=u_j(x),\ \forall\;i,\ j\in I'\},\label{eqn: Sigma without up}\\
  \sing[I']^{\uparrow}:&=\{x\in \outerdom\mid u(x)=u_i(x),\ \forall\;i\in I'\}.
\label{eqn: Sigma up}
\end{align}
Also for any $k\in \Z_{\geq 0}$ we define
\begin{align}
 M_k:&=\{x\in\R^n\mid u\text{ has multiplicity exactly }k\text{ at }x\},\\
 M_{\geq k}:&=\{x\in\R^n\mid u\text{ has multiplicity at least }k\text{ at }x\},
\label{Multiplicity k}
 \end{align}
where $u$ is the optimal potential as in \eqref{eq: u is a maximum of potentials} and multiplicity here taken relative to the collection $\{\targetpiece[i]\}_{i\in I}$ in Definition \ref{defin: multiplicity}.
\end{defin}
Under a suitable assumption of affine independence, a quick application of the usual implicit function theorem yields the following corollary from Proposition \ref{prop: u is a maximum of C^1 potentials}.
\begin{cor}[Affine independence of convex targets yields $C^1$ smooth tears of each expected codimension]
\label{cor: convexity yields C^1 smoothness of tears}
%Fix $c(x,\bar x)=-\euclidean x{\bar x}$ and absolutely continuous probability measures $\mu$ and $\nu$ on $\R^n$ whose densities are bounded away from zero and infinity on their supports.
Assume $c$ is a cost function satisfying \eqref{B1} and \eqref{MTW}, $\mu$ and $\nu$ are probability measures on $M$ and $\Mbar$ respectively, and conditions (I) and (II) (before \eqref{eqn: FKM conditions}) hold. Let $\spt \nu = \cup_{i \in I}\targetpiece[i]$ be a finite disjoint union of compact sets,
%Let %$\nu = (Du)_\#\mu$ where 
and $u= \max u_i$ be from Lemma \ref{lem: u is a maximum of potentials}.
Finally suppose there is a collection of indices $i_1,\ldots, i_k\in I$ for which
$\{\coord{\targetpiece[i_1]}{x},\ldots, \coord{\targetpiece[i_k]}{x}\}$ forms an affinely independent collection
of strictly convex sets for every $x\in \sing[i_1,\ldots,i_k]$.  Then $\sing[i_1,\ldots,i_k]$ is a $C^1$ submanifold of $M$ having codimension $k-1$.
\end{cor}
\begin{proof}
Reordering if necessary,  we may assume $i_j=j$ for each $j\le k$.
The set $\sing[1,\ldots,k]$ then consists of the zero set of the system of $k-1$ equations
\begin{equation}\label{first k coincide}
u_1(x) = u_2(x) = \cdots = u_{k}(x),
\end{equation}
which are all contained in $C^1(\outerdom)$ by Proposition \ref{prop: u is a maximum of C^1 potentials}.
The implicit function theorem condition for the zero set of this system to be a $C^1$ submanifold
of the appropriate dimension is that the vectors $\{D u_j(x)- D u_k(x)\}_{j=1}^{k-1}$ be linearly independent
%(at each $x \in \sing[1,2,\ldots,k]$), 
when \eqref{first k coincide} holds,
which is equivalent to affine independence of  $\{D u_j(x)\}_{j=1}^k$.
But since $D u_i(x) \in \coord{\targetpiece[i]}{x}$ by \eqref{eqn: gradient maps into pieces} and \eqref{B1}, this follows from the affine independence of $\{\coord{\targetpiece[i]}{x}\}_{i=1}^k$.
\end{proof}
%\begin{defin}[Multiplicity sets]\label{defin: coincidence sets}
%Set $c(x, \xbar)=-\euclidean{x}{\xbar}$ on $\R^n\times \R^n$, and let 
%$\sourcemeas$ and $\targetmeas$ be probability measures 
%with compact support, where $\sourcemeas$ is absolutely continuous  and
% $\target = \cup_{i\in I} \overline \Omega_i$ is a {\em finite disjoint} union of {\em compact} sets $\overline \Omega_i$.
% 
%For any $k\in \Z_{\geq 0}$ we define the sets
%\begin{align}
% M_k:&=\{x\in\R^n\mid u\text{ has multiplicity exactly }k\text{ at }x\},\\
% M_{\geq k}:&=\{x\in\R^n\mid u\text{ has multiplicity at least }k\text{ at }x\},
%\label{Multiplicity k}
% \end{align}
%where $u$ is the optimal potential as in \eqref{eq: u is a maximum of potentials} and multiplicity here taken relative to the collection $\{\targetpiece[i]\}_{i\in I}$ in Definition \ref{defin: multiplicity}.
%\end{defin}
Next, we establish two elementary relationships between the sets $\sing[]^\uparrow$ and $M$. Specifically, we show that the closure $M_k^{cl}$ of all points with multiplicity lie in a union of tears; we later prove that when the disjoint components of $\spt \nu = \cup_{i\in I} \overline \Omega_i$
can be separated by hyperplanes pairwise \eqref{eqn: subdifferentials disjoint},
these tears lie in DC submanifolds.
%
%We show 
%that when the disjoint components of $\spt \nu = \cup_{i\in I} \overline \Omega_i$
%can be separated by hyperplanes pairwise \eqref{eqn: subdifferentials disjoint},
%then the closure $M_k^{cl}$ of all points with multiplicity lie in a union of tears, which we later show to be DC manifolds.

\begin{lem}[Covering multiplicity sets with tears]
%Assume $\overline \Omega_j$ disjoint from $ %\overline \Lambda_i := 
%\ch (\overline \Omega_i)$ for every pair $j \ne i$
%in Definition \ref{defin: coincidence sets}.  Then
%$$%\begin{align*}
% M_{\geq 2} \subset \bigcup_{i\neq j\in I}\sing[ij]^\uparrow \quad \mbox{\rm and} \quad
% M_k^{\cl} \subset M_{\geq k}.
%$$%\end{align*}
%
%NEW STATEMENT
%
Suppose that $c$ is a cost function satisfying \eqref{B1}, $\mu$ and $\nu$ are probability measures with $\mu$ absolutely continuous with respect to the volume measure, and $\target=\bigcup_{i\in I}\targetpiece[i]$ is a disjoint union of compact sets. Then multiplicity is upper semicontinuous:
\begin{align}\label{eqn: multiplicity closure}
  M_k^{\cl} \subset M_{\geq k}.
\end{align}
Additionally, fix a positive integer $k$ and suppose that for any collection of indices $I'\subset I$ with $\#(I')=k$ and $x\in \outerdom$
%\begin{align}
%\ch\left(\bigcup_{i\in I'}\coord{\targetpiece[i]}{x}\right)\cap \coord{\targetpiece[j]}{x}=\emptyset. \label{eqn: k+1 are independent}
%\end{align}
%$\{i_1,\ldots, i_{k+1}\}\subset I$ 
\begin{align}
 \left\{\coord{\targetpiece[i]}{x}\right\}_{i\in I'}\label{eqn: k are independent}
\end{align}
is affinely independent. Then
\begin{align}
 M_{\geq k} \subset \bigcup_{\{I'\subset I\mid \#(I')=k\}}\sing[I']^\uparrow.
\end{align}
\end{lem}

\begin{proof}
%Suppose $x \in \Omega \setminus \bigcup_{i\neq j\in I}\sing[ij]^\uparrow$.  
%From \eqref{eq: u is a maximum of potentials} we infer $x \in \sing[i]^\uparrow$ for some $i$,
%and $\p u_i(x) \subset \ch (\overline \Omega_i)$.  Now Lemma \ref{lem: dm} 
%implies $x$ has multiplicity 1,  to establish the first inclusion.
%
%Now suppose $x_0\in M_k^{\cl}$, so there is a sequence $\{x_m\}_{m=1}^\infty\subset M_k$ converging to $x_0$. We may pass to a subsequence and assume, without loss of generality, that each $\subdiff{u}{x_m}$ only intersects $\targetpiece[1],\ldots,\targetpiece[k]$ out of the collection $\{\targetpiece[i]\}_{i\in I}$, let $\xbar_{i, m}\in \subdiff{u}{x_m}\cap\targetpiece[i]$ for $i\in \{1,\ldots, k\}$. Since each $\targetpiece[i]$ is compact, we may pass to subsequences to assume each $\xbar_{i, m}$ converges as $m\to\infty$ to some $\xbar_i\in \targetpiece[i]$, by upper semicontinuity of the subdifferential we see that $\xbar_i\in\subdiff{u}{x_0}$, meaning $x_0\in M_{\geq k}$. 
%
%NEW PROOF
%
Suppose $x_0\in M_k^{\cl}$, so there is a sequence $\{x_m\}_{m=1}^\infty\subset M_k$ converging to $x_0$. We may pass to a subsequence and assume, without loss of generality, that each $\csubdiff{u}{x_m}$ only intersects $\targetpiece[1],\ldots,\targetpiece[k]$ out of the collection $\{\targetpiece[i]\}_{i\in I}$, and take $\xbar_{i, m}\in \csubdiff{u}{x_m}\cap\targetpiece[i]$ for $i\in \{1,\ldots, k\}$. Since each $\targetpiece[i]$ is compact, we may pass to further subsequences to assume each $\xbar_{i, m}$ converges as $m\to\infty$ to some $\xbar_i\in \targetpiece[i]$, and by upper semicontinuity of the $c$-subdifferential we see that $\xbar_i\in\csubdiff{u}{x_0}$, meaning $x_0\in M_{\geq k}$. 

Now assume \eqref{eqn: k are independent} holds and take $x_0 \in \Omega \setminus \bigcup_{\{I'\subset I \mid \# (I')=k\}}\sing[I']^\uparrow$. If $\#(I)<k$, then clearly $x_0\not\in M_{\geq k}$, thus assume $\#(I)\geq k$. 
From \eqref{eq: u is a maximum of potentials} it is clear that $u(x_0)=u_i(x_0)$ for at least one index $i$, and this can only hold for at most $k'\leq k-1$ distinct indices; suppose we have $u(x_0)=u_{i_j}(x_0)$ for $1\leq j\leq k'$ and strict inequality for all other indices. Then by Lemma \ref{lem: dm} and \eqref{eqn: gradient maps into pieces}
\begin{align*}
 \coord{\csubdiff{u}{x_0}}{x_0}&\subset \subdiff{u}{x_0}\subset \ch\left(\bigcup_{1\leq j\leq k'}\ch\left(\coord{\targetpiece[i_j]}{x_0}\right)\right)= \ch\left(\bigcup_{1\leq j\leq k'}\coord{\targetpiece[i_j]}{x_0}\right).
\end{align*}
Thus if the multiplicity of $u$ at $x_0$ is $k$ or greater, there exists an index $i'\not\in \{i_1,\ldots, i_{k'}\}$ for which $\csubdiff{u}{x_0}\cap\targetpiece[i']\neq \emptyset$, by the above inclusion this implies there is a point in $\coord{\targetpiece[i']}{x_0}$ which can be written as the convex combination of $k'$ points, one from each of the sets $\{\coord{\targetpiece[i_1]}{x_0},\ldots, \coord{\targetpiece[i_k']}{x_0}\}$. Since $k'\leq k-1$ and $\#(I)\geq k$, we can complete $\{i_1,\ldots, i_k',\ i'\}$ to a subset of $I$ with cardinality $k$ to obtain a contradiction with \eqref{eqn: k are independent}, hence $x_0\not\in M_{\geq k}$.

%Now suppose $x_0\in M_2^{\cl}$, so there is a sequence $\{x_m\}_{m=1}^\infty\subset M_2$ converging to $x_0$. We may pass to a subsequence and assume, without loss of generality, that each $\subdiff{u}{x_m}$ only intersects $\targetpiece[1]$ and $\targetpiece[2]$ out of the collection $\{\targetpiece[i]\}_{i\in I}$, let $\xbar_{i, m}$, $\xbar_{j, m}$ be the intersections. Since $\targetpiece[1]$ and $\targetpiece[2]$ are compact, we may pass to a subsequence and assume $\xbar_{i, m}$ and $\xbar_{j, m}$ converge as $m\to\infty$, by upper semicontinuity of the subdifferential we see that $\subdiff{u}{x_0}$ must intersect at least $\targetpiece[1]$ and $\targetpiece[2]$, hence $x_0\in M_{\geq 2}$. 
\end{proof}

\section{Global structure of optimal map discontinuities: quadratic cost}
\label{section: global quadratic}

We state the results of this section in the model case $c(x, \xbar)=-\euclidean{x}{\xbar}$ on $\R^n\times \R^n$, where the proofs are much simpler and the geometric picture easier to understand. It is easily verified that this cost function satisfies \eqref{B1} and \eqref{MTW}, both $\cExp{x}{\cdot}$ and $\cstarExp{\xbar}{\cdot}$ are the identity mapping for any $x$ and $\xbar$, and $c$- and $c^*$-convexity of sets reduces to the usual convexity of a set.

Our first result is the following proposition which --- apart from its final sentence --- follows rapidly
from our explicit function theorem. It will be extended to MTW costs in a subsequent section.

\begin{prop}[Hyperplane separated components induce DC tears]\label{prop: euclidean case interface}
 Let $c(x, \xbar)=-\euclidean{x}{\xbar}$. Also suppose $\sourcemeas$ and $\targetmeas$ are absolutely continuous probability measures with bounded supports, and $\target=\targetpiece[1]\cup \targetpiece[2]$ is such that $\targetpiece[1]$ and $\targetpiece[2]$ are strongly separated by some hyperplane $\Pi$.
 
 Then an optimal potential $u$ transporting $\sourcemeas$ to $\targetmeas$ can be written $u=\max\{u_1, u_2\}$, where $u_1$ and $u_2$ are convex functions, finite on $\R^n$ such that 
\begin{align}\label{eqn: euclidean mapping destination}
 \nabla u_i(x)&\in \targetpiece[i],\quad \forall\; x\in\Dom(\nabla u).
\end{align}
Moreover, the sets 
\begin{align*}
 \sing[]:&=\left\{x\in\R^n\mid \subdiff{u}{x}\cap\targetpiece[i]\neq\emptyset,\ i=1, 2\right\}=\{x\in\R^n\mid u_1(x)=u_2(x)\},\\
 C_1:&=\left\{x\in \R^n\mid \subdiff{u}{x}\cap\targetpiece[2]=\emptyset\right\}=\{x\in\R^n\mid u_1(x)>u_2(x)\},\\
 C_2:&=\left\{x\in\R^n\mid \subdiff{u}{x}\cap\targetpiece[1]=\emptyset\right\}=\{x\in\R^n\mid u_1(x)<u_2(x)\}.
\end{align*}
 are connected and given by the graph, open epigraph, and open subgraph respectively of a globally Lipschitz DC function $h$ defined on the hyperplane $\Pi$.
 
 If $\source$ is convex and $\targetpiece[i]$ is connected for either $i=1$ or $2$, then $\source\cap (C_i\cup \sing[])$ is also connected.
\end{prop}
\begin{proof}
 Let us assume $\Pi=\{x\in\R^n\mid x^n=0\}=\R^{n-1}$. By Lemma \ref{lem: u is a maximum of potentials} (note we do not necessarily need boundedness of $\outerdom$, see Remark \ref{rmk: no compactness needed}) %only compactness of $\source$ to say that $u^*=u^c$ is well-defined everywhere) 
 we find that $u=\max\{u_1, u_2\}$, both $u_i$ are convex and finite on $\R^n$, and we have \eqref{eqn: euclidean mapping destination}. Since $\targetpiece[1]$ and $\targetpiece[2]$ are strongly separated by $\R^{n-1}$, so are their convex hulls, and \eqref{eqn: euclidean mapping destination} implies $\subdiff{u_i}{\R^n}\subset \ch(\targetpiece[i])$. Thus we can apply Corollary \ref{thm: explicit function theorem} to obtain the function $h$ defined on $\R^{n-1}$ along with all claimed properties above; the connectedness from continuity of $h$.

%\marginpar{Whoops! \ref{cor: explicit function theorem} requires convexity of $\Omega$}
 
 Now assume $\source$ is convex and $\targetpiece[1]$ is connected. Let $d(x):=d(x, \source)^2$ which is finite and convex on $\R^n$, and define $\tilde{u}:=u+d$. An easy calculation gives
\begin{align*}
 \subdiff{d}{x}=
\begin{cases}
 \{0\},&x\in \source,\\
 \dist(x, \source)\frac{x-\pi_{\source}(x)}{\norm{x-\pi_{\source}(x)}},&x\not\in \source,
\end{cases}
\end{align*}
where $\pi_{\source}(x)$ is the (unique) closest point projection of $x$ onto $\source$. Thus we see by \cite[Theorem 23.8]{Rockafellar70} that 
\begin{align}
\subdiff{\tilde{u}}{x}=\subdiff{u}{x},\quad \forall\;x\in \source.\label{eqn: tilde and original subdifferentials equal}
\end{align}

Next we will show that $\subdiff{\tilde{u}^*}{\xbar}\subset \source$ for every $\xbar\in \targetpiece[1]$ (this is a nontrivial claim for $\xbar\in \targetpiece[1]^\partial$). By \cite[Theorem 16.4]{Rockafellar70} we have
\begin{align}\label{eqn: u tilde transform}
 \tilde{u}^*(\xbar)=\inf_{\ybar\in \R^n}(u^*(\xbar-\ybar)+d^*(\ybar)),
\end{align}
we will calculate $d^*(\ybar)$. Let us write $h(\ybar):=\sup_{x\in\source}\euclidean{x}{\ybar}$ for the \emph{support function} of $\source$, since $\source$ is compact, for each $\ybar\in \R^n$ there exists $z(\ybar)\in\source$ such that $h(\ybar)=\euclidean{z(\ybar)}{\ybar}$. Clearly $d^*(0)=0$, so assume $\ybar\neq 0$. Then by definition,
\begin{align*}
 d^*(\ybar)&=\sup_{x\in \R^n}(\euclidean{x}{\ybar}-\dist(x, \source)^2)=\sup_{\{x\in \R^n\mid \euclidean{x}{\ybar}>\euclidean{z(\ybar)}{\ybar}\}}(\euclidean{x}{\ybar}-\dist(x, \source)^2).
\end{align*}
Fix any $x$ such that $\euclidean{x}{\ybar}>\euclidean{z(\ybar)}{\ybar}$, and an arbitrary $y\in\source$, then for some $\lambda\in [0, 1)$ we have $x_\lambda:=(1-\lambda)y+\lambda x$ satisfies $\euclidean{x_\lambda}{\ybar}=\euclidean{z(\ybar)}{\ybar}$. Then we calculate
\begin{align*}
 \norm{x-y}&\geq \norm{x-x_\lambda}\geq \euclidean{x-x_\lambda}{\frac{\ybar}{\norm{\ybar}}}=\euclidean{x-z(\ybar)}{\frac{\ybar}{\norm{\ybar}}},
\end{align*}
hence taking an infimum over $y\in \source$,
\begin{align*}
 \euclidean{x}{\ybar}-\dist(x, \source)^2&%\geq\euclidean{x}{\ybar}- \euclidean{x-z(\ybar)}{\frac{\ybar}{\norm{\ybar}}}^2=
 \geq h(\ybar)+\euclidean{x-z(\ybar)}{\ybar}-\frac{\euclidean{x-z(\ybar)}{\ybar}^2}{\norm{\ybar}^2}.%\\
% &=\frac{\euclidean{\norm{\ybar}^2x}{\ybar}-\euclidean{x-z(\ybar)}{\ybar}^2}{\norm{\ybar}^2}\\
% &=\frac{\euclidean{\norm{\ybar}^2x}{\ybar}-\euclidean{x-z(\ybar)}{\ybar}^2}{\norm{\ybar}^2}
\end{align*}
This last quantity can be seen to be maximized over $\euclidean{x}{\ybar}>\euclidean{z(\ybar)}{\ybar}$ when $\euclidean{x-z(\ybar)}{\ybar}=\frac{\norm{\ybar}^2}{2}$, yielding
\begin{align*}
%t-\frac{t^2}{\norm{\ybar}^2}=f, f'(t)=1-\frac{2t}{\norm{\ybar}^2}=0\implies t=\frac{\norm{\ybar}^2}{2}\\
 d^*(\ybar)=h(\ybar)+\frac{\norm{\ybar}^2}{2}-\frac{\norm{\ybar}^2}{4}=h(\ybar)+\frac{\norm{\ybar}^2}{4}.
\end{align*}
By choosing $\ybar=0$ in \eqref{eqn: u tilde transform}, for any $\xbar\in \R^n$ we clearly have
\begin{align*}
 \tilde{u}^*(\xbar)\leq u^*(\xbar).
\end{align*}
On the other hand, suppose $\xbar_0\in \targetpiece[1]^{\interior}$. By \cite[Theorem 2.12]{Villani03} $u^*$ is an optimal potential transporting $\targetmeas$ to $\sourcemeas$, then by \cite[Theorem 10.28]{Villani09} and convexity of $\source$, we have that $\subdiff{u^*}{\xbar_0}\in \source$, let $x_0\in\subdiff{u^*}{\xbar_0}$. Then for any $\ybar\in \R^n$,
\begin{align*}
 u^*(\xbar_0-\ybar)+h(\ybar)+\frac{\norm{\ybar}^2}{4}\geq  u^*(\xbar_0)+\euclidean{\xbar_0-\ybar-\xbar_0}{x_0}+\euclidean{\ybar}{x_0}=u^*(x_0),
\end{align*}
thus taking an infimum over $\ybar\in \R^n$ and recalling \eqref{eqn: u tilde transform} gives $\tilde{u}^*\geq u^*$ on $\targetpiece[1]^{\interior}$. 
Since the Legendre transform of a convex function is always closed, we then have $\tilde{u}^*\equiv u^*$ on all of $\targetpiece[1]=\targetpiece[1]^{\cl}$. Now let $\xbar_0\in\targetpiece[1]$ and suppose $x_0\in \subdiff{\tilde{u}}{\xbar_0}$. Then for any $\xbar$, $\ybar\in \R^n$, again using \eqref{eqn: u tilde transform},
\begin{align*}
 u^*(\xbar-\ybar)+h(\ybar)+\frac{\norm{\ybar}^2}{4}&\geq \tilde{u}^*(\xbar)\geq \tilde{u}^*(\xbar_0)+\euclidean{\xbar-\xbar_0}{x_0}\\
 &=u^*(\xbar_0)+\euclidean{\xbar-\xbar_0}{x_0}.
\end{align*}
We can let $\ybar$ vary over $\R^n\setminus \{0\}$ while setting $\xbar=\ybar+\xbar_0$ in the equation above, then dividing through by $\norm{\ybar}$ we find
\begin{align*}
 \sup_{x\in \source}\euclidean{x}{\frac{\ybar}{\norm{\ybar}}}+\frac{\norm{\ybar}}{4}\geq \euclidean{x_0}{\frac{\ybar}{\norm{\ybar}}},
\end{align*}
taking $\ybar\to 0$ radially gives
\begin{align*}
 \sup_{x\in \source}\euclidean{x}{\omega}\geq \euclidean{x_0}{\omega},\quad \forall \omega\in \S^{n-1},
\end{align*}
hence we must have $x_0\in \source$ as claimed.

 We now claim that
\begin{align}\label{eqn: support intersect superlevelset}
 \subdiff{\tilde{u}^*}{\targetpiece[1]}=\source\cap (C_1\cup \sing[]),
\end{align}
then the proof will be complete by applying Lemma \ref{lem: connectivity of c-subdifferential}. 
Suppose $x_0\in \source\cap (C_1\cup \sing[])$. Recall by \eqref{eqn: tilde and original subdifferentials equal}, $\subdiff{u}{x_0}= \subdiff{\tilde{u}}{x_0}$.
%
%Since $u\equiv \tilde{u}$ on $\source$ with $\tilde{u}\geq u$ everywhere, we must have $\subdiff{u}{x_0}\subset \subdiff{\tilde{u}}{x_0}$.
 There are two possibilities, either $u_1(x_0)>u_2(x_0)$, or $u_1(x_0)=u_2(x_0)$. In the first case, $\subdiff{u}{x_0}=\subdiff{u_1}{x_0}$, while in the second case, by Lemma \ref{lem: dm} we have $\subdiff{u}{x_0}=\ch(\subdiff{u_1}{x_0}\cup \subdiff{u_2}{x_0})$. In either case, since $\subdiff{u_1}{x_0}\cap \targetpiece[1]\neq \emptyset$ by \eqref{eqn: euclidean mapping destination}, there exists $y_0\in \targetpiece[1]$ such that $y_0\in \subdiff{\tilde{u}}{x_0}$. Hence $x_0\in \subdiff{\tilde{u}^*}{y_0}\subset \subdiff{\tilde{u}^*}{\targetpiece[1]}$. %Thus $\source\cap (C_1\cup \sing[])\subset \subdiff{\tilde{u}^*}{\ch \targetpiece[1]}$ and likewise, $\source\cap (C_2\cup \sing[])\subset \subdiff{\tilde{u}^*}{\ch \targetpiece[2]}$ 

Now suppose $x_0\in  \subdiff{\tilde{u}^*}{\targetpiece[1]}$ but $u_2(x_0)>u_1(x_0)$. As we have shown above, $x_0\in \source$. Then by \eqref{eqn: tilde and original subdifferentials equal} combined with Lemma \ref{lem: dm}, $\subdiff{\tilde{u}}{x_0}=\subdiff{u}{x_0}=\subdiff{u_2}{x_0}\subset \ch(\targetpiece[2])$. However this is a contradiction, as this gives $\subdiff{\tilde{u}}{x_0}\cap \targetpiece[1]=\emptyset$. This concludes the proof of \eqref{eqn: support intersect superlevelset}.
%
%By \cite[Theorem 10.28]{Villani09} again and convexity of $\source$, along with the fact that $\tilde{u}$ is only finite on $\source$, we immediately find that $\subdiff{\tilde{u}^*}{\ch \targetpiece[1]}\subset\source$, i.e. $x_0\in\source$. 
%
% then since $\source$ is convex, we can find a sequence of points $x_k\in \source^{\interior}$ converging to $x_0$, such that $u_2(x_k)>u_1(x_k)$. Then since $u=\tilde{u}$ in a neighborhood of each $x_k$ we have $\subdiff{\tilde{u}}{x_k}=\subdiff{u}{x_k}$, while by Lemma \ref{lem: dm} we see $\subdiff{u}{x_k}=\subdiff{u_2}{x_k}\subset \ch\targetpiece[2]$. In particular, this means $\subdiff{\tilde{u}}{x_k}\cap \ch\targetpiece[1]=\emptyset$ for all $x_k$. Now we can find $y_k\in \subdiff{\tilde{u}}{x_k}$, possibly passing to a subsequence assume $y_k\to y_0$
%
%
\end{proof}

We can also obtain some structure in the case where $\target$ consists of more than two regions
separated by hyperplanes. Before we state the results, some setup.

Again, we restrict the discussion to the bilinear cost
$c(x, \xbar)=-\euclidean{x}{\xbar}$ on $\R^n\times \R^n$,
while $\sourcemeas$ and $\targetmeas$ are absolutely continuous probability measures with bounded supports. 
% Let $\{\Pi_j\}_{j\in J}$ be a collection of $n-1$ dimensional hyperplanes in $\R^n$
% {\em whose union is disjoint from $\target$???}, $\{\alttargetpiece[i]\}_{i\in I}$ be the collection of connected %components of $\R^n\setminus \left(\bigcup_{j\in J} \Pi_j\right)$ which intersect $\target$, and write
We'll assume $\spt \nu = \cup_{i \in I} \targetpiece[i]$ is a decomposition into finitely many compact disjoint sets;
i.e.  henceforth we assume that $I$ is \emph{finite}.  
Then if $u$ is an optimal potential transporting $\sourcemeas$ to $\targetmeas$, 
by Lemma \ref{lem: u is a maximum of potentials} there exist convex functions $u_i$, $i\in I$ on $\R^n$ such that 
\begin{equation}\label{eq: u is a maximum of potentials}
 u=\sup_{i\in I}u_i \quad \mbox{\rm with}\quad \nabla u_i(x)\in \targetpiece[i],\ \forall\;x\in \Dom(\nabla u).
\end{equation}

If some $\targetpiece[i]$ is strictly convex, $\source$ is convex, and the densities of $\sourcemeas$ and $\targetmeas$ are bounded away from zero and infinity on their supports, by Proposition \ref{prop: u is a maximum of C^1 potentials} we have $u_i\in C^1(\R^n)$.  
 We'll often require that each
 $\targetpiece[i]$ can be strongly separated from each $\targetpiece[j]$ by a hyperplane,
so that their convex hulls are disjoint:
on is that the sets $\ch(\targetpiece[i])$ are mutually disjoint, hence
\begin{align}\label{eqn: subdifferentials disjoint}
 \subdiff{u_i}{\R^n}\subset \ch(\targetpiece[i])\text{ are mutually disjoint}.
\end{align}

We begin with two corollaries of Theorem \ref{thm: explicit function theorem} (the sets $\sing[I']$ and $\sing[I']^\uparrow$ below for a collection of indices $I'$ are defined by \eqref{eqn: Sigma without up} and \eqref{eqn: Sigma up} respectively):
%Proposition \ref{prop: euclidean case interface}: 

\begin{cor} [DC rectifiability of $\Sigma_{ij}$]\label{cor:pairwise DC}
If $\overline \Omega_i$ and $\overline \Omega_j$ can be strongly separated by a hyperplane $\Pi$
 for some $i \ne j$ in Definition \ref{defin: coincidence sets},
then $\sing[ij] := \sing[\{i,j\}]$ is a globally Lipschitz DC graph over $\Pi$.
% as in Lemma \ref{prop: euclidean case interface}
\end{cor}

\begin{proof}
The convex hull of $\overline \Omega_i$ contains $\p u_i(\R^n)$ and is strongly separated from $\p u_j(\R^n) \subset \ch(\overline \Omega_j)$ by $\Pi$. The claim therefore follows from 
Theorem \ref{thm: explicit function theorem}.
\end{proof}

For the quadratic cost,  this result allows us to deduce a variant of
Proposition~\ref{prop: u is a maximum of C^1 potentials} 
%/ Corollary \ref{cor: interior homeo/diffeo}
which requires neither convexity of $\spt \mu$ nor {\em strict} convexity of $\targetpiece[1]$:

\begin{cor}[Continuous optimal maps to convex target pieces]
\label{cor: continuous maps to convex target pieces}
Fix $c(x,\bar x)=-\euclidean x{\bar x}$ and absolutely continuous probability measures $\mu$ and $\nu$ on $\R^n$ whose densities are bounded away from zero and infinity on their (compact) supports.
Let %$\nu = (Du)_\#\mu$ where 
$u= \max u_i$ be from Lemma \ref{lem: u is a maximum of potentials}.
%\eqref{eq: u is a maximum of potentials}. 
Assume $\targetpiece[1]$ is convex,  and disjoint from $\ch(\targetpiece[i])$ for each $i>1$
such that $\sing[i]^\uparrow$ intersects  $\Omega_1 := (\source) \cap \sing[1]^\uparrow$.
If, in addition $(\source)^\partial \cap \sing[1]^\uparrow$
has zero volume, then $Du_1 \in C^{\alpha}_{loc}(\Omega_1^{\interior})$ and is injective on $\Omega_1^{\interior}$.
\end{cor}

\begin{proof}
The boundary of $\Omega_1$ is contained in the union of those
 $\sing[1,i]^\uparrow$ intersecting $\source$ and $(\source)^\partial \cap \sing[1]^\uparrow$.
%Theorem \ref{thm: higher codimension euclidean} 
Corollary \ref{cor:pairwise DC} shows the former are DC hypersurfaces,
hence contain zero volume, like the latter.  Caffarelli's results \cite{Caffarelli92} now assert 
$u_1 \in C^{1,\alpha}_{loc}(\Omega_1^{\interior})$ and is strictly convex there.
%The Legendre transform of $u_1$ must therefore be strictly convex on the interior of $\targetpiece[1]$,
%and in fact up to the boundary by the strict convexity assumed of $\targetpiece[1]$.  
%Since $u_1^*(y) = \infty$ outside $\targetpiece[1]$, this shows $u_1 \in C^1(\R^n)$.
\end{proof}

In the above corollary, $D u_1$ gives a homeomorphism between the interior of $\Omega_1:= (\source) \cap \sing[1]^\uparrow$ and some open subset $V_1:=Du_1(\Omega_1^{\interior})$ of full volume in $\targetpiece[1]$;  however, the price
we pay for the lack of convexity of $\spt \mu$ is that we can no longer conclude differentiability of $u_1$ up to the boundary of $\Omega_1$ because we cannot preclude the possibility that $u^*$ fails to be strictly convex along a segment in $\targetpiece[1]\setminus V_1$.

The next theorem shows that $\sing[i_1,\ldots,i_k]^{\uparrow}$ is a disjoint union of $\sing[i_1,\ldots,i_k]^\uparrow\cap M_k$ and $\bigcup_{j\in I\setminus\{i_1,\ldots, i_k\}}\sing[i_1,\ldots,i_k, j]^\uparrow$: the first being a
% relatively open subset with positive $\mathcal{H}^{n+1-k}$ measure, 
DC submanifold of codimension $k-1$,
the second a finite union of closed sets with Hausdorff dimension at most $n-k$.
%zero $\mathcal{H}^{n+1-k}$ measure.
For implications of affine independence in a simpler setting,
%with strictly convex target components 
see the $C^1$ description of higher
codimension tears coming from strictly convex target components in 
Corollary \ref{cor: convexity yields C^1 smoothness of tears}.

\begin{thm}[DC rectifiability of  higher multiplicity tears]\label{thm: higher codimension euclidean}
Fix $c(x,\bar x)=-\euclidean x{\bar x}$ and probability measures $\mu$ and $\nu$ on $\R^n$ with $\mu$ absolutely continuous and %compact support and
$\spt \nu = \cup_{i \in I}\targetpiece[i]$ a finite disjoint union of compact sets.
Let $\nu = (Du)_\#\mu$ where $u$ from \eqref{eq: u is a maximum of potentials} is convex. 
Fix a collection of indices $i_1,\ldots, i_k\in I$.
If $\{\ch(\targetpiece[i_1]),\ldots, \ch(\targetpiece[i_k])\}$ is an affinely independent collection, for any $x_0\in\sing[i_1,\ldots,i_k]$ there exists $r_0>0$ such that $B_{r_0}(x_0)\cap \sing[i_1,\ldots,i_k]$ is contained in the image of an open subset of $\R^{n+1-k}$ under a bi-Lipschitz DC mapping.

%$\{\targetpiece[i]\}_{i\in I}$ satisfies the following condition:
Suppose in addition that %$\sing[i_1,\ldots,i_k, j]\neq \emptyset$ 
the existence of a point $x$ such that $\subdiff{u}{x}\cap \targetpiece[i]\neq \emptyset$ for all of $i=i_1, \ldots, i_k$, and $j$ implies
 \begin{align}
& \{\ch(\targetpiece[i_1]),\ldots, \ch(\targetpiece[i_k]),\ \ch(\targetpiece[j])\}\text{ is an affinely independent collection.}\notag\\
%&\quad\text{for any }j \text{ s.t. } \sing[i_1,\ldots,i_k, j]\neq \emptyset,
\label{eqn: one more is independent}
 \end{align}
Then
\begin{align}
\sing[i_1,\ldots,i_k]^\uparrow\cap M_k&=\{x\in\R^n\mid u(x)=u_{i_1}(x)=\ldots=u_{i_k}(x)>\max_{j\in I\setminus\{i_1,\ldots, i_k\}}u_j(x)\},\label{eqn: M_k means only touches k}\\
(\sing[i_1,\ldots,i_k]^\uparrow\cap M_k)&\cap \bigcup_{j\in I\setminus\{i_1,\ldots, i_k\}}\sing[i_1,\ldots,i_k, j]^\uparrow=\emptyset,\label{eqn: disjointness of multiplicity sets}\\
(\sing[i_1,\ldots,i_k]^\uparrow\cap M_k)&\cup \bigcup_{j\in I\setminus\{i_1,\ldots, i_k\}}\sing[i_1,\ldots,i_k, j]^\uparrow=\sing[i_1,\ldots,i_k]^{\uparrow}.\label{eqn: closure of multiplicity sets}
\end{align}
Moreover $\sing[i_1,\ldots,i_k]^\uparrow\cap M_k$ is a relatively open subset of $\sing[i_1,\ldots,i_k]^{\uparrow}$.
\end{thm}

\begin{proof}
Without loss of generality, we may assume $i_1=1,\ldots, i_k=k$.

First assume $\{\ch(\targetpiece[i])\}_{i=1}^k$ is an affinely independent collection and $x_0\in\sing[1,\ldots,k]$. Defining $F: \R^{n+1-k}\times \R^{k-1}\to \R^{k-1}$ by
\begin{align*}
 F(x):=(u_1(x)-u_k(x),\ldots, u_{k-1}(x)-u_k(x)),
\end{align*}
by assumption $F(x_0)=0$, we will now show that every element of $J^CF(x_0)$ has rank $k-1$. Let $M\in J^CF(x_0)$, and  suppose the $i$th row is given by a vector of the form 
\begin{align*}
v_i:= \lim_{m\to\infty}\nabla(u_i-u_k)(x_m)
\end{align*}
with $x_m\to x_0$ and $x_m\in \Dom(\nabla u_i)\cap \Dom(\nabla u_k)$. Then there must exist points $\xbar_i\in \targetpiece[i]$ for $i\in \{1,\ldots, k\}$ such that $v_i=\xbar_i-\xbar_k$, and the assumption of affine independence implies $M$ has rank $k-1$. By Carath{\'e}odory's theorem (\cite[Theorem 17.1] {Rockafellar70} any other $M\in J^CF(x_0)$ can be written as the convex combination of $n+1$ matrices as above, meaning that we have $v_i=\xbar_i-\xbar_k$ this time with $\xbar_i\in \ch(\targetpiece[i])$ for $i\in \{1,\ldots, k\}$, again the hypothesis yields that $M$ has rank $k-1$. Thus we can apply the DC constant rank theorem (Theorem \ref{thm: DC constant rank}) to obtain the first claim.

Now assume condition \eqref{eqn: one more is independent} holds. For brevity, let us notate the set on the right hand side of \eqref{eqn: M_k means only touches k} by $S_k$. Suppose $u(x_0)=u_i(x_0)$ for any fixed index $i\in I$, then by Lemma \ref{lem: dm} we have $\subdiff{u_i}{x_0}\subset \subdiff{u}{x_0}$. Any extremal point of $\subdiff{u_i}{x_0}$ is a limit of points of the form $\nabla u_i(x_m)$ where $x_m\in \Dom(\nabla u_i)$ and $x_m\to x_0$, then since $\nabla u_i(\Dom(\nabla u))\subset \targetpiece[i]$ which is a closed set, we see $\subdiff{u}{x_0}\cap \targetpiece[i]\neq \emptyset$. Thus, we immediately see $\sing[1,\ldots,k]^\uparrow\cap M_k\subset S_k$. On the other hand suppose $x_0\in S_k$, then by definition $x_0\in \sing[1,\ldots,k]^\uparrow$. Suppose by contradiction $x_0\not\in M_k$, then there must exist $j\in I\setminus \{1, \ldots, k\}$ such that $\exists \xbar_0\in\subdiff{u}{x_0}\cap\targetpiece[j]$. Since $\subdiff{u}{x_0}\cap\targetpiece[i]\neq \emptyset$ for $i\in \{1,\ldots, k\}$ by Lemma \ref{lem: dm}, \eqref{eqn: one more is independent} implies the collection 
\begin{align*}
 \{\ch(\targetpiece[1]),\ldots, \ch(\targetpiece[k]),\ \ch(\targetpiece[j])\}
\end{align*}
is affinely independent. However, by Lemma \ref{lem: dm} and the definition of $S_k$, we must have that $\xbar_0$ is contained in the convex hull of $k$ points, one from each of $\{\ch(\targetpiece[1]),\ldots, \ch(\targetpiece[k])\}$ contradicting this affine independence, 
%
%By \cite[Theorem 10.28]{Villani09}, there exist $x_m\in \Dom(\nabla u)\cap \source$ such that $\lim_{m\to\infty}\nabla u(x_m)=\xbar_0$. Since the sets $\targetpiece[i]$ are compact and mutually disjoint, passing to a subsequence we may assume $\nabla u(x_m)\in \targetpiece[l]$ for all $m$. Then by \eqref{eqn: subdifferentials disjoint} and Lemma \ref{lem: dm}, we must have $u(x_m)=u_j(x_m)$. Since $\source$ is compact, we can pass to another subsequence to assume $x_m$ converges to some $x_\infty$, continuity of $u$ and the $u_i$ imply that $x_\infty\in \sing[1,\ldots, k, j]$. However, another application of Lemma \ref{lem: dm} to $\subdiff{u}{x_0}$ shows the existence of $\xbar_0$ is in contradiction with \eqref{eqn: one more is independent}, 
proving \eqref{eqn: M_k means only touches k}. The claim \eqref{eqn: disjointness of multiplicity sets} then follows immediately.

Next, by continuity of the $u_i$ and $u$ we immediately see
\begin{align*}
\sing[1,\ldots,k]^{\uparrow}\subset (\sing[1,\ldots,k]^\uparrow\cap M_k)\cup \bigcup_{j \in I\setminus \{1,\ldots, k\}}\sing[1,\ldots,k, j]^\uparrow,
\end{align*}
while by \eqref{eqn: M_k means only touches k} the opposite inclusion holds proving \eqref{eqn: closure of multiplicity sets}.

Finally, suppose $x\in \sing[1,\ldots,k]^\uparrow\cap M_k$. By \eqref{eqn: M_k means only touches k}, there is some open ball $B_r(x)$ on which $\displaystyle\min_{1\leq i\leq k}u_i>\max_{k+1\leq j\leq K}u_j$. Then clearly $B_r(x)\cap \sing[1,\ldots, k]^\uparrow\subset \sing[1,\ldots,k]^\uparrow\cap M_k$, hence $\sing[1,\ldots,k]^\uparrow\cap M_k$ is relatively open in $\sing[1,\ldots, k]^\uparrow$.
\end{proof}

We also mention that under affine independence, there can be at most one tear of multiplicity $n+1$.

\begin{prop}[Uniqueness of maximal multiplicity tears]\label{prop: only one n+1 order singularity}
Let  $c(x, \xbar)=-\euclidean{x}{\xbar}$, and assume $\sourcemeas$, $\targetmeas$ are absolutely continuous probabilities measures on $\R^n$ with bounded supports. Also suppose $\left\{\targetpiece[i]\right\}_{i=1}^{n+1}$ is {\em any} affinely independent collection of path connected subsets of $\R^n$ (which may or may not decompose $\target$).
Then if $\convpot$ is an optimal potential transporting $\sourcemeas$ to $\targetmeas$, it can have at most one
point of multiplicity $n+1$ relative to $\left\{\targetpiece[i]\right\}_{i=1}^{n+1}$.
\end{prop}
\begin{proof}
 Suppose by contradiction there exist two points $x_0\neq y_0$ where $u$ has multiplicity $n+1$, then $\subdiff{\convpot}{x_0}$ and $\subdiff{\convpot}{y_0}$ each must intersect all of the sets $\targetpiece[i]$. First note that $\subdiff{\convpot}{x_0}$, $\subdiff{\convpot}{y_0}$ must have affine dimension $n$ (hence nonempty interior), otherwise there would be an $n-1$ dimensional affine plane intersecting all $\targetpiece[i]$. Now the convex function $\convpotdual$ is seen to be nondifferentiable on $\subdiff{\convpot}{x_0}\cap\subdiff{\convpot}{y_0}$, hence this intersection must have zero Lebesgue measure. In particular, the interiors of $\subdiff{\convpot}{x_0}$ and $\subdiff{\convpot}{y_0}$ are disjoint, and by~\cite[Theorem 11.3]{Rockafellar70}, $\R^n$ is divided into two closed, opposing halfspaces $H_+$ and $H_-$ with $\subdiff{\convpot}{x_0}\subset H_+$, $\subdiff{\convpot}{y_0}\subset H_-$.

Let us take $\xbar_i\in \subdiff{\convpot}{x_0}\cap\targetpiece[i]$ and $\ybar_i\in \subdiff{\convpot}{y_0}\cap\targetpiece[i]$; we see that $\xbar_i\in H_+$ while $\ybar_i\in H_-$ for each $1\leq i\leq n+1$. Now each $\targetpiece[i]$ is path connected, thus for each $i$ there exists some continuous path $\gamma_i(t)$ with $\gamma_i(0)=\xbar_i$ and $\gamma_i(1)=\ybar_i$, which remains inside $\targetpiece[i]$. Clearly there must exist some time $t_i\in[0, 1]$ at which $\gamma_i$ intersects the hyperplane $H_+\cap H_-$ for each $1\leq i \leq n+1$. However, this would imply that $H_+\cap H_-$ is an $n-1$ dimensional affine plane intersecting all of the sets $\targetpiece[i]$, a contradiction.
\end{proof}

%\section*{An idea which didn't go where I wanted it to}
%\marginpar{Worth keeping?}

\section{$C^{1,\alpha}$ smoothness of optimal map discontinuities: quadratic cost}
\label{section: smoother quadratic}

In a previous section,  affine independence of the target pieces was identified as the geometric
manifestation of the implicit function theorem hypothesis which guarantees DC smoothness of the corresponding tears. 
%\marginpar{must $\source$ be convex? \\ NO! Caffarelli}
%If $\source$ is compact and convex, the target components $\bar \Omega_{i_1}, \ldots, \bar \Omega_{i_k}$  
%are affinely independent and {\em strictly} convex, and the bounds (I) hold,  then the conclusions
%$u_{i_j} \in C^1(\R^n)$ of Proposition~\ref{prop: u is a maximum of C^1 potentials} imply
%the DC submanifold $M_k\cap \sing[i_1,\ldots,i_k]^\uparrow$  is also $C^1$ smooth.
%%provided by Theorem \ref{thm: higher codimension euclidean}. 
%This is the content of the next corollary. 
This section is devoted to improving this smoothness 
%of $M_k \cap \sing[i_1,\ldots,i_k]^\uparrow$  
to $C^{1,\alpha}_{loc}$ on $(\spt \mu)^{\interior}$.
%of $M_k \cap \Sigma_{i_1,\ldots,i_k}$.
%\marginpar{(Higher affine independence \\ yields a submanifold \\ with (nonsmooth) boundary)}
%To obtain the $C^{1,\alpha}_{loc}$ smoothness of such tears,
In order to establish this goal, we begin by recalling the required machinery from \cite{CaffarelliMcCann10}. Again, we will be working in the setting of $c(x, \xbar)=-\euclidean{x}{\xbar}$ on $\R^n\times \R^n$.

\begin{defin}[Affine doubling]\label{def: doubling}
 Suppose $\mu$ is a Borel measure on $\R^n$ and $x\in X\subset \R^n$. An open neighborhood $\nbhd[x]$ of $x$ is said to be a \emph{doubling neighborhood of $\mu$ with respect to $X$} if there exists a constant $\delta>0$ (called the \emph{doubling constant of $\mu$ on $\nbhd[x]$}) such that for any convex set $Z\subset \nbhd[x]$ whose (Lebesgue) barycenter is in $X$, 
\begin{align*}
 \mu(\frac{1}{2}Z)\geq \delta^2\mu(Z),
\end{align*}
here the dilation of $Z$ is with respect to its barycenter.
\end{defin}
\begin{defin}[Centered sections]\label{def: centered sections}
 If $\phi: \R^n\to \R\cup \{\infty\}$ is a convex function with $\subdiff{\phi}{\R^n}^{\interior}\neq \emptyset$, $\epsilon>0$, and $x_0\in \R^n$, the \emph{centered section} of $\phi$ at $x_0$ of height $\epsilon$ is defined by 
\begin{align*}
 Z_\epsilon^\phi(x_0):=\{x\in \R^n\mid \phi(x)<\epsilon +\phi(x_0)+\euclidean{v_\epsilon}{x-x_0}\}
\end{align*}
where $v_\epsilon$ is chosen so that $x_0$ is the barycenter of $Z_\epsilon^\phi(x_0)$, which is bounded.
\end{defin}
It is known that such a $v_\epsilon$ exists, and is unique (see \cite[Theorems A.7 and A.8]{CaffarelliMcCann10}).
\begin{defin}[$p$-uniform convexity]\label{def: p uniform convexity}
 If $p\geq 2$, a convex function $u$ is said to be \emph{$p$-uniform convex} on a set $\Omega$ if there is a finite constant $k>0$ such that for any $x_1$, $x_2\in \Omega$ and $\xbar_1\in \subdiff{u}{x_1}$, $\xbar_2\in \subdiff{u}{x_2}$,
\begin{align*}
 \euclidean{\xbar_2-\xbar_1}{x_2-x_1}\geq k^{1-p}\norm{x_2-x_1}^p.
\end{align*}
\end{defin}
\begin{rmk}\label{rem: p uniform convexity}
 This definition differs from the \emph{a priori} weaker \cite[Definition 7.9]{CaffarelliMcCann10}, but is equivalent. Indeed, the above inequality still holds if $\xbar_1$, $\xbar_2$ are points that can be written as limits of the form $\lim_{k\to\infty}\nabla u(x_{k, i})$ where $x_{k, i}\in \Dom(\nabla u)\cap\Omega$ and $x_{k, i}\to x_i$ as $k\to\infty$. Then since any $\xbar_i\in \subdiff{u}{x_i}$ can be written as a convex combination of such points, the formulation in Definition \ref{def: p uniform convexity} holds.
\end{rmk}

With these definitions in hand, we can state and prove the following refinement in the case when one of the pieces $\targetpiece[i]$ is strictly convex. Corollary \ref{cor: exceptional sets small} below %relates it to Corollary \ref{cor: convexity yields C^1 smoothness of tears}
will give conditions under which the exceptional sets $E_i$ of the theorem below 
lie in the boundary of $\source$.

\begin{thm}[H\"older continuity of optimal maps to closed convex target pieces]\label{thm: improvement when strictly convex}
Fix $c(x,\bar x)=-\langle x,\bar x\rangle$ and probability measures $\sourcemeas$, $\targetmeas$ with densities bounded away from zero and infinity on their supports in $\R^n$.  Let $\source$ be convex and $\target = \cup_{i\in I} \bar \Omega_i$ a finite disjoint union of closed sets
strongly separated by hyperplanes pairwise \eqref{eqn: subdifferentials disjoint}.
If $\targetpiece[i]$ is strictly convex, then
%$\nabla u_i$ is a homeomorphism between $(\sing[i]^\uparrow\cap\source)^{\interior}$ and $\targetpiece[i]^{\interior}$. Moreover,
 \begin{align*}
 u_i\in C^{1, \alpha}_{loc}((\sing[i]^\uparrow\cap \source)\setminus E_i)
 \end{align*}
  for some $\alpha\in (0, 1)$ (which depends only $n$, and the bounds of the densities of $\sourcemeas$ and $\targetmeas$ away from zero and infinity on their supports) where
\begin{align}\label{eqn: exceptional points}
E_i  := & \{ x \in (\sing[i]^\uparrow\cap \source)^\partial \mid  \nabla u_i(x) \in N_{\source}(x) + \ch \left(\subdiff{u}{x} \cap (\target \setminus \bar \Omega_i)\right)  \}  
\end{align}
and  $N_{\source}(x) := \{ v \in \R^n \mid \langle v, y-x \rangle \le 0 \mbox{\rm\ for all}\ y \in \source \}$
denotes the outer normal cone to the convex set $\source$ at $x$. 
\end{thm}

\begin{proof}
We may assume $i=1$.  Proposition \ref{prop: u is a maximum of C^1 potentials} 
asserts that $u_1\in C^1(\R^n)$, and Corollary \ref{cor: interior homeo/diffeo} implies
$\nabla u$ gives a homeomorphism between 
$(\target \cap \sing[1]^\uparrow)^{\interior}$ %$R:= \{ x \in \spt \mu \mid u = u_1\}^{\interior}$ 
and $\bar \Omega_i^{\interior}$ which extends continuously to the boundary.
The purpose of this theorem is to establish a H\"older estimate away from the exceptional set $E_1$.

Let us write for any Borel $A\subset \R^n$, $M_1(A):=\Leb{\nabla u_1(A)}$, the \emph{Monge-Amp{\`e}re measure} of $u_1$ (here $\Leb{\cdot}$ denotes the Lebesgue measure). Since $\nabla u_1(\R^n)\subset \targetpiece[1]$ which is convex, by \cite[Lemma 2]{Caffarelli92} we have for some constant $C>0$ depending only the bounds of the densities of $\sourcemeas$ and $\targetmeas$ away from zero and infinity on their supports, for any Borel $A\subset \R^n$

\begin{align}\label{eqn: u_1 is aleksandrov solution}
 C^{-1}\Leb{A\cap\sing[1]^\uparrow\cap \source}\leq M_1(A)\leq C\Leb{A\cap \sing[1]^\uparrow\cap \source}.
\end{align}
Suppose $x_0\in (\source)^\partial\cap (\sing[1]^\uparrow)^{\interior}$. Then for some $r_0>0$ small, the intersection $B_{r_0}(x_0)\cap \source\cap \sing[1]^\uparrow$ is convex and any convex $Z\subset B_{r_0}(x_0)\cap \source\cap \sing[1]^\uparrow$ satisfies \eqref{eqn: u_1 is aleksandrov solution}. Thus the proof of \cite[Lemma 7.5]{CaffarelliMcCann10} applies and we see $B_{r_0}(x_0)$ is a doubling neighborhood of $M_1$ with respect to $\source\cap \sing[1]^\uparrow$, with doubling constant $\delta_0$ depending only on $\sourcemeas$, $\targetmeas$, and $n$.

Next define the convex function $\tilde{u}$ by
\begin{align*}
 \tilde{u}(x)=
\begin{cases}
 u(x),&x\in \source,\\
 \infty, &\text{else},
\end{cases}
\end{align*}
then its Legendre transform $\tilde{u}^*$ is an optimal potential transporting $\targetmeas$ to $\sourcemeas$ which is finite on all of $\R^n$ with $\subdiff{\tilde{u}^*}{\R^n}\subset \source$ by convexity of $\source$. Since the restriction of $\tilde{u}^*$ will be an optimal potential transporting the restriction of $\targetmeas$ to $\targetpiece[1]$ to the restriction of $\sourcemeas$ to $\sing[1]^\uparrow\cap \source$ and $\targetpiece[1]$ is connected, by subtracting a constant we can assume $\tilde{u}^*=u^*$ on $\targetpiece[1]$. Writing for any Borel $A\subset \R^n$, $\tilde{M}(A):=\Leb{\subdiff{\tilde{u}^*}{A}}$ (the \emph{Monge-Amp{\`e}re measure} of $\tilde{u}^*$), by \cite[Lemma 2]{Caffarelli92} we then have for some constant $C>0$ depending only the bounds of the densities of $\sourcemeas$ and $\targetmeas$ away from zero and infinity on their supports, for any Borel $A\subset \R^n$
\begin{align*}
 C^{-1}\Leb{A\cap\target}\leq \tilde{M}(A)\leq C\Leb{A\cap \target}.
\end{align*}
In turn, since $\targetpiece[1]$ is convex we find the proof of \cite[Lemma 7.5]{CaffarelliMcCann10} applies hence for any $x\in \targetpiece[1]$ and $r>0$ such that $B_{r}(x)\cap \bigcup_{i\in I\setminus \{1\}}\targetpiece[i]=\emptyset$ , the open ball $B_{r}(x)$ is a doubling neighborhood of $\tilde{M}$ with respect to $\targetpiece[1]$, with doubling constant $\delta_0$ depending only on $\sourcemeas$, $\targetmeas$, and $n$.

Next, we will show that for $r>0$ fixed, there is some $\epsilon_0>0$ such that whenever 
 $x\in \sing[1]^\uparrow\cap\source$ and $\xbar=\nabla u_1(x)$ are such that 
\begin{align}\label{eqn: disjoint from bad points}
(\nabla u_1)^{-1}(B_r(\xbar))\cap 
%\{ y \in \sing[1]^\uparrow\cap \source \mid  \nabla u_1(y) \in N_{\source}(y) + \ch [\p u(x) \cap (\target \setminus \bar \Omega_1)]\} 
E_1 = \emptyset,  
\end{align}
and $\epsilon<\epsilon_0$, then the centered section $Z^{\tilde{u}^*}_\epsilon(\xbar)\subset B_r(\xbar)$. The proof will closely follow that of \cite[Lemma 7.11]{CaffarelliMcCann10}. Suppose the claim fails: for some fixed $r>0$ there exist sequences   $\xbar_j=\nabla u_1(x_j)$ with $x_j\in \sing[1]^\uparrow\cap\source$ 
satisfying \eqref{eqn: disjoint from bad points}, $\epsilon_j\searrow 0$ with $Z^{\tilde{u}^*}_{\epsilon_j}(\xbar_j)\not\subset B_r(\xbar_j)$. Extracting subsequences yields $\xbar_j\to \xbar_\infty$ and $x_j\to x_\infty$ with $\nabla u_1(x_\infty)=\xbar_\infty \in \targetpiece[1]$, still satisfying \eqref{eqn: disjoint from bad points}; let us also define 
\begin{align*}
 Z_{\min}:=\{\xbar\in \R^n\mid \tilde{u}^*(\xbar)=\tilde{u}^*(\xbar_\infty)+\euclidean{\xbar-\xbar_\infty}{x_\infty}\}=\subdiff{\tilde{u}}{x_\infty}.
\end{align*}
We can see that Claim \#1 in the proof of \cite[Lemma 7.11]{CaffarelliMcCann10} still holds, so in particular there is a nontrivial line segment contained in $Z_{\min}$, centered at $\xbar_\infty$ but otherwise disjoint from 
the set $\targetpiece[1]$ on which $\tilde u$ is strictly convex. 
Thus $\xbar_\infty \in (\targetpiece[1])^\partial$ and Corollary \ref{cor: interior homeo/diffeo}
implies $x_\infty \in (\sing[1]^\uparrow \cap \source)^\partial$.
Reordering if necessary, we may assume $u_i(x_\infty)$ depends monotonically on $i$, with
$u_1(x_\infty) = u_2(x_\infty) = \cdots = u_k(x_\infty) > u_{k+1}(x_\infty)$ for some $k\ge 1$.
Then
\begin{align}
\nonumber
 \subdiff{\tilde{u}}{x_\infty}&=\subdiff{u}{x_\infty}+N_{\source}(x_\infty)
\\ &=\ch(\{\xbar_\infty \}\cup\subdiff{u_2}{x_\infty} \cup \cdots \cup \subdiff{u_k}{x_\infty})+N_{\source}(x_\infty)
\label{eqn: decompose into bounded and convex cone}
\end{align}
decomposes as the sum of a bounded component and a convex cone,
in view of Lemma \ref{lem: dm}.  
Since  \eqref{eqn: disjoint from bad points} for $\xbar_k$ implies
$(\nabla u_1)^{-1}(B_r(\xbar_\infty)) \cap E_1 =\emptyset$, we see
$\xbar_\infty$ is not contained in the closed convex set
$$
\ch(N_{\source}(x_\infty)  +\cup_{i=2}^k \subdiff{u_i}{x_\infty} )
= \ch (\subdiff{u}{x_\infty} \cap (\target \setminus \bar \Omega_1)),
$$ 
hence can be strongly separated from it by a hyperplane (\cite[Corollary 11.4.2]{Rockafellar70}).  Any segment 
in \eqref{eqn: decompose into bounded and convex cone} centered at $x_\infty$ must 
be parallel to this hyperplane.  But this can only occur if the closed cone $N_{\source}(x_\infty)$
contains a complete line parallel to this segment,  contradicting the fact that $\source$ has non-empty
interior.

Thus \cite[Theorem 7.13 and Corollary 7.14]{CaffarelliMcCann10} will apply (note that differentiability of $\tilde{u}^*$ is not actually necessary to do this), proving that $u$ is $C^{1, \alpha}_{loc}$ on $(\sing[1]^\uparrow\cap \source)\setminus E_1$.

\end{proof}

In addition to giving conditions under which the exceptional sets $E_i$ of the theorem above  
lie in the boundary of $\source$, the following corollary shows the codimension $k$ submanifolds of 
Corollary \ref{cor: convexity yields C^1 smoothness of tears} enjoy H\"older differentiability,
except possibly where they intersect the boundary $(\source)^{\p}$ tangentially.

\begin{cor}[H\"older %differentiability except at
%limits exceptional sets to non-transverse
regularity away from
tangential tear-boundary intersections]\label{cor: exceptional sets small}
Fix $x \in E_1$ in Theorem \ref{thm: improvement when strictly convex}. 
Assume $u_i(x) \ge u_{i+1}(x)$ for all $i\in I$,
and $u_1(x)=u_k(x)>u_{k+1}(x)$. %,  so that $x \in M_k$.  
Also suppose the collection %the convex hulls of 
$\{\ch\left(\subdiff{u}{x} \cap \targetpiece[i]\right)\}_{i=1}^k$ is affinely independent.  
%Fix $j \le k$ and suppose
If $\targetpiece[1]$ is strictly convex then $x \in (\source)^\partial$. %Fix $j\le k$.
If additionally, $\targetpiece[i]$ is strictly convex for all $i \le k $ and $(\source)^\partial$ is differentiable at $x$, then
$\sing[\{1,2,\ldots,k\}]$ intersects $(\source)^\partial$ tangentially, meaning that the outer unit
normal to $\source$ at $x$ is also normal to the $C^1$ submanifold $\sing[\{1,2,\dots,k\}]$.
In this case, $\sing[\{1,2,\dots,k\}]^\uparrow \cap \source$ is $C^{1,\alpha}_{loc}$ smooth, away from
 any such tangential intersections (and any possible non-differentiabilities of $(\source)^\partial$).
\end{cor}

\begin{proof}
Suppose $x \in E_1 \subset  \sing[1]^\uparrow \cap \spt \mu$. By our assumptions and Lemma \ref{lem: dm}, we have $\subdiff{u}{x}\subset \ch\left(\bigcup_{i=1}^k\targetpiece[i]\right)$, hence
\begin{align*}
 \ch \left(\subdiff{u}{x} \cap (\target \setminus \bar \Omega_1)\right)\subset \ch \left(\bigcup_{i=2}^k(\subdiff{u}{x} \cap \targetpiece[i])\right)=\bigcup_{i=2}^k\ch \left(\subdiff{u}{x} \cap \targetpiece[i]\right)
\end{align*}
Thus there
exist $\xbar_i \in \ch\left(\subdiff{u}{x} \cap \targetpiece[i]\right)$ and $t_i \ge 0$ with $1=\sum_{i=2}^k t_i$ such that 
\begin{align}\label{eqn: lc normal}
\sum_{i=2}^k t_i (\nabla u_1(x) - \xbar_i)  \in N_{\source}(x)
\end{align}
according to \eqref{eqn: exceptional points} of Theorem \ref{thm: improvement when strictly convex}.
Setting $\xbar_1 = \nabla u_1(x)$,  
the affine independence of $\{\xbar_i\}_{i \le k}$ makes $\{\bar x_1-\bar x_i\}_{2\le i \le k}$ linearly independent.
Thus $\sum_{i=2}^k t_i =1$ forces the sum in \eqref{eqn: lc normal} not to vanish.  

Now $x \in (\source)^{\interior}$ would force $N_{\source}(x) =\{0\}$, 
%--- a contradiction.
% with the fact that $1=\sum_{i=2}^k t_i$ in %
contradicting the last sentence. %\eqref{eqn: lc normal}. 
Thus we conclude $x$ is contained in the boundary of
$\source$.  If, in addition, $\targetpiece[i]$ is strictly convex for all $i \le k$ then 
$\nabla u_1(x) - \bar x_i = \nabla u_1(x)-\nabla u_i(x)$ is a (non-zero) normal to the hypersurface
$\sing[\{1,i\}] = \{u_1=u_i\}$, which is $C^1$ smooth 
by Corollary \ref{cor: convexity yields C^1 smoothness of tears}, noting that a collection of two sets is affinely independent if they are disjoint.
Thus the sum in \eqref{eqn: lc normal} is normal to the codimension 
$k-1$ submanifold $\sing[\{1,\ldots,k\}] = \cap_{i=2}^k \sing[\{1,i\}]$ of the same corollary.
Since  \eqref{eqn: lc normal} is non-vanishing, it is an outer normal to $\spt \mu$ when the latter 
is differentiable at $x$.  Away from such points,  the improvement in regularity from $C^1$ to $C^{1,\alpha}_{loc}$ comes from Theorem \ref{thm: improvement when strictly convex} and the implicit 
function theorem.
\end{proof}

When $k=2$ and both target components are strictly convex,  an analous result was shown 
simultaneously and independently from us by Chen \cite{Chen16p},  who went on to show $C^{2,\alpha}$
regularity of the tear provided the target components are sufficiently far apart.

\section{Global structure of optimal map discontinuities: MTW costs}\label{subsection: global structure general cost}
\label{section: global MTW}

For quadratic transportation costs, we have already shown that when the support of the target measure consists of a number of affinely independent regions, the optimal transport map induces a partition of the source domain into sets corresponding to each of these regions. In this section,  we extend one such result --- 
Proposition \ref{prop: euclidean case interface} --- to MTW costs. %more general costs satisfying \eqref{MTW}.  
While we expect other results
from Sections \ref{section: global quadratic} and \ref{section: smoother quadratic} also to have analogs for such costs,   we do not pursue such extensions in the present
manuscript.

\begin{thm}[Pairwise partitions of source]\label{thm: interface}
 Suppose the cost function $c$ satisfies \eqref{B1} and \eqref{MTW}, and $\outerdom$ and $\outertarget$ are $c$-convex with respect to each other. Also suppose $\sourcemeas$ and $\targetmeas$ are absolutely continuous probability measures on $\outerdom$ and $\outertarget$ respectively, where $\target=\targetpiece[1]\cup \targetpiece[2]$ is such that there exists $\xbar_0\in\outertarget$ for which the sets $\bigcup_{x\in\outerdom}[-D_{\xbar}D_xc(x, \xbar_0)]^{-1}(\coord{\targetpiece[1]}{x})$ and $\bigcup_{x\in\outerdom}[-D_{\xbar}D_xc(x, \xbar_0)]^{-1}(\coord{\targetpiece[2]}{x})$ are strongly separated by a hyperplane $\Pi\subset \tanspMbar{\xbar_0}$.
% 
%  the hyperplane $\{p\in \cotanspMbar{\xbar_0}\mid \euclidean{p}{v}=a_0\}$ strongly separates $\bigcup_{x\in\outerdom}[-D_{\xbar}D_xc(x, \xbar_0)]^{-1}(\coord{\targetpiece[1]}{x})$ from $\bigcup_{x\in\outerdom}[-D_{\xbar}D_xc(x, \xbar_0)]^{-1}(\coord{\targetpiece[2]}{x})$ with spacing $d_0$. 
 
 Then an optimal potential $u$ transporting $\sourcemeas$ to $\targetmeas$ can be written as $u=\max\{u_1, u_2\}$, where $u_1$ and $u_2$ are $c$-convex functions such that 
\begin{align}\label{eqn: mapping destination}
 \cExp{x}{D u_i(x)}&\in \targetpiece[i],\quad a.e.\ x\in\outerdom.
\end{align}
Moreover, under the global coordinates induced by $x\mapsto -D_{\xbar}c(x, \xbar_0)$ on $\outerdom$, the sets $\{u_1=u_2\}$, $\{u_1>u_2\}$, and $\{u_1<u_2\}$ in $\outerdom$ are given by the graph, open epigraph, and open subgraph respectively of a DC function $h$ defined on the projection of $\coord{\outerdom}{\xbar_0}$ onto the hyperplane $\Pi$.

Additionally, if 
\begin{align}\label{eqn: disjointness of c-subdifferential condition}
 \left[\bigcup_{x\in\outerdom}\cExp{x}{\ch\left(\coord{\targetpiece[1]}{x}\right)}\right]\cap \left[\bigcup_{x\in\outerdom}\cExp{x}{\ch\left(\coord{\targetpiece[2]}{x}\right)}\right]=\emptyset,
\end{align}
then the sets $\{u_1=u_2\}$, $\{u_1\geq u_2\}$, and $\{u_1\leq u_2\}$ are connected.
\end{thm}

%We now turn to the proof with more general cost function.

\begin{proof}[Proof of Theorem \ref{thm: interface}]
Again Lemma \ref{lem: u is a maximum of potentials} gives the representation  $u=\max\{u_1, u_2\}$ and \eqref{eqn: mapping destination}, note we have not exploited any convexity properties of the $\targetpiece[i]$ so far. 
Write 
\begin{align*}
 \outerdom_=:&=\{x\in\outerdom\mid u_1(x)=u_2(x)\},\\
 \outerdom_<:&=\{x\in\outerdom\mid u_1(x)<u_2(x)\},\quad
 \outerdom_>:=\{x\in\outerdom\mid u_1(x)>u_2(x)\},\\
  \outerdom_\leq:&=\{x\in\outerdom\mid u_1(x)\leq u_2(x)\},\quad
 \outerdom_\geq:=\{x\in\outerdom\mid u_1(x)\geq u_2(x)\}.
\end{align*}

Now make a change of variables under $\cstarExp{\xbar_0}{\cdot}$ and define $\util_i: \coord{\outerdom}{\xbar_0}\to\R$ by 
\begin{align*}
 \util_i(p):&=u_i(\cstarExp{\xbar_0}{p}),\quad \util:=\max\{\util_1, \util_2\}.
\end{align*}
We will identify $\cotanspMbar{\xbar_0}\cong \tanspMbar{\xbar_0}\cong \R^n$, and without loss of generality assume the separating hyperplane $\Pi$ is $\{p^n=\midheight[0]\}$ for some $\midheight[0]\in\R$, with width $d_0>0$.

Now take any point $p_0\in\outerdom^{\interior}$ with $\util_1(p_0)=\util_2(p_0)$. For a sufficiently small $r>0$, there is some $C>0$ for which $\utiltil_i:=\util_i+\frac{C}{2}\norm{p-p_0}^2$ are both convex functions on $B_r(p_0)\subset \outerdom$. 
Since $c$ satisfies \eqref{MTW} and $\outerdom$ and $\outertarget$ are $c$-convex with respect to each other, writing $x_0:=\cstarExp{\xbar_0}{p_0}$ we have for $i=1$, $2$,
\begin{align*}
 \subdiff{\utiltil_i}{p_0}&=\subdiff{\util_i}{p_0}=[-D_{\xbar}D_xc(x_0, \xbar_0)]^{-1}(\subdiff{u_i}{x_0})\\
 &=[-D_{\xbar}D_xc(x_0, \xbar_0)]^{-1}\coord{\csubdiff{u_i}{x_0}}{x_0}\\
 &\subset \ch\left([-D_{\xbar}D_xc(x_0, \xbar_0)]^{-1}(\coord{\targetpiece[i]}{x_0})\right),
\end{align*}
which are strongly separated from each other by $\{p^n=a_0\}$ with spacing $d_0$ by assumption. Then by \cite[Corollary 24.5.1]{Rockafellar70}, there is some $r>0$ for which $\subdiff{\utiltil_1}{B_r(p_0)}$ and $\subdiff{\utiltil_2}{B_r(p_0)}$ are still strongly separated with spacing $d_0$. Let us write $\tilde{\outerdom}_=$, $\tilde{\tilde{\outerdom}}_=$ for $\outerdom_=$ with $u_i$ replaced by $\util_i$ or $\utiltil_i$ and $\outerdom$ by $\coord{\outerdom}{\xbar_0}$ or $\coord{\outerdom}{\xbar_0}\cap B_r(p_0)$ respectively (and likewise for $<$ and $>$). Then we may apply Corollary \ref{cor: explicit function theorem} to find that the sets $\tilde{\tilde{\outerdom}}_{=}$, $\tilde{\tilde{\outerdom}}_{<}$, and $\tilde{\tilde{\outerdom}}_{>}$ 
%\begin{align*}
%\tilde{\tilde{\outerdom}}_{=}=\tilde{\outerdom}_{=}\cap B_{r}(0),\ \tilde{\tilde{\outerdom}}_{<}=\tilde{\outerdom}_{<}\cap B_{r}(0),\ \tilde{\tilde{\outerdom}}_{>}=\tilde{\outerdom}_{>}\cap B_{r}(0),
%\end{align*}
 are the graph, open subgraph, and open epigraph respectively of the function
\begin{align*}
p'\mapsto \frac{-\utiltil^*_{p'}(a_0-d_0)+\utiltil^*_{p'}(a_0+d_0)}{2d_0}=h(p')
\end{align*}
over $B_r(p_0)\subset \R^{n-1}$ where again, $\utiltil^*_{p'}$ is the Legendre transform in just the $n$th variable. Now by \cite[Theorem 16.4]{Rockafellar70} we see that 
\begin{align*}
\utiltil^*_{p'}(a_0-d_0)%&=(\util_{p'}^*\Box \left(\frac{C(\norm{p'}^2+(\cdot)^2)}{2}\right)^*)(a_0-d_0)\\
&=\inf_{s\in\R}(\util_{p'}^*(a_0-d_0-s)+\left(\frac{C(\norm{p'-p_0'}^2+(\cdot-p_0^n)^2)}{2}\right)^*(s)),
%\\
%&=\inf_{s\in\R}(\util_{p'}^*(a_0-d_0-s)+\frac{s^2}{2C})-\frac{C}{2}\norm{p'}^2.
\end{align*}
and a quick calculation yields
\begin{align}
\left(\frac{C(\norm{p'-p_0'}^2+(\cdot-p_0^n)^2)}{2}\right)^*(s)&= \sup_{t\in \R}(ts-\frac{C}{2}(\norm{p'-p_0'}^2+(t-p_0^n)^2))\notag\\
&=\sup_{t\in \R}(ts-\frac{C}{2}(t-p_0^n)^2)-\frac{C}{2}\norm{p'-p_0'}^2\notag\\
 &=\frac{s^2}{2C}+sp_0^n-\frac{C(p_0^n)^2}{2}-\frac{C}{2}\norm{p'-p'_0}^2.\label{eqn: utiltil conj}
%  &s/C=t\\
%  &-\inf_{s}(u^*_{p'}(a_0-d_0-s)+\frac{s^2}{2C})-\frac{C}{2}\norm{p'}^2\\
%  &\quad+\inf_{r}(u^*_{p'}(a_0+d_0-r)+\frac{r^2}{2C})+\frac{C}{2}\norm{p'}^2\\
%  &=\inf_{r}\sup_{s}(u^*_{p'}(a_0+d_0-r)+\frac{r^2}{2C}-u^*_{p'}(a_0-d_0-s)-\frac{s^2}{2C})
\end{align}
At the same time, from the proof of Theorem \ref{thm: explicit function theorem} we see that $\util_{p'}^*$ is a convex function with $p_0^n\in \subdiff{\util_{p'}^*}{a_0-d_0}$. Thus by \eqref{eqn: utiltil conj} we find, 
\begin{align*}
 \utiltil^*_{p'}(a_0-d_0)&=\inf_{s\in\R}(\util_{p'}^*(a_0-d_0-s)+\frac{s^2}{2C}+sp_0^n-\frac{C(p_0^n)^2}{2}-\frac{C}{2}\norm{p'-p'_0}^2)\\
 &\geq \util_{p'}^*(a_0-d_0)-sp_0^n+\frac{s^2}{2C}+sp_0^n-\frac{C(p_0^n)^2}{2}-\frac{C}{2}\norm{p'-p'_0}^2\\
 &\geq \util_{p'}^*(a_0-d_0)-\frac{C(p_0^n)^2}{2}-\frac{C}{2}\norm{p'-p'_0}^2,
\end{align*}
with equality achieved for the choice $s=0$.
A similar calculation shows 
\begin{align*}
 \utiltil^*_{p'}(a_0+d_0)=\util_{p'}^*(a_0+d_0)-\frac{C(p_0^n)^2}{2}-\frac{C}{2}\norm{p'-p'_0}^2
\end{align*}
hence 
\begin{align*}
 h(p')=\frac{-\util_{p'}^*(a_0-d_0)+\util_{p'}^*(a_0+d_0)}{2d_0},
\end{align*}
the significance being that this function does not depend on the constant $C$ in $\utiltil$, hence is independent of the point $p_0$. Since $\utiltil_1(p)=\utiltil_2(p)$ for $p\in B_r(p_0)$ if and only if $\util_1(p)=\util_2(p)$, by the continuity of $u_i$ up to the boundary of $\outerdom$ we obtain that $\tilde{\outerdom}_{=}$, $\tilde{\outerdom}_{<}$, $\tilde{\outerdom}_{>}$ are equal to the graph, open subgraph, open epigraph respectively of $h$ over the projection of $\coord{\outerdom}{\xbar_0}$ on $\Pi$.

Now assume condition \eqref{eqn: disjointness of c-subdifferential condition} holds. Since $\csubdiff{u_i}{x}\subset \cExp{x}{\ch\left(\coord{\targetpiece[i]}{x}\right)}$, by \eqref{eqn: mapping destination} combined with \cite[Theorem 3.1]{Loeper09}, we see that
\begin{align*}
 \csubdiff{u_1}{\outerdom}\cap \csubdiff{u_2}{\outerdom}=\emptyset.
\end{align*}
We now claim 
\begin{align}\label{eqn: equality of c-subdifferential images}
 \cstarsubdiff{u^c}{\csubdiff{u_1}{\outerdom}}=\outerdom_{\geq},
\end{align}
which concludes the proof by Lemma \ref{lem: connectivity of c-subdifferential} (and a symmetric argument switching the roles of $u_1$ and $u_2$).

Suppose $u_1(x)\geq u_2(x)$. Then since $\coord{\csubdiff{u_i}{x}}{x}=\subdiff{u_i}{x}$, Lemma \ref{lem: dm} yields that 
\begin{align*}
 \csubdiff{u}{x}=\csubdiff{u_1}{x}\text{ or }\csubdiff{u}{x}=\cExp{x}{\ch(\subdiff{u_1}{x}\cup \subdiff{u_2}{x})}.
\end{align*}
In either case, there exists $\xbar\in \csubdiff{u_1}{x}\cap \csubdiff{u}{x}$ which implies $x\in\cstarsubdiff{u}{\xbar}$, and in particular $x\in \cstarsubdiff{u^c}{\csubdiff{u_1}{\outerdom}}$, thus $\cstarsubdiff{u^c}{\csubdiff{u_1}{\outerdom}}\supset\outerdom_{\geq}$. 
On the other hand, suppose $x\in \cstarsubdiff{u^c}{\csubdiff{u_1}{\outerdom}}$ but $u_1(x)<u_2(x)$. Then there exist $y\in\outerdom$ and $\xbar\in\outertarget$ with $\xbar\in\csubdiff{u_1}{y}$ and $x\in \cstarsubdiff{u^c}{\xbar}$, or equivalently $\xbar\in\csubdiff{u}{x}$. However, since $u_1(x)<u_2(x)$, we can again use Lemma \ref{lem: dm} to see that  $\csubdiff{u}{x}=\csubdiff{u_2}{x}$. This contradicts the disjointness of  $\csubdiff{u_1}{\outerdom}$ and $\csubdiff{u_2}{\outerdom}$, thus we must have \eqref{eqn: equality of c-subdifferential images}.
\end{proof}

\section{Stability of tears}
\label{section: stability of tears}

Our main goal of this section is to establish a stability result for the multiplicity of singularities of an optimal potential, under certain perturbations of the target measure. To do so, we must first choose an appropriate notion of perturbation for the target measure. In this case, we would only expect stability under perturbations of the target measure that prohibit moving even small amounts of mass to a far away location. Thus a good candidate is the $\mathcal{W}_\infty$ metric defined below.
\begin{defin}[$\infty$-Kantorovich-Rubinstein-Wasserstein distance]\label{def: W-infinity}
 Given two probability measures $\nu_1$ and $\nu_2$ on $\Mbar$, the $\mathcal{W}_\infty$ distance between them is defined by
\begin{align*}
 \Winfty{\nu_1}{\nu_2}:=\inf\left\{\lVert d_{\gbar}\rVert_{L^\infty(\gamma)}\mid \gamma\in\Pi(\nu_1, \nu_2)\right\}.
\end{align*}
Here, $d_{\gbar}$ is the geodesic distance on $\Mbar$ induced by the associated Riemannian metric, and $\Pi(\nu_1, \nu_2)$ is the set of probability measures on $\Mbar\times\Mbar$ whose left and right marginals are $\nu_1$ and $\nu_2$, repectively
\end{defin}

To obtain stability, we again require affine independence (Definition \ref{defin: affine independence of sets}) of the pieces of the support of the target measure. See Example \ref{ex: unstable example} for a counterexample to stability when this independence is not present.

We are now ready to state the stability result.

%\marginpar{I somehow feel stronger stability results should be true, with weaker metrics and non-convex targets}

\begin{thm}[Stability of tears]\label{thm: OT stability theorem}
 Suppose a cost function $c: \outerdom\times\outertarget\to\R$ satisfies \eqref{B1} and \eqref{MTW}, and the measures $\sourcemeas$ and $\targetmeas$ on  $\outerdom \subset \M$ and $\outertarget \subset \Mbar$ respectively satisfy conditions (I) and (II) above \eqref{eqn: FKM conditions}--\eqref{eqn: GK conditions}.
 Also let $u$ be an optimal potential transporting $\sourcemeas$ to $\targetmeas$ with cost $c$ and suppose $\convpot$ has multiplicity $k+1\leq K$ at $x_0\in\sourceint$, relative to a finite collection $\left\{\targetpiece[i]\right\}_{i=1}^{K}$
of disjoint compact sets whose union is $\target$.  Reorder if necessary,  so that $u$ also has multiplicity $k+1$ with 
respect to the subcollection $\left\{\coord{\targetpiece[i]}{x_0}\right\}_{i=1}^{k+1}$ consisting of the first $k+1$ sets;  assume this subcollection
is affinely independent and consists of strictly convex sets.

 Then for any $\epsilon>0$, there exists a $\threshold>0$ depending only on $\epsilon,c$, $\source$, and $\{\targetpiece[i]\}_{i=1}^K$, such that for any $\perturbmeas[\threshold]$ with $\Winfty{\targetmeas}{\perturbmeas[\threshold]}<\threshold$ and any optimal potential $\perturbconvpot{\threshold}$ transporting $\sourcemeas$ to $\perturbmeas[\threshold]$, there is a DC submanifold of dimension $n-k$ in 
$B_{\epsilon}{\left(x_0\right)}\subset \R^n$ on which $\perturbconvpot{\threshold}$ has multiplicity $k+1$ relative to $\left\{\nbhdof[\threshold]{\targetpiece[i]}\right\}_{i=1}^{K}$ at every point.
\end{thm}

The discrepancy of $k$ versus $k+1$ between Theorem~\ref{thm: stability theorem} and Theorem \ref{thm: OT stability theorem} arises because the affine hull of $k+1$ affinely independent points generates an affine subspace of dimension $k$.

We first show a lemma which uses the affine independence assumption to deduce $\dim \p u(x_0)=k$, so that
Theorem~\ref{thm: stability theorem} can be applied. To do so requires some finer properties of the $c$-subdifferentials of each of the functions $u_i$ which make up $u$ in the decomposition constructed in Lemma \ref{lem: u is a maximum of potentials}.

%\marginpar{I'm lost}
%:  in the proof of this lemma,  but not the statement, we seem to be assuming $c$-convexity
%of the $\targetpiece[i]$}

\begin{lem}\label{lem: with singular point}
 Suppose $\left\{\convpot[i]\right\}_{i=1}^K$ is the collection of $c$-convex functions obtained by applying Lemma~\ref{lem: u is a maximum of potentials} to the optimal potential $u$ under the conditions of Theorem~\ref{thm: OT stability theorem}. Ordering indices as in Theorem \ref{thm: OT stability theorem}, $u_i \in C^1(\Omega)$ for $i\le k+1$ and
\begin{align}
 \subdiff{\convpot}{x_0}\cap \coord{\targetpiece[i]}{x_0}
 &=\begin{cases}
 \left\{D\convpot[i](x_0)\right\},& 1\leq i \leq k+1,\\
 \emptyset,&k+1< i \leq K,
 \end{cases}\label{eqn: singular subdifferential intersections}\\
 \subdiff{\convpot}{x_0}&=\ch\left(\bigcup_{1\leq i\leq k+1}\{D\convpot[i](x_0)\}\right),\label{eqn: singular subdifferential representation}\\
 \convpot(x_0)&=\convpot[i](x_0),\ 1\leq i\leq k+1\label{eqn: equality at singular point},\\
 \convpot(x_0)&>\convpot[i](x_0),\ k+1<i\leq K.\label{eqn: other potentials are lower at singular point}
\end{align}
Additionally, $\dim\subdiff{\convpot}{x_0}=k$.
\end{lem}

\begin{proof}
%\marginpar{What does condition (III) refer to?}
 Apply Proposition~\ref{prop: u is a maximum of C^1 potentials} to obtain $\left\{\convpot[i]\right\}_{i=1}^K$. Recall that $\coord{\csubdiff{\convpot}{x_0}}{x_0}=\subdiff{\convpot}{x_0}$ under \eqref{B1}, \eqref{MTW}, and conditions (I)--(II); thus Proposition~\ref{prop: u is a maximum of C^1 potentials} and the fact that the multiplicity of $\convpot$ at $x_0$ relative to $\left\{\targetpiece[i]\right\}_{i=1}^K$ is $k+1$ implies that $\subdiff{\convpot}{x_0}$ intersects exactly $k+1$ of the sets $\coord{\targetpiece[i]}{x_0}$, each at exactly one point.

Re-number the indices $1\leq i \leq K$ so that $\subdiff{\convpot}{x_0}$ intersects $\coord{\targetpiece[i]}{x_0}$ only for $1\leq i\leq k+1$. Since $D\convpot[i](x_0)\in\coord{\targetpiece[i]}{x_0}$ for each $i$, Lemma~\ref{lem: dm} along with the mutual disjointness of the $\targetpiece[i]$ immediately gives  \eqref{eqn: singular subdifferential intersections}, \eqref{eqn: singular subdifferential representation}, \eqref{eqn: equality at singular point}, and \eqref{eqn: other potentials are lower at singular point}.

Finally by \eqref{eqn: singular subdifferential representation}, it is clear that $\affdim\left(\subdiff{\convpot}{x_0}\right)\leq k$. However, if $\affdim\left(\subdiff{\convpot}{x_0}\right)< k$, the collection $\left\{\coord{\targetpiece[i]}{x_0}\right\}_{i=1}^{k+1}$ would fail to be affinely independent, thus we must have equality. This finishes the proof.
\end{proof}

We are now in a situation to appeal to Theorem~\ref{thm: stability theorem} and finish the proof of the stability theorem.
\begin{proof}[Proof of Theorem~\ref{thm: OT stability theorem}]
We first apply Lemma~\ref{lem: with singular point} and reorder indices if necessary to obtain $c$-convex functions $\convpot[i]$, $1\leq i\leq K$ with properties \eqref{eqn: singular subdifferential intersections} through \eqref{eqn: other potentials are lower at singular point}.

 Now fix an $\epsilon>0$ and suppose by contradiction that the theorem fails to hold: then there exist sequences $\threshold[j]\searrow 0$ and $\perturbmeas[j]$ with $\Winfty{\targetmeas}{\perturbmeas[j]}<\threshold[j]$, and optimal potentials $\perturbconvpot{j}$ transporting $\sourcemeas$ to $\perturbmeas[j]$ with cost function $c$, but
 $\perturbconvpot{j}$ does \emph{not} have $\threshold[j]$-multiplicity $k+1$ at each point 
of a codimension $k$, DC submanifold of 
$B_{\epsilon}{\left(x_0\right)}$. Since $c\in C^4(\outerdom\times\outertarget)$ and each $\perturbconvpot{j}$ is $c$-convex, the collection $\left\{\perturbconvpot{j}\right\}_{j=1}^\infty$ is uniformly Lipschitz. Then by Arzel{\`a}-Ascoli (after adding constants to each $\perturbconvpot{j}$, which does not change the $\threshold[j]$-multiplicity of any points) we can extract a subsequence, still indexed by $j$, that converges uniformly. By stability of optimal transport maps (see for example, \cite[Corollary 5.23]{Villani09}) and convexity of $\source$ this limit must be (again, up to adding a constant) equal to $u$. 
 
 Now by taking $j$ large enough we may ensure the sets $\nbhdof[{{\threshold[j]}}]{\targetpiece[i]}$ are mutually disjoint for each $j$; note that by the definition of $\mathcal{W}_\infty$, the assumption $\Winfty{\targetmeas}{\perturbmeas[j]}<\threshold[j]$ implies $\spt\targetmeas^j\subset\bigcup_{i=1}^{K}\nbhdof[{{\threshold[j]}}]{\targetpiece[i]}$. Thus, as in Lemma \ref{lem: u is a maximum of potentials} we obtain 
 \begin{align*}
 \perturbconvpot[i]{j}(x):&=\sup_{\xbar\in\nbhdof[{{\threshold[j]}}]{\targetpiece[i]}}(-c(x, \xbar)-\left(u^j\right)^{c}(\xbar)),\\
\perturbconvpot{j}(x)&=\max_{1\leq i\leq K}{\perturbconvpot[i]{j}(x)},
\end{align*}
for $x\in\source$ as long as $j$ is large enough. We also comment here that $\perturbconvpot[i]{j}$ converges uniformly to $\convpot[i]$ for each $1\leq i\leq K$, while the compactness of each set $\nbhdof[{{\threshold[j]}}]{\targetpiece[i]}$ implies that 
\begin{align}\label{eqn: each perturbed csubdifferential touches a piece}
\csubdiff{\perturbconvpot[i]{j}}{x}\cap \nbhdof[{{\threshold[j]}}]{\targetpiece[i]}\neq\emptyset,\quad\forall\;x\in\outerdom.
\end{align}

We can now take a local coordinate system near $x_0$ to view all functions as defined in a subset of $\R^n$, since all $\perturbconvpot[i]{j}$ and $\convpot[i]$ are $c$-convex they have uniformly bounded constant of semi-convexity near $x_0$ (see \cite[Proposition C.2]{GangboMcCann96}). Thus $\convpot$ and $\left\{\perturbconvpot{j}\right\}_{j=1}^\infty$ satisfy the conditions of Theorem~\ref{thm: stability theorem}, and for $j$ sufficiently large, we obtain existence of a 
DC submanifold $\perturbsing[n-k]{j}\subset B_{\epsilon}{\left(x_0\right)}$ of codimension $k$
 satisfying $\dim \subdiff{\perturbconvpot{j}}{x} \ge k$ for every $x\in \perturbsing[n-k]{j}$. 
%Since the Euclidean distance coming from local coordinates is comparable to the Riemannian distance on a small %neighborhood of $x_0$, we find that $\haus[n-k]_g(\perturbsing[n-k]{j})>0$ as well.

At this point, fix any $x\in\perturbsing[n-k]{j}$. By \eqref{eqn: perturbed touching} and Lemma~\ref{lem: dm} we see that 
\begin{align*}
 \csubdiff{\perturbconvpot{j}}{x}&=\cExp{x}{\subdiff{\perturbconvpot{j}}{x}}\\
 &=\cExp{x}{\ch{\left(\bigcup_{1\leq i\leq k+1}\subdiff{\perturbconvpot[i]{j}}{x}\right)}},
\end{align*}
thus \eqref{eqn: each perturbed csubdifferential touches a piece} implies that for $j$ large enough $\perturbconvpot{j}$ has $\threshold[j]$-multiplicity at least $k+1$ at $x$. On the other hand by the mutual disjointness of $\left\{\targetpiece[i]\right\}_{i=1}^K$ and recalling $\subdiff{\perturbconvpot[i]{j}}{x}=\coord{\csubdiff{\perturbconvpot[i]{j}}{x}}{x}$, Lemma~\ref{lem: uniform convergence of subdifferentials} yields that for $j$ large enough, $1\leq i\leq k+1$, 
and $i \neq i' \le K$, we have $\csubdiff{\perturbconvpot[i]{j}}{x}\cap\nbhdof[{{\threshold[j]}}]{\targetpiece[i']}=\emptyset$; in particular this implies $\perturbconvpot{j}$ has $\threshold[j]$-multiplicity no more than $k+1$ at $x$. Thus if $j$ is large enough, $\perturbconvpot{j}$ has $\threshold[j]$-multiplicity exactly $k+1$ at every point in $\perturbsing[n-k]{j}$, which finishes the proof by contradiction.
\end{proof}

\appendix
\section{Failure of stability without affine independence}\label{appendix}
In this appendix, we provide an example to illustrate the importance of the affine independence condition on the support of the target measure in Theorem \ref{thm: OT stability theorem}. Note simply by definition, no collection of $n+2$ or more sets can be affinely independent in $\R^n$. The example we illustrate below has a target measure on $\R^2$ whose support consists of four strictly convex sets, and the associated optimal potential has a point of multiplicity $4$ which is unstable under certain $\mathcal{W}_\infty$ perturbations. The source measure will have constant density, and the target measure will be absolutely continuous with density bounded from above. This density does not have a lower bound away from zero in its whole support, so it does not exactly satisfy all of the remaining (i.e. other than affine independence) hypotheses of Theorem \ref{thm: OT stability theorem}, but we comment that the resulting optimal potential is an envelope of globally $C^1$ functions, which is the only way in which these other conditions are required in the proof of this theorem. In particular, this example strongly suggests that to obtain stability there must be some restriction on the multiplicity in relation to the ambient dimension.

%\marginpar{Can this example have a short statement,  followed by a longer proof?}

\begin{prop}\label{ex: unstable example}
 Let $c(x, \xbar)=-\euclidean{x}{\xbar}$ on $\R^2\times\R^2$. Denoting points $(x, y)\in \R^2$,  let 
 \begin{align*}
 D:=\left\{(x, y)\in \R^2\mid x^2-r_0^2\leq y\leq r_0^2-x^2\right\}
 \end{align*}
where $r_0>0$ is a small constant to be determined, and take $\sourcemeas$ to be the uniform probability measure on $D$ (see Figure \ref{figure: fig1}). Also define the function
\begin{align*}
 u=\max_{1\leq i\leq 4}u_i
\end{align*}
where
\begin{align*}
 u_1(x, y)&=x^2+y^2-x^6+y,\\
 u_2(x, y)&=4x^2+y^2-y^6+x-3xy,\\
 u_3(x, y)&=4x^2+y^2-y^6-x+3xy,\\
 u_4(x, y)&=4y^4+y^2-\norm{x}^3+y^2\max\left\{0,-\sgn(y)\right\}+3\norm{x}^{\frac{3}{2}},
\end{align*}
and take $\targetmeas$ to be the pushforward of $\sourcemeas$ under $D u$. Then $\targetmeas$ is absolutely continuous with density bounded away from infinity on its support, $\target$ is the disjoint union of nonempty, compact, strictly convex sets $\{\targetpiece[1],\ldots, \targetpiece[4]\}$, each $u_i\in C^1(\R^n)$, and $u$ has a singularity of multiplicity $4$ at $(0, 0)$ relative to this collection. Moreover, for any $\delta>0$ there exists a sequence of measures $\nu^j$ converging to $\targetmeas$ in $\mathcal{W}_\infty$ for which the associated optimal potentials mapping $\sourcemeas$ to $\nu^j$ do not have any singularities of $\delta$-multiplicity $4$ relative to $\{\targetpiece[1],\ldots, \targetpiece[4]\}$.
\end{prop}

\begin{proof}
First, we mention the choice of $r_1$ is taken so that the line $y=-r_1$ passes through the intersection of the curves $y=x^2-r_0^2$ and $y=-\norm{x}$. Thus it is easy to see that $D$ is convex.

Second, we note that $u_1,\ldots u_4$ are convex on $D$ if $r_0$ is sufficiently small. Indeed, since we are in two dimensions, the characteristic polynomial of the Hessian matrix of a $C^2$ function $f$ is $\lambda^2-\Delta f \lambda + \det D^2f$. Thus by the quadratic formula, if $\Delta f\geq 0$ and $\det D^2f\geq 0$ both eigenvalues will be nonnegative, hence $f$ will be convex. This immediately gives the convexity of $u_1$, $u_2$, $u_3$ by a quick calculation near the origin. For $u_4$, we can first see $3\norm{x}^{\frac{3}{2}}-\norm{x}^3$ is a convex function of one variable in $\R$ near zero (by calculating the subdifferential of the function), hence also as a function on $\R^2$. Then the remaining terms are also clearly convex, thus so is their sum $u_4$.

Next, if we let $U_i:=\left\{u=u_i\right\}\cap D$, some tedious but routine calculations yield that
\begin{align*}
 U_1&=\left\{(x, y)\in D\mid y\geq \norm{x}\right\},\\
 U_2&=\left\{(x, y)\in D\mid x\geq 0,\ -\sqrt{\norm{x}}\leq y\leq \norm{x}\right\},\\
 U_3&=\left\{(x, y)\in D\mid x\leq 0,\ -\sqrt{\norm{x}}\leq y\leq \norm{x}\right\},\\
 U_4&=\left\{(x, y)\in D\mid y\leq -\sqrt{\norm{x}}\right\},
\end{align*}
for $r$ small enough. We will show that $D u_i(U_i)$ is a strictly convex set for each $i$.

Before embarking on this verification, let us record
\begin{align*}
 \nabla u_1(x, y)&=(2x-6x^5, 2y+1),\\
  \nabla u_2(x, y)&=(8x+1-3y, 2y-6y^5-3x),\\
   \nabla u_3(x, y)&=(8x-1-3y, 2y-6y^5+3x),\\
    \nabla u_4(x, y)&=(\sgn(x)(\frac{9}{2}\norm{x}^{\frac{1}{2}}-3x^2), 16y^3+2(1+max\left\{0,-\sgn(y)\right\})y).
\end{align*}
The idea will be to take a portion of $U_i^\bdry$ and write it parametrically as $\gamma(t)$. Then we can write 
\begin{align*}
(f(t), g(t)):=\nabla u_i(\gamma(t)),
\end{align*}
and consider one of either 
\begin{align*}
 y(x):&=g(f^{-1}(x)),\quad
 x(y):=f(g^{-1}(y)),
\end{align*}
(i.e., we write one of the coordinates as a function of the other, and consider the image of the boundary curve as the graph of this function). By determining the strict convexity or concavity of these functions (depending on which variable we have solved for, and which side the image of $U_i$ lies), we can then conclude strict convexity of $\nabla u_i(U_i)$ (see Figure \ref{figure: fig1} below for a rough sketch of these regions, diagram is not to scale). We will either directly solve for a variable and verify convexity / concavity, or use the formulae
\begin{align*}
(f^{-1})'(x)&=\frac{1}{f'(f^{-1}(x))},\quad
(f^{-1})''(x)=-\frac{f''(f^{-1}(x))}{(f'(f^{-1}(x)))^3}
\end{align*}
to see
\begin{align}
 y''(x)&=g''(f^{-1}(x))((f^{-1})'(x))^2+g'(f^{-1}(x))(f^{-1})''(x)\notag\\
 &=((f^{-1})'(x))^2(g''(f^{-1}(x))-g'(f^{-1}(x))\frac{f''(f^{-1}(x))}{f'(f^{-1}(x))}).\label{eqn: boundary curve convexity}
\end{align}
(or their analogues if considering $x(y)$). 

\begin{figure}[H] 
  \centering
    \includegraphics[width=\textwidth]{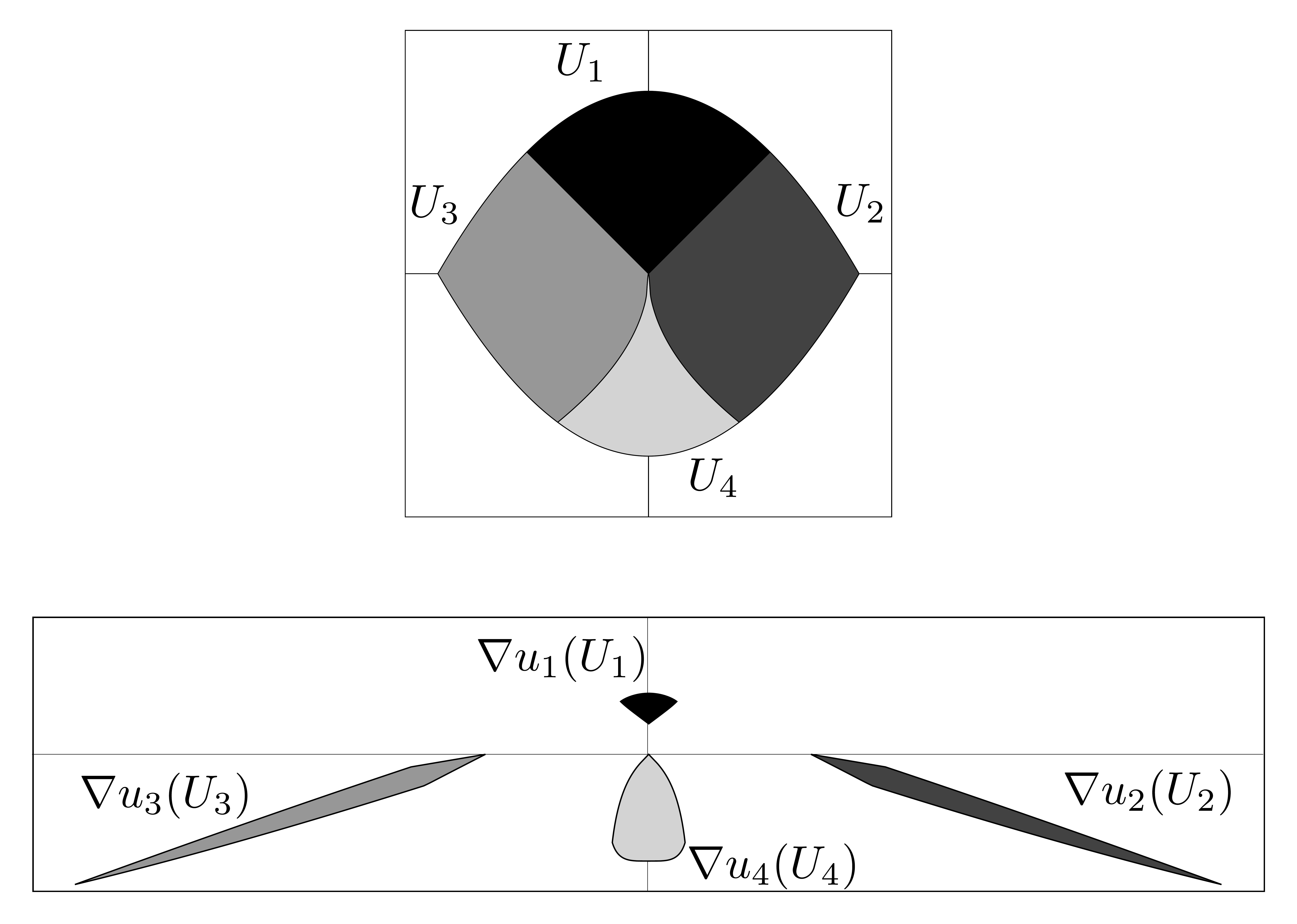}
     \caption{}\label{figure: fig1}
\end{figure}
The remainder is a series of calculations, below $r_1$ and $r_2$ are the $x$-coordinates of the intersection of the line $y=x$ with the curve $y=r_0^2-x^2$, and the intersection of the curves $y=-\norm{x}^{\frac{1}{2}}$ and $y=x^2-r_0^2$ respectively, note that $r_1$, $r_2<r_0$, hence we will always be in the situation $\norm{t}\leq r_0$ in the calculations below.

\begin{framed}
\begin{align*}
 &\bm{\nabla u_1(U_1^\bdry\cap \{y=r_0^2-x^2\})},\ \bm{\gamma(t):=(t, r_0^2-t^2)},\\
 &\bm{-r_1\leq t\leq r_1}, \textbf{ region below curve}.
\end{align*} 
\end{framed}
Then 
\begin{align*}
(f(t), g(t))&=(2t-6t^5, 2(r^2-t^2)+1),\\
(f'(t), g'(t))&=(2-30t^4, -4t),\\
(f''(t), g''(t))&=(-120t^3, -4)
\end{align*}
so for $t=f^{-1}(x)$, 
\begin{align*}
%(f^{-1})'(x)&=\frac{1}{f'(f^{-1}(x))}>0,\\
%g''(f^{-1}(x))-g'(f^{-1}(x))\frac{f''(f^{-1}(x))}{f'(f^{-1}(x))}
\frac{y''(x)}{((f^{-1})'(x))^2}=-4-\frac{480t^4}{2-30t^4}<0,
 \end{align*}
if $r_0$ is small so $y$ is strictly concave.
 
%\fbox{$\bm{U_1^\bdry\cap D^\bdry}$, $\bm{\gamma(t):=(t, \sqrt{r^2-t^2})}$, $\bm{t\in (-\frac{r}{\sqrt{2}}, \frac{r}{\sqrt{2}})}$.} 
%
%Then 
%\begin{align*}
%(f(t), g(t))=(2t-6t^5, 2\sqrt{r^2-t^2}+1)
%\end{align*}
%and (writing $t=f^{-1}(x)$), 
%\begin{align*}
%%(f^{-1})'(x)&=\frac{1}{f'(f^{-1}(x))}>0,\\
%&g''(f^{-1}(x))-g'(f^{-1}(x))\frac{f''(f^{-1}(x))}{f'(f^{-1}(x))}\\
%&=-\frac{2}{\sqrt{r^2-t^2}}-\frac{2t^2}{(r^2-t^2)^{\frac{3}{2}}}-\frac{240t}{\sqrt{r^2-t^2}}\paren{\frac{t}{2-30t^4}}<0,
% \end{align*}
% so by \eqref{eqn: boundary curve convexity}, $y''(x)<0$, and $y$ is strictly concave.
 \begin{framed}
\begin{align*}
 &\bm{\nabla u_1(U_1^\bdry\cap \{y=x\})},\ \bm{\gamma(t):=(t, t)},\\
 &\bm{0\leq t\leq r_1}, \textbf{ region above curve}.
\end{align*} 
\end{framed}
Then
\begin{align*}
 (f(t), g(t))=(2t-6t^5, 2t+1),
\end{align*}
so directly solving:
\begin{align*}
 x(y)=f(g^{-1}(y))=f(\frac{y-1}{2})=y-1-\frac{6}{2^5}(y-1)^5
\end{align*}
 which is strictly concave. We make a similar calculation for $U_1^\bdry\cap \left\{y=-x\right\}$, then since $\nabla u_1$ maps vertical line segments to vertical line segments with the same orientation and the first coordinate is strictly increasing as long as $r_0$ is small, this shows that $\nabla u_1(U_1)$ is a strictly convex set.
 
 \begin{framed}
\begin{align*}
 &\bm{\nabla u_2(U_2^\bdry\cap \{y=x\})},\ \bm{\gamma(t):=(t, t)},\\
 &\bm{0\leq t\leq r_1}, \textbf{ region below curve}.
\end{align*} 
\end{framed}
Then
\begin{align*}
 (f(t), g(t))=(8t+1-3t, 2t-6t^5-3t)=(5t+1, -t-6t^5),
\end{align*}
directly solving,
\begin{align*}
 y(x):=g(f^{-1}(x))=g(\frac{x-1}{5})=\frac{1-x}{5}-\frac{6(x-1)^5}{5^5}
\end{align*}
which is strictly concave.

 \begin{framed}
\begin{align*}
 &\bm{\nabla u_2(U_2^\bdry\cap \{y=-\sqrt{\norm{x}}\})},\ \bm{\gamma(t):=(t, -t^{\frac{1}{2}})},\\
 &\bm{0\leq t\leq r_2}, \textbf{ region above curve}.
\end{align*} 
\end{framed}
Then
\begin{align*}
 (f(t), g(t))&=(8t+1+3t^{\frac{1}{2}}, -2t^{\frac{1}{2}}+6t^{\frac{5}{2}}-3t),\\
   (f'(t), g'(t))&=(8+\frac{3}{2}t^{-\frac{1}{2}}, -t^{-\frac{1}{2}}+15t^{\frac{3}{2}}-3),\\
   (f''(t), g''(t))&=(-\frac{3}{4}t^{-\frac{3}{2}}, \frac{t^{-\frac{3}{2}}}{2}+\frac{45}{2}t^{\frac{1}{2}}),
\end{align*}
for $t=f^{-1}(x)$, 
\begin{align*}
%&g''(f^{-1}(x))-g'(f^{-1}(x))\frac{f''(f^{-1}(x))}{f'(f^{-1}(x))}\\
\frac{y''(x)}{((f^{-1})'(x))^2}&=\frac{t^{-\frac{3}{2}}}{2}+\frac{45}{2}t^{\frac{1}{2}}+\frac{\frac{3}{4}t^{-\frac{3}{2}}(-t^{-\frac{1}{2}}+15t^{\frac{3}{2}}-3)}{8+\frac{3}{2}t^{-\frac{1}{2}}}\\
&\geq t^{-\frac{3}{2}}\left(\frac{1}{2}-\frac{3(t^{-\frac{1}{2}}+3)}{4(8+\frac{3}{2}t^{-\frac{1}{2}})}\right)\\
&\geq t^{-\frac{3}{2}}\left(\frac{1}{2}-\frac{3(1+3\sqrt{t})}{4(8\sqrt{t}+\frac{3}{2})}\right)\geq t^{-\frac{3}{2}}\left(\frac{1}{2}-\frac{3(1+3\sqrt{r_2})}{6}\right)>0
 \end{align*}
 if $r_0$ is small enough, making $y$ strictly convex.

%
%\fbox{$\bm{\nabla u_2(U_2^\bdry\cap  \{y=-\sqrt{x}\}\cap \{x\geq 0\}})$, $\bm{\gamma(t):=(t, -t^{\frac{1}{2}})}$, $\bm{t> 0}$,  \textbf{region above curve}.} 
%
%Then
%\begin{align*}
% (f(t), g(t))&=(8t+1+3t^{\frac{1}{2}}, -2t^{\frac{1}{2}}+6t^{\frac{5}{2}}-3t),\\
%  (f'(t), g'(t))&=(8+\frac{3}{2}t^{-\frac{1}{2}}, -t^{-\frac{1}{2}}+15t^{\frac{3}{2}}-3),\\
%   (f''(t), g''(t))&=(-\frac{3}{4}t^{-\frac{3}{2}}, \frac{1}{2}t^{-\frac{3}{2}}+\frac{45}{2}t^{\frac{1}{2}}),
%\end{align*}
%for $t=f^{-1}(x)$, 
%\begin{align*}
%%&g''(f^{-1}(x))-g'(f^{-1}(x))\frac{f''(f^{-1}(x))}{f'(f^{-1}(x))}\\
%\frac{y''(x)}{((f^{-1})'(x))^2}&=\frac{t^{-\frac{3}{2}}}{2}+\frac{45}{2}t^{\frac{1}{2}}+\frac{\frac{3}{4}t^{-\frac{3}{2}}(-t^{-\frac{1}{2}}+15t^{\frac{3}{2}}-3)}{8+\frac{3}{2}t^{-\frac{1}{2}}}\\
%&\geq t^{-\frac{3}{2}}\left(\frac{1}{2}-\frac{3(t^{-\frac{1}{2}}+3)}{4(8+\frac{3}{2}t^{-\frac{1}{2}})}\right)\\
%&\geq t^{-\frac{3}{2}}\left(\frac{1}{2}-\frac{3(1+3\sqrt{t})}{4(8\sqrt{t}+\frac{3}{2})}\right)\geq t^{-\frac{3}{2}}\left(\frac{1}{2}-\frac{3(1+3\sqrt{r})}{6}\right)>0
% \end{align*}
% if $r$ is small enough, making $y$ strictly convex.
  \begin{framed}
\begin{align*}
 &\bm{\nabla u_2(U_2^\bdry\cap \{y=r_0^2-x^2\})},\ \bm{\gamma(t):=(t, r_0^2-t^2)},\\
 &\bm{r_1\leq t\leq r_0}, \textbf{ region below curve}.
\end{align*} 
\end{framed}

\begin{align*}
 (f(t), g(t))&=(8t+1-3(r_0^2-t^2), 2(r_0^2-t^2)-6(r_0^2-t^2)^5-3t),\\
  (f'(t), g'(t))&=(8+6t, -4t+60t(r_0^2-t^2)^4-3),\\
   (f''(t), g''(t))&=(6, -4+60(r_0^2-t^2)^4-480t^2(r_0^2-t^2)^3),
\end{align*}
%considering \eqref{eqn: boundary curve convexity} for $x(y)$ instead, we calculate (where $t=g^{-1}(x)$)
%\begin{align*}
%& f''(t)-f'(t)\frac{g''(t)}{g'(t)}\\
%&=-\frac{8r^2}{(r^2-t^2)^{\frac{3}{2}}}+\paren{3+\frac{8t}{\sqrt{r^2-t^2}}}\paren{\frac{3r^2-120t^3(r^2-t^2)^{\frac{3}{2}}}{(r^2-t^2)^{\frac{3}{2}}}}/\paren{\frac{2\sqrt{r^2-t^2}-30t^4\sqrt{r^2-t^2}+3t}{\sqrt{r^2-t^2}}}\\
%&\leq -\frac{8}{r}+\frac{(3\sqrt{r^2-t^2}+8t)(3t^2-120t^3(r^2-t^2)^{\frac{3}{2}})}{(r^2-t^2)^{\frac{3}{2}}(2\sqrt{r^2-t^2}-30t^4\sqrt{r^2-t^2}+3t)}\\
%&\leq -\frac{8}{r}+\frac{11r(3r^2+120r^9)}{(r^2-t^2)^{\frac{3}{2}}(2\sqrt{r^2-t^2}-30t^4\sqrt{r^2-t^2}+3t)}
%\end{align*}
for $t=f^{-1}(x)$, 
\begin{align*}
%&g''(f^{-1}(x))-g'(f^{-1}(x))\frac{f''(f^{-1}(x))}{f'(f^{-1}(x))}\\
\frac{y''(x)}{((f^{-1})'(x))^2}&=-4+60(r_0^2-t^2)^4-480t^2(r_0^2-t^2)^3+\frac{6(4t-60t(r_0^2-t^2)^4+3)}{8+6t}\\
&\leq -4+60r_0^8+\frac{3(4t+3)}{4}\leq -4+60r_0^8+\frac{9}{4}+3r_0<0
 \end{align*}
 when $r_0$ is small, so $y$ is strictly concave.

  \begin{framed}
\begin{align*}
 &\bm{\nabla u_2(U_2^\bdry\cap \{y=x^2-r_0^2\})},\ \bm{\gamma(t):=(t, t^2-r_0^2)},\\
 &\bm{r_2\leq t\leq r_0}, \textbf{ region above curve}.
\end{align*} 
\end{framed}

\begin{align*}
 (f(t), g(t))&=(8t+1-3(t^2-r_0^2), 2(t^2-r_0^2)-6(t^2-r_0^2)^5-3t),\\
  (f'(t), g'(t))&=(8-6t, 4t-60t(t^2-r_0^2)^4-3),\\
   (f''(t), g''(t))&=(-6, 4-60(t^2-r_0^2)^4-480t^2(t^2-r_0^2)^3),
\end{align*}
for $t=f^{-1}(x)$, 
\begin{align*}
%&g''(f^{-1}(x))-g'(f^{-1}(x))\frac{f''(f^{-1}(x))}{f'(f^{-1}(x))}\\
\frac{y''(x)}{((f^{-1})'(x))^2}&=4-60(t^2-r_0^2)^4-480t^2(t^2-r_0^2)^3+\frac{6(4t-60t(t^2-r_0^2)^4-3)}{8-6t}\\
&\geq 4-60r_0^8-\frac{6(60t(r_0^2-t^2)^4+3)}{8-6t}\geq 4-60r_0^8-\frac{6(60r_0^9+3)}{8-6r_0}>0
 \end{align*}
 when $r_0$ is small, so $y$ is strictly convex.
  
Since $\nabla u_2$ maps all horizontal lines to lines with the same slope, the above verifications give strict convexity of $\nabla u_2(U_2)$, a symmetric argument shows the strict convexity of $\nabla u_3(U_3)$.

  \begin{framed}
\begin{align*}
 &\bm{\nabla u_4(U_4^\bdry\cap \{y=-\sqrt{\norm{x}}\})},\ \bm{\gamma(t):=(t, -t^{\frac{1}{2}})},\\
 &\bm{0\leq t\leq r_2}, \textbf{ region below curve}.
\end{align*} 
\end{framed}

\begin{align*}
 (f(t), g(t))&=(\frac{9}{2}t^{\frac{1}{2}}-3t^2, -16t^{\frac{3}{2}}-3t^{\frac{1}{2}}),\\
 (f'(t), g'(t))&=(\frac{9}{4}t^{-\frac{1}{2}}-6t, -24t^{\frac{1}{2}}-\frac{3}{2}t^{-\frac{1}{2}}),\\
  (f''(t), g''(t))&=(-\frac{9}{8}t^{-\frac{3}{2}}-6, -12t^{-\frac{1}{2}}+\frac{3}{4}t^{-\frac{3}{2}}),
\end{align*}
for $t=f^{-1}(x)$, 
\begin{align*}
%&g''(f^{-1}(x))-g'(f^{-1}(x))\frac{f''(f^{-1}(x))}{f'(f^{-1}(x))}\\
\frac{y''(x)}{((f^{-1})'(x))^2}&=-12t^{-\frac{1}{2}}+\frac{3}{4}t^{-\frac{3}{2}}-\frac{(\frac{9}{8}t^{-\frac{3}{2}}+6)(24t^{\frac{1}{2}}+\frac{3}{2}t^{-\frac{1}{2}})}{\frac{9}{4}t^{-\frac{1}{2}}-6t}\\
&< \frac{3}{4}t^{-\frac{3}{2}}-\frac{(\frac{9}{8}t^{-\frac{3}{2}})(\frac{3}{2}t^{-\frac{1}{2}})}{\frac{9}{4}t^{-\frac{1}{2}}}\leq t^{-\frac{3}{2}}\left(\frac{3}{4}-\frac{3}{4}\right)=0
 \end{align*}
 when $r_0$ is small, so $y$ is strictly concave. A symmetric calculation holds for the boundary curve where $x\leq 0$.
 
\begin{framed}
\begin{align*}
 &\bm{\nabla u_4(U_4^\bdry\cap \{y=x^2-r_0^2\})},\ \bm{\gamma(t):=(t, x^2-r_0^2)},\\
 &\bm{-r_2\leq t\leq r_2}, \textbf{ region above curve}.
\end{align*} 
\end{framed}

\begin{align*}
 (f(t), g(t))&=(\frac{9}{2}t^{\frac{1}{2}}-3t^2, 16(t^2-r_0^2)^3+3(t^2-r_0^2)),\\
 (f'(t), g'(t))&=(\frac{9}{4}t^{-\frac{1}{2}}-6t, 96t(t^2-r_0^2)^2+6t),\\
  (f''(t), g''(t))&=(-\frac{9}{8}t^{-\frac{3}{2}}-6, 96(t^2-r_0^2)^2+384t^2(t^2-r_0^2)+6),
\end{align*}
this time for $t=g^{-1}(y)$ and $t>0$, 
\begin{align*}
x''(y)&=f''(g^{-1}(y))-f'(g^{-1}(y))\frac{g''(g^{-1}(y))}{g'(g^{-1}(y))}\\
&=-\frac{9}{8}t^{-\frac{3}{2}}-6-\frac{(\frac{9}{4}t^{-\frac{1}{2}}-6t)(96(t^2-r_0^2)^2+384t^2(t^2-r_0^2)+6)}{96t(t^2-r_0^2)^2+6t}\\
&<0
 \end{align*}
 when $r_0$ is small, so $x$ is a strictly concave function of $y$ when $x>0$. A symmetric argument holds when $x<0$. When $x=0$, we find the tangent line to $\nabla u_4(U_4)$ at the boundary point $(0, -16r_0^6-3r_0^2)$ is the horizontal line through that point, which is easily seen to lie below $\nabla u_4(U_4)$, touching only at $(0, -16r_0^6-3r_0^2)$, with a similar argument for the tangent line at $(0, 0)$. Since $\nabla u_4$ sends vertical lines to vertical lines with the same orientation, this shows $\nabla u_4(U_4)$ is strictly convex, completing the verification.

Finally, we easily see that $u_i$ is strictly convex for each $i$, and the above calculation of regions shows $\nabla u$ is injective on the union of the interiors of the $U_i$. A quick calculation shows $\det D^2u_i$ is actually bounded away from zero on $U_i$ for each $i$, this gives that $\nu$ is absolutely continuous with density bounded away from infinity (in fact, this density is actually bounded away from zero on the images of $U_1$, $U_2$, and $U_3$).
 
 Now we can see that $\nabla u_2(U_2)$, $\nabla u_3(U_3)$, and $\nabla u_4(U_4)$ all lie in the half space $\{(x, y)\in\R^2\mid y\leq 0\}$ and all have nonempty intersections with the $x$-axis. Fix any $\delta>0$, we now take the sequence of measures $\{\nu^j\}_{j=1}^\infty$ to be $\targetmeas$, but with the set $\nabla u_4(U_4)$ shifted upward by $\delta/j$, it is clear that $\Winfty{\targetmeas}{\nu^j}\leq \delta/j$. Let $u^j$ be an optimal potential transporting $\mu$ to $\nu^j$, and suppose there is a point $x$ that is a singularity of $\delta$-multiplicity $4$ for $u^j$ relative to $\{\nabla u_1(U_1),\ldots \nabla u_4(U_4)\}$. Since the extremal points of $\subdiff{u}{x}$ must be contained in $\spt\nu^j$, this could only happen if $\subdiff{u}{x}$ intersects both $\nabla u_2(U_2)$ and $\nabla u_3(U_3)$. However this would also force $\subdiff{u}{x}$ to have nonempty intersection with the interior of $\nabla u_4(U_4)+\frac{\delta}{j} e_2$ and we can derive a contradiction by the same argument as in the proof of Proposition \ref{prop: u is a maximum of C^1 potentials}, thus no point can have a $\delta$-multiplicity of $4$.
\end{proof}

%\marginpar{Add reference to Zajicek \cite{Zajicek79}}

\bibliography{mybibliofreediscontinuities}
\bibliographystyle{plain}

\end{document}